\documentclass[10pt]{amsart}
\usepackage{latexsym}
\usepackage{amsmath}
\usepackage{amssymb}
\usepackage{mathrsfs}
\usepackage{graphicx}
\usepackage{color}
\usepackage{pgfpages}
\usepackage{ifthen}
\usepackage{leftidx,tensor}
\usepackage[T1]{fontenc}
\usepackage[latin1]{inputenc}
\usepackage{mathtools}
\usepackage{comment}
\usepackage{dsfont}
\usepackage[nocompress]{cite}

\usepackage[shortlabels]{enumitem}
\usepackage{aliascnt}
\usepackage[bookmarks=true,pdfstartview=FitH, pdfborder={0 0 0}, colorlinks=true,citecolor=red, linkcolor=blue]{hyperref}
\usepackage{bbm}
\usepackage{nicefrac}

\theoremstyle{plain}
\newtheorem{thm}{Theorem}[section]
\newaliascnt{cor}{thm}
\newaliascnt{prop}{thm}
\newaliascnt{lem}{thm}
\newtheorem{prop}[prop]{Proposition}
\newtheorem{lem}[lem]{Lemma}
\aliascntresetthe{cor}
\aliascntresetthe{prop}
\aliascntresetthe{lem} 
\theoremstyle{definition}
\newaliascnt{defn}{thm}
\newaliascnt{asu}{thm}
\newaliascnt{con}{thm}
\aliascntresetthe{defn}
\aliascntresetthe{asu}
\aliascntresetthe{con}
\newcounter{stp}
\newcounter{stpi}
\newcounter{stpci}
\newcounter{stpiii}

\theoremstyle{remark}
\newaliascnt{rem}{thm}
\newaliascnt{exa}{thm}
\newaliascnt{masu}{thm}
\newaliascnt{nota}{thm}
\newaliascnt{sett}{thm}
\newtheorem{rem}[rem]{Remark}

\aliascntresetthe{rem}
\aliascntresetthe{exa}
\aliascntresetthe{masu}
\aliascntresetthe{nota}
\aliascntresetthe{sett}
\numberwithin{equation}{section}

\setlist[enumerate]{font = \normalfont}

\newcommand{\F}{\mathbb{F}}
\newcommand{\E}{\mathbb{E}}
\newcommand{\R}{\mathbb{R}}
\newcommand{\C}{\mathbb{C}}
\newcommand{\B}{\mathbb{B}}

\newcommand{\bP}{\mathbb{P}}

\newcommand{\D}{\mathbb{D}}

\newcommand{\bS}{\mathbb{S}}
\newcommand{\bM}{\mathbb{M}}
\newcommand{\bK}{\mathbb{K}}

\newcommand{\W}{\mathbb{W}}
\newcommand{\bO}{\mathbb{O}}
\newcommand{\tE}{\Tilde{\mathbb{E}}}

\newcommand{\rW}{\mathrm{W}} 
\newcommand{\rL}{\mathrm{L}}
\newcommand{\rE}{\mathrm{E}}
\newcommand{\rH}{\mathrm{H}}
\newcommand{\rT}{\mathrm{T}}
\newcommand{\rC}{\mathrm{C}}
\newcommand{\rB}{\mathrm{B}}

\newcommand{\rD}{\mathrm{D}}

\newcommand{\rX}{\mathrm{X}}
\newcommand{\rY}{\mathrm{Y}}
\newcommand{\rJ}{\mathrm{J}}
\newcommand{\rI}{\mathrm{I}}
\newcommand{\rII}{\mathrm{II}}
\newcommand{\rIII}{\mathrm{III}}
\newcommand{\ri}{\mathrm{i}}
\newcommand{\rii}{\mathrm{ii}}
\newcommand{\riii}{\mathrm{iii}}

\newcommand{\DeltaN}{\Delta_{\mathrm{N}}}
\newcommand{\BUC}{\mathrm{BUC}}
\newcommand{\SO}{\mathrm{SO}}
\newcommand{\rd}{\,\mathrm{d}}

\newcommand{\fs}{\mathrm{fs}}
\newcommand{\fl}{\mathrm{fl}}
\newcommand{\pr}{\mathrm{pr}}
\newcommand{\lc}{\mathrm{lc}}

\newcommand{\mre}{\mathrm{e}}
\newcommand{\mri}{\mathrm{i}}

\newcommand{\rh}{\mathrm{h}}
\newcommand{\rl}{\mathrm{l}}
\newcommand{\rr}{\mathrm{r}}
\newcommand{\dist}{\mathrm{dist}}
\newcommand{\Cmr}{C_\mathrm{MR}}

\newcommand{\mS}{m_\mathcal{S}}
\newcommand{\Afs}{A_\fs}
\newcommand{\AfsI}{A_{\fs,\rI}}
\newcommand{\Bfs}{B_\fs}
\newcommand{\cS}{\mathcal{S}}
\newcommand{\cF}{\mathcal{F}}
\newcommand{\cE}{\mathcal{E}}
\newcommand{\cO}{\mathcal{O}}
\newcommand{\cB}{\mathcal{B}}
\newcommand{\cL}{\mathcal{L}}
\newcommand{\cM}{\mathcal{M}}
\newcommand{\cN}{\mathcal{N}}
\newcommand{\cG}{\mathcal{G}}

\newcommand{\cK}{\mathcal{K}}
\newcommand{\cJ}{\mathcal{J}}
\newcommand{\cT}{\mathcal{T}}

\newcommand{\cD}{\mathcal{D}}
\newcommand{\cH}{\mathcal{H}}
\newcommand{\tcK}{\Tilde{\cK}}
\newcommand{\tv}{\Tilde{v}}

\newcommand{\tQ}{\Tilde{Q}}
\newcommand{\tp}{\Tilde{\pi}}
\newcommand{\tell}{\Tilde{\ell}}
\newcommand{\tomega}{\Tilde{\omega}}
\newcommand{\tz}{\Tilde{z}}
\newcommand{\tS}{\Tilde{S}}
\newcommand{\hv}{\widehat{v}}

\newcommand{\hQ}{\widehat{Q}}
\newcommand{\hp}{\widehat{\pi}}
\newcommand{\hell}{\widehat{\ell}}
\newcommand{\homega}{\widehat{\omega}}
\newcommand{\hz}{\widehat{z}}

\newcommand{\ovv}{\overline{v}}

\newcommand{\ovQ}{\overline{Q}}
\newcommand{\ovell}{\overline{\ell}}
\newcommand{\ovomega}{\overline{\omega}}
\newcommand{\eps}{\varepsilon}
\newcommand{\del}{\partial}
\DeclareMathOperator{\tr}{tr}

\DeclareMathOperator{\Id}{Id}
\DeclareMathOperator{\mdiv}{div}
\DeclareMathOperator{\Rep}{Re}
\DeclareMathOperator{\Imp}{Im}
\newcommand{\tin}{\enspace \text{in} \enspace}
\newcommand{\ton}{\enspace \text{on} \enspace}
\newcommand{\tfor}{\enspace \text{for} \enspace}
\newcommand{\tforall}{\enspace \text{for all} \enspace}
\newcommand{\tand}{\enspace \text{and} \enspace}
\newcommand{\tso}{\enspace \text{so} \enspace}

\newcommand{\twhere}{\enspace \text{where} \enspace}
\newcommand{\twith}{\enspace \text{with} \enspace}
\newcommand{\taswellas}{\enspace \text{as well as} \enspace}

\begin{document}

\title[Dynamics of the general $Q$-tensor model interacting with a rigid body]{Dynamics of the general $Q$-tensor model interacting with a rigid body}

\author{Felix Brandt}
\address{Department of Mathematics, University of California, Berkeley, Berkeley, 94720, CA, USA.}
\email{fbrandt@berkeley.edu}
\author{Matthias Hieber}
\address{Technische Universit\"{a}t Darmstadt,
Schlo\ss{}gartenstra{\ss}e 7, 64289 Darmstadt, Germany.}
\email{hieber@mathematik.tu-darmstadt.de}
\author{Arnab Roy}
\address{Basque Center for Applied Mathematics (BCAM), Alameda de Mazarredo 14, 48009 Bilbao, Spain.}
	\address{IKERBASQUE, Basque Foundation for Science, Plaza Euskadi 5, 48009 Bilbao, Bizkaia, Spain.}
\email{aroy@bcamath.org}

\begin{abstract}
In this article, the fluid-rigid body interaction problem of nematic liquid crystals described by the general Beris-Edwards $Q$-tensor model is studied.
It is proved first that the total energy of this problem decreases in time.
The associated mathematical problem is a quasilinear mixed-order system with moving boundary.
After the transformation to a fixed domain, a monolithic approach based on the added mass operator and lifting arguments is employed to establish the maximal $\rL^p$-regularity of the linearized problem in an anisotropic ground space.
This paves the way for the local strong well-posedness for large data and global strong well-posedness for small data of the interaction problem.
\end{abstract}

\maketitle

\section{Introduction}

Nematic liquid crystals are described by the Doi-Onsager, the Ericksen-Leslie and the Landau-de Gennes theory.
While Ericksen-Leslie models are vector models, $Q$-tensor models including the Landau-de Gennes theory rely on symmetric, traceless $3 \times 3$-matrices $Q$ to model the biaxial alignment of molecules.
The evolution of~$Q$ is influenced by the free energy of the molecules and aspects such as tumbling and alignment effects.
For further background, we refer to the monographs of Virga \cite{Vir:94} and Sonnet and Virga \cite{SV:12} or the survey articles of Ball \cite{Ball:17}, Hieber and Pr\"uss \cite{HP:18}, Wang, Zhang and Zhang \cite{WZZ:21} and Zarnescu \cite{Zar:21}.

In this work, we investigate the fluid-rigid body interaction problem of the general Beris-Edwards $Q$-tensor model \cite{BE:94}.
More precisely, we show the local strong well-posedness for large data and the global strong well-posedness for small data of the interaction problem.
In particular, in the underlying $Q$-tensor model, we consider an arbitrary ratio of tumbling and alignment effects $\xi \in \R$.
For the fluid velocity~$u$ and pressure~$\pi$ and the molecular orientation of the liquid crystals~$Q$, taking values in the space $\bS_0^3$ of symmetric and traceless $3 \times 3$-matrices, the general $Q$-tensor model is of the form
\begin{equation}\label{eq:Q-tensor simplified}
\left\{
    \begin{aligned}
        \partial_t u + \left(u \cdot \nabla\right)u + \nabla \pi 
        &= \Delta u + \mdiv\bigl(\tau(Q,H(Q)) + \sigma(Q,H(Q))\bigr), \enspace \mdiv u = 0, \\
        \partial_t Q + \left(u\cdot \nabla\right)Q - S(\nabla u,Q) 
        &= H(Q).
    \end{aligned}
\right.
\end{equation}
Here $H(Q) = \Delta Q - Q + (Q^2 - \tr(Q^2) \nicefrac{\Id_3}{3}) - \tr(Q^2) Q$ is the variational derivative of the free energy functional.
The non-Newtonian stresses due to the molecular orientation of the biaxial liquid crystals
\begin{equation*}
    \begin{aligned}
        \tau(Q,H(Q))
        &= - \nabla Q \odot \nabla Q - \xi(Q + \nicefrac{\Id_3}{3}) H(Q) - \xi H(Q)(Q + \nicefrac{\Id_3}{3}) + 2 \xi(Q + \nicefrac{\Id_3}{3}) \tr\bigl(Q H(Q)\bigr) \tand\\
        \sigma(Q,H(Q)) 
        &= Q H(Q) - H(Q) Q = Q \Delta Q - \Delta Q Q
    \end{aligned}
\end{equation*}
correspond to the symmetric and antisymmetric part of the stress tensor.
For the deformation tensor~$\D(u)$ and its antisymmetric counterpart $\W(u)$, the term $S(\nabla u,Q)$ accounting for stretching and rotation is given by
$S(\nabla u,Q) = (\xi \D(u) + \W(u))(Q + \nicefrac{\Id_3}{3}) + (Q + \nicefrac{\Id_3}{3})(\xi \D(u) - \W(u)) - 2\xi (Q + \nicefrac{\Id_3}{3}) \tr(Q \nabla u)$. 

Beris-Edwards models for liquid crystals have been investigated by many authors considering simplifications or modifications of the model.
In the context of weak solutions, we refer for instance to the articles \cite{BZ:11, PZ:11, PZ:12, GGRB:14, FRSZ:14, ADL:14, Wil:15}.
Strong solutions were e.\ g.\ studied in \cite{ADL:16, CRWX:16, SS:19, MS:22}.
These results either concern the 2D case, the situation of $\xi = 0$ or simplified variants of the model.
Only very recently, a rather complete understanding of the strong well-posedness and dynamics of the \emph{general} Beris-Edwards model with \emph{arbitrary} ratio of tumbling and alignment effects~$\xi \in \R$ in 3D was obtained in \cite{HHW:24}.

In the present article, we build on this new development and extend the analysis to the associated interaction problem with a rigid body.
The physical motivation for liquid crystal interaction problems is the formation of \emph{liquid crystal colloids} by dispersion of colloidal particles in the liquid crystal host medium, where we view colloidal particles as rigid bodies.
We refer for instance to \cite{MU, MO} for more details on the underlying physics.
The interaction with a rigid body poses significant challenges.
Most importantly, the resulting problem is a \emph{moving domain problem}, and the equations describing the dynamics of the liquid crystals are coupled to the equations governing the motion of the rigid body via the complicated stress tensor and the continuity of the fluid and rigid body velocity on their moving interface.

More precisely, the dynamics of the rigid body follows Newton's laws.
Consequently, for the total stress tensor $\Sigma(u,\pi,Q) = 2\D(u) + \tau(Q,H(Q)) + \sigma(Q,H(Q)) - \pi \Id_3$, the mass $\mS$ and inertia matrix $J$ of the rigid body, the position of its center of mass $h$, its angular velocity $\Omega$, its domain $\cS(t)$ and the outward unit normal $\nu(t)$, the rigid body motion is determined by
\begin{equation}\label{eq:simplified rigid body eqs}
    \mS h''(t)
    = -\int\limits_{\partial \cS(t)} \Sigma(u,\pi,Q) \nu(t)\rd \Gamma \tand 
    (J\Omega)' (t) = -\int\limits_{\partial \cS(t)} (x-h(t)) \times \Sigma(u,\pi,Q) \nu(t)\rd \Gamma.
\end{equation}
In particular, the complex effects due to the non-Newtonian stresses $\tau$ and $\sigma$ enter in these equations, and estimating the boundary integrals leads to difficulties.
The domain of the present problem changes in time due to the motion of the rigid body.
The other coupling of the dynamics of the body to the $Q$-tensor model arises from the equality of fluid and body velocity on their interface $\del \cS(t)$ and reads as
\begin{equation}\label{eq:simplified interface cond}
    u(t,x) = h'(t) + \Omega(t) \times (x-h(t)), \tfor t \in (0,T), \enspace x \in \del \cS(t).
\end{equation}

For handling the present moving boundary problem given in simplified form in \eqref{eq:Q-tensor simplified}--\eqref{eq:simplified interface cond}, we first employ a local transformation to a fixed domain.
The resulting problem is a quasilinear non-autonomous one, and we tackle it by establishing maximal $\rL^p$-regularity for a suitable linearization first.

The $Q$-tensor model \eqref{eq:Q-tensor simplified} is a mixed-order system, where the diagonal parts are of second order, while the off-diagonal parts are of order one and three.
In contrast to the rather well-understood case of systems with principal part, mixed-order systems lack a unified theory.
Here we choose an anisotropic ground space for the liquid crystal part, so every entry is of highest order.
Maximal regularity is only known to be valid on $\rL^2 \times \rH^1$, see \cite{HHW:24}, or on $\rL^q \times \rW^{1,q}$ for~$q$ close to~$2$, see \cite[Section~7.6]{AH:23}.
It is an interesting open question whether maximal regularity for the linearized general $Q$-tensor model can be achieved for general $q \in (1,\infty)$ as well.
Previous works in this direction are limited to simplified models.

The block structure with all blocks of highest order also leads to significant challenges.
We employ here a monolithic approach, so we keep the interface condition~\eqref{eq:simplified interface cond} in the study of the linearized problem and solve the linearization of \eqref{eq:Q-tensor simplified} and~\eqref{eq:simplified rigid body eqs} simultaneously.
The \emph{added mass operator} together with an involved lifting procedure allows us to write the linearized problem in terms of a so-called \emph{fluid-structure operator} tailored to the present setting of the $Q$-tensor model.

We establish the maximal $\rL^p$-regularity on the anisotropic space $\rL_{\sigma}^2(\cF_0) \times \rH^1(\cF_0,\bS_0^3) \times \R^3 \times \R^3$, where~$\cF_0$ is the initial fluid domain, and $\rL_\sigma^2(\cF_0)$ denotes the space of solenoidal vector fields in~$\rL^2(\cF_0)^3$.
In particular, the condition $Q \in \bS_0^3$ is included in the function spaces in elegant way.
The maximal $\rL^p(\R_+)$-regularity with exponential decay for the linearization at zero follows from spectral theory.

In order to obtain the first main result of this paper, the local strong well-posedness of the interaction problem \eqref{eq:Q-tensor simplified}--\eqref{eq:simplified interface cond} for large initial data made precise below in \autoref{thm:loc strong wp}, we estimate the nonlinear terms and employ the maximal regularity.
The anisotropic ground space $\rL^2 \times \rH^1$ for the liquid crystals part renders the nonlinear problem more complicated.
The time interval of existence in this case is limited by the requirement that we exclude the case of collision of the rigid body with the outer fluid boundary.

With regard to the second main result of the present article, the global strong well-posedness and exponential decay of the solution for small initial data as asserted in \autoref{thm:global strong wp for small data}, we first prove that the total energy of the interaction problem \eqref{eq:Q-tensor simplified}--\eqref{eq:simplified interface cond} is decreasing, which reveals the shape of the equilibria.
In a second step, we make use of a fixed point argument, relying on the maximal $\rL^p(\R_+)$-regularity of the linearization at zero as well as adjusted decay estimates of the transform.
Let us emphasize that we are especially able to show that no collision of the rigid body with the outer fluid boundary occurs.
This follows from the exponential stability of the semigroup associated to the linearized problem around zero and the smallness of the initial data.
The exponential stability also reveals that all velocities tend to zero as time tends to infinity.
However, it remains an interesting open problem to determine if the position of the center of the moving rigid body converges to some point $h_{\infty}\in \R^3$ as $t\rightarrow\infty$.
In the situation of a rigid body immersed in a viscous, incompressible Newtonian fluid filling the entire $\R^3$, this problem has recently been tackled in the case of a ball in \cite{EMT:23}, then for a rigid body of arbitrary shape in \cite{MT:23} and for the case of several bodies in \cite{bravin2025long}.

Fluid-rigid body interaction problems of incompressible, Newtonian fluids are classical and abundantly documented.
We mention e.\ g.\ the pioneering article of Serre \cite{Ser:87} and the papers \cite{DE:99, CSMT:00, nevcasova2021} in the weak setting or the works \cite{Galdi:02, GS:02, Takahashi:03, GGH:13, EMT:23} in the strong setting, see also \cite{KKLTTW:18} for an overview of the topic.
Most of the above articles deal with the semilinear setting, whereas here, due to the quasilinear structure of the stress tensor, we need to deal with a \emph{quasilinear} fluid-structure interaction problem.

In contrast to the case of Newtonian fluids, the investigation of fluid-rigid body interaction problems of liquid crystals started only recently.
In \cite{geng2023global}, Geng, Roy and Zarnescu showed the global existence of a weak solution to the interaction problem of a simplified $Q$-tensor model, whereas here, we study the general $Q$-tensor model with arbitrary ratio $\xi \in \R$ and especially obtain uniqueness of solutions.
In \cite{BBHR:23}, the local strong well-posedness and the global strong well-posedness for initial data close to equilibria of the fluid-rigid body interaction problem of the simplified Ericksen-Leslie model (see \cite{LL:95}) were proved.
The complexity of the stress tensor, the mixed-order character of the underlying $Q$-tensor model as well as the fact that the molecular orientation is described by a matrix render the analysis of the present interaction problem fundamentally different.

For convenience of the reader, let us give a brief, non-technical overview of the procedure to obtain the well-posedness results of the present interaction problem of a rigid body with the general Beris-Edwards $Q$-tensor model.
First, we determine the set of equilibrium solutions by examining the critical points of the energy functional associated with the interaction problem.
Next, we employ a transform that only acts locally, i.\ e., in a neighborhood of the translating and rotating rigid body.
By continuity of the body velocities, at least for short time, we can guarantee that there is no collision of the rigid body with the outer fluid boundary.
After transforming the moving boundary problem to a fixed domain, we perform the analysis of the linearized problem.
The key idea here is to use a lifting procedure in order to eliminate the pressure from the surface integrals of the stress tensor, and to reformulate the linearized problem in operator form.
This is done by means of the so-called added mass operator.
The idea of this concept can be traced back to Bessel, see \cite{Sto:51, Sto:09}, and has become classical in the investigation of fluid-rigid body interaction problems.
Thanks to a refined spectral analysis, we prove the exponential stability for the linearization at zero, which, together with the smallness of the initial data, helps to show that no collision of the rigid body with the outer fluid boundary occurs for sufficiently small initial data.
The proofs of the well-posedness results are based on a fixed point procedure, relying in turn on the analysis of the linearized problem, and they are complemented by suitable nonlinear estimates, where we efficiently handle the transform.
Let us emphasize that the complicated nature of the stress tensor, in particular in the terms with the molecular orientation $Q$, leads to involved terms that need to be estimated.

This article has the following outline.
\autoref{sec:model} is dedicated to the introduction of the Beris-Edwards model and the interaction problem with a rigid body.
The purpose of \autoref{sec:energy and equilibria} is to show that the total energy is non-increasing, which also enables us to determine the set of equilibria.
We then assert the main results of the paper, \autoref{thm:loc strong wp} on the local strong well-posedness for large data and \autoref{thm:global strong wp for small data} on the global strong well-posedness and exponential decay for small data, in \autoref{sec:main results}.
In \autoref{sec:change of var}, we transform the interaction problem to a fixed domain.
\autoref{sec:lin theory} is devoted to the linear theory; we reformulate the linearized problem by introducing an associated fluid-structure operator and then verify its maximal $\rL^p$-regularity.
The maximal $\rL^p(\R_+)$-regularity and exponential decay for the linearization around zero follow from spectral analysis.
In \autoref{sec:proof local}, we prove \autoref{thm:loc strong wp} based on the maximal regularity and nonlinear estimates.
\autoref{sec:proof of glob strong wp} finally discusses the proof of \autoref{thm:global strong wp for small data}.

\section{The fluid-rigid body interaction problem of the general $Q$-tensor model}\label{sec:model}

First, we describe the general $Q$-tensor model.
By $\cF_T$, we denote the spacetime associated to the fluid domain and made precise below.
The variables are the fluid velocity $u^{\cF} \colon \cF_T \to \R^3$, the pressure $\pi^{\cF} \colon \cF_T \to \R$ and the molecular orientation of the liquid crystal $Q^{\cF} \colon \cF_T \to \bS_{0,\R}^3$.
Here $\bS_{0,\R}^3$ denotes the space of symmetric $3 \times 3$-matrices with trace zero, so $\bS_{0,\R}^3 \coloneqq \{M \in \R^{3 \times 3} : M = M^\top \tand \tr M = 0\}$.
The complex version $\bS_{0,\C}^3$ is defined analogously.
We also write $\bS_0^3$ if no confusion arises.
For $\bK \in \{\R,\C\}$, we denote by $\bM_{0,\bK}^3 \coloneqq \{M \in \bK^{3 \times 3} : \tr M = 0\}$ the space of traceless matrices in~$\bK^{3 \times 3}$ and also occasionally omit the subscript $_\bK$.
More background on $Q$-tensor models, including the relation with the Ericksen-Leslie model, can be found in \cite{BZ:11}, see also \cite[Section~2]{HHW:24}.

One feature of the $Q$-tensor model is the free energy functional taking the form
\begin{equation}\label{eq:free energy functional}
    \rE_Q(Q^\cF) = \int_{\cF(t)} \Bigl(\frac{\lambda}{2} |\nabla Q^\cF(x)|^2 + \frac{a}{2}\tr\bigl((Q^\cF(x))^2\bigr) - \frac{b}{3}\tr\bigl((Q^\cF(x))^3\bigr)\Bigr) + \frac{c}{4}\tr\bigl((Q^\cF(x))^2\bigr)^2 \rd x
\end{equation}
for constants $a$, $\lambda > 0$ and $b$, $c \in \R$.
The first addend in \eqref{eq:free energy functional} corresponds to the elastic energy, while the other terms are related to the Landau-de Gennes thermotropic energy, see \cite{dGP:95}.

For $\D \coloneqq \D(u^{\cF}) = \nicefrac{1}{2}(\nabla u^{\cF} + (\nabla u^{\cF})^\top)$, and in the presence of tumbling and alignment effects, the full stress tensor with non-Newtonian stresses reads as
\begin{equation}\label{eq:complete stress tensor Q-tensor}
    \Tilde{\Sigma}(u^{\cF},\pi^{\cF},Q^{\cF}) = 2\mu \D(u^{\cF}) + \tau(Q^{\cF},H(Q^{\cF})) + \sigma(Q^{\cF},H(Q^{\cF})) - \pi^\cF \Id_3,
\end{equation}
where $\mu > 0$ represents a viscosity parameter.
Using the one constant approximation for the Oseen-Frank energy of liquid crystals and a Landau-de Gennes expression for the bulk energy, the variational derivative~$H(Q^{\cF})$ of the free energy functional $\rE_Q$ from \eqref{eq:free energy functional} is
\begin{equation}\label{eq:H Q-tensor}
    H(Q^{\cF}) \coloneqq \lambda \Delta Q^{\cF} - a Q^{\cF} + b\bigl((Q^{\cF})^2 - \tr\bigl((Q^{\cF})^2\bigr) \nicefrac{\Id_3}{3}\bigr) - c \tr\bigl((Q^{\cF})^2\bigr) Q^{\cF},
\end{equation}
while the symmetric and antisymmetric part of the stress tensor are given by
\begin{equation}\label{eq:tau and sigma}
    \begin{aligned}
        \tau(Q^\cF) = \tau(Q^{\cF},H(Q^{\cF}))
        &\coloneqq - \lambda \nabla Q^{\cF} \odot \nabla Q^{\cF} - \xi(Q^{\cF} + \nicefrac{\Id_3}{3}) H(Q^{\cF}) - \xi H(Q^{\cF})(Q^{\cF} + \nicefrac{\Id_3}{3})\\
        &\quad + 2 \xi(Q^{\cF} + \nicefrac{\Id_3}{3}) \tr\bigl(Q^{\cF} H(Q^{\cF})\bigr) \tand\\
        \sigma(Q^\cF) = \sigma(Q^{\cF},H(Q^{\cF})) 
        &\coloneqq Q^{\cF} H(Q^{\cF}) - H(Q^{\cF}) Q^{\cF} = \lambda \bigl(Q^{\cF} \Delta Q^{\cF} - \Delta Q^{\cF} Q^{\cF}\bigr),
    \end{aligned}
\end{equation}
respectively.
Here the parameter $\xi \in \R$ measures the ratio of tumbling and alignment effects, and the $(i,j)$-component of $\nabla Q^{\cF} \odot \nabla Q^{\cF}$ equals $\tr(\del_i Q^{\cF} \del_j Q^{\cF})$.
With $\W \coloneqq \W(u^{\cF}) = \nicefrac{1}{2}(\nabla u^{\cF} - (\nabla u^{\cF})^\top)$, we introduce the expression $S(\nabla u^{\cF},Q^{\cF})$ providing information on the way the gradient of the velocity $\nabla u^{\cF}$ stretches and rotates the order parameter~$Q^{\cF}$ and given by
\begin{equation*}
    S(\nabla u^{\cF},Q^{\cF}) \coloneqq (\xi \D + \W)(Q^{\cF} + \nicefrac{\Id_3}{3}) + (Q^{\cF} + \nicefrac{\Id_3}{3})(\xi \D - \W) - 2\xi (Q^{\cF} + \nicefrac{\Id_3}{3}) \tr\bigl(Q^{\cF} \nabla u^{\cF}\bigr).
\end{equation*}

For simplicity, we assume the parameters to be equal to $1$ in the sequel, so $\mu = \lambda = a = b = c = 1$, because it does not affect the analysis.
The resulting Beris-Edwards model is given by
\begin{equation}\label{eq:LC-fluid equationsiso}
\left\{
    \begin{aligned}
        {\partial_t u^{\cF}} + \left(u^{\cF}\cdot \nabla\right)u^{\cF} + \nabla \pi^{\cF} 
        &= \Delta u^{\cF} + \mdiv\bigl(\tau(Q^{\cF}) + \sigma(Q^{\cF})\bigr), \enspace \mdiv u^{\cF} = 0, &&\tin \cF_T ,\\
        \partial_t Q^{\cF} + \left(u^{\cF}\cdot \nabla\right)Q^{\cF} - S(\nabla u^{\cF},Q^{\cF}) 
        &= H(Q^{\cF}), &&\tin \cF_T .
    \end{aligned}
\right.
\end{equation}

In the sequel, we denote by $\cO \subset \R^3$ a bounded domain with boundary of class $\rC^3$, and $0 < T \le \infty$ represents a positive time.
We consider a rigid body moving in the interior of $\cO$, and the body is assumed to be closed, bounded and simply connected.
Furthermore, for $t \in (0,T)$, we denote by $\cS(t)$ the domain occupied by the rigid body at time $t$.
The rest of the domain $\cF(t) \coloneqq \cO \setminus \cS(t)$ is assumed to be filled by a viscous, incompressible fluid with liquid crystals.
We use the notation $\cS_0$ for the initial domain of the rigid body and suppose its boundary $\del \cS_0$ to be of class $\rC^3$ as well.
As a result, $\cF_0 =\cO \setminus \cS_0$ is the initial fluid domain.
On the other hand, the solid domain $\cS(t)$ at time $t$ is given by $\cS(t)= \left\{h(t)+ \bO(t)y : y\in \cS_{0}\right\}$, where $h(t)$ is the center of mass of the body, and $\bO(t) \in \SO(3)$ is linked to the rotation of the rigid body.
Since $\bO'(t)\bO^{\top}(t)$ is skew symmetric, there exists a unique $\Omega(t)\in \R^3$ with $\bO'(t)\bO^{\top}(t)y=\Omega (t)$ for all $t > 0$ and $y \in \R^3$.
By $\cF_T$, we denote the spacetime related to $\cF(t)$, i.\ e., $\cF_T \coloneqq \{(t,x) : t \in (0,T), \, x \in \cF(t)\}$, and~$\cS_T$ and $\del \cS_T$ corresponding to $\cS(t)$ and $\del \cS(t)$ are defined likewise.
The rigid body velocity is
\begin{equation*} 
    u^{\cS}(t,x)= h'(t)+  \Omega(t) \times (x-h(t)), \tforall (t,x) \in \cS_T, 
\end{equation*}
with the translational velocity $h' \colon (0,T) \to \R^3$ and the angular velocity of the body $\Omega \colon (0,T) \to \R^3$.  

We denote by $\rho_{\cS}$ the density of the rigid body and assume $\rho_{\cS} \equiv 1$.
As a result, $\mS = \int_{\cS(t)} \rd x$ is its total mass.
Without loss of generality, we suppose the initial position of the center of mass of the body to be at the origin, so $h(0) = 0$. Moreover, the inertia matrix $J(t)$ and its initial value $J_0$ take the shape
\begin{equation*}
    J_0 =  \int\limits_{\cS_0}\left(|x|^2\Id_3 - x\otimes x\right) \rd x \taswellas J(t) =  \int\limits_{\cS(t)}
 \left(|x-h(t)|^2\Id_3 - (x-h(t))\otimes (x-h(t))\right)\rd x.
\end{equation*}
Therefore, $J(t)$ is a symmetric, positive definite matrix and fulfills the identities
\begin{equation}\label{eq:Syl and Jab}
    J(t)=\bO(t)J_0 \bO^{\top}(t) \tand J(t)x_1\cdot x_2=\int\limits_{\cS_0} (x_1 \times \bO(t)y)\cdot (x_2\times \bO(t)y)\rd y, \tforall x_1,x_2 \in \R^3.
\end{equation}
Using the stress tensor $\Tilde{\Sigma}(u^{\cF},\pi^{\cF},Q^{\cF})$ from \eqref{eq:complete stress tensor Q-tensor}, and noting that $\nu(t)$ denotes the unit outward normal to the boundary of $\cF(t)$, i.\ e., $\nu(t)$ directed inwards to $\cS(t)$, the motion of the rigid body is governed by Newton's law
\begin{equation}\label{eq:rigid body eqs intro}
    \mS h''(t)
    = -\int\limits_{\partial \cS(t)} \Tilde{\Sigma}(u^{\cF},\pi^{\cF},Q^{\cF}) \nu(t)\rd \Gamma, \enspace 
    (J\Omega)' (t) = -\int\limits_{\partial \cS(t)} (x-h(t)) \times  \Tilde{\Sigma}(u^{\cF},\pi^{\cF},Q^{\cF}) \nu(t)\rd \Gamma.
\end{equation}

Next, we discuss the boundary and interface conditions.
We assume the fluid velocity to satisfy homogeneous boundary conditions on the outer boundary and to coincide with the rigid body velocity on their common interface, whereas $Q^{\cF}$ fulfills homogeneous Neumann boundary conditions, i.\ e.,
\begin{equation}\label{eq:bdry cond}
    u^{\cF} =0, \enspace \partial_\nu Q^{\cF} = 0, \ton (0,T) \times \del \cO, \tand u^{\cF} = u^{\cS}, \enspace \partial_\nu Q^{\cF} = 0, \ton \del \cS_T.
\end{equation}
The last ingredient for the full system are the initial conditions
\begin{equation}\label{eq:initial}
    u^{\cF}(0,x)=v_0(x),\enspace Q^{\cF}(0,x)=Q_0(x), \tin \cF_0,\enspace h(0)=0, \enspace h'(0)=\ell_0 \tand \Omega(0)=\omega_0.
\end{equation}

Let us conclude this section by settling some more concepts required in this article.
First, we make precise function spaces on time-dependent domains.
For this purpose, let $X \colon \cF_T \to \cF_0$ be a map with~$\varphi \colon \cF_T \to (0,T) \times \cF_0$, $(t,x) \mapsto (t,X(t,x))$ being a $\rC^1$-diffeomorphism, and $X(\tau,\cdot) \colon \cF(\tau) \to \cF_0$ being~$\rC^3$-diffeomorphisms for all $\tau \in [0,T]$.
For $p \in (1,\infty)$, $s \in \{0,1\}$ and $l \in \{0,1,2,3\}$, we then set
\begin{equation*}
    \rW^{s,p}\bigl(0,T;\rH^l(\cF(\cdot))\bigr) \coloneqq \bigl\{f(t,\cdot) \colon \cF(t) \to \R : f \circ \varphi \in \rW^{s,p}\bigl(0,T;\rH^l(\cF_0)\bigr)\bigr\}.
\end{equation*}
In this paper, $\rL_0^2(\cF_0)$ represents the $\rL^2(\cF_0)$-functions with spatial average zero, while $\rL_\sigma^2(\cF_0)$ denotes the solenoidal vector fields on $\rL^2(\cF_0)^3$, i.\ e., for the Helmholtz projection $\bP$, we define $\rL_\sigma^2(\cF_0) \coloneqq \bP \rL^2(\cF_0)^3$.
For $z_0 = (v_0,Q_0,\ell_0,\omega_0)$, $\rY_\gamma \coloneqq \rB_{2p}^{2-\nicefrac{2}{p}}(\cF_0)^3 \cap \rL_{\sigma}^2(\cF_0) \times \rB_{2p}^{3 - \nicefrac{2}{p}}(\cF_0,\bS_0^3) \times \R^3 \times \R^3$ and $p > \nicefrac{4}{3}$, we define
\begin{equation}\label{eq:X_gamma}
    \rX_\gamma \coloneqq \{z_0 \in \rY_\gamma : v_0 = 0, \enspace \del_n Q_0 = 0, \ton \del \cO, \enspace v_0 = \ell_0 + \omega_0 \times y, \enspace \del_n Q_0 = 0, \ton \del \cS_0\},
\end{equation}
where $\rB_{2p}^s(\cF_0)$ denotes the Besov space.
The traces and normal derivatives are well-defined due to $p > \nicefrac{4}{3}$, see e.\ g.\ \cite[Theorem~4.7.1]{Tri:78}.
In the main results, we will assume that $p > 4$ in order to establish suitable estimates of the nonlinear terms.
These estimates, carried out in detail in \autoref{sec:proof local} and \autoref{sec:proof of glob strong wp}, are also based on Sobolev embeddings, which require that $p > 4$.
If $p \in (1,\nicefrac{4}{3})$, then we set $\rX_\gamma \coloneqq \{z_0 \in \rY_\gamma : v_0 \cdot n = 0, \ton \del \cO, \enspace v_0 \cdot n = (\ell_0 + \omega_0 \times y) \cdot n, \ton \del \cS_0\}$,  where $n$ denotes the outer unit normal vector to the fluid boundary, so it is directed towards $\del \cS_0$.

\section{Energy decay and equilibria of the interaction problem}\label{sec:energy and equilibria}

In this section, we show that the total energy of the interaction problem \eqref{eq:LC-fluid equationsiso}--\eqref{eq:initial} is a non-increasing functional, and we deduce the set of equilibria.
First, we decompose $\tau = \tau_{\rh} + \tau_{\rl}$ and $H = H_{\rh} + H_{\rl}$ into the respective higher and lower order terms.
Invoking the commutator $[A,B] = A B - B A$ and anticommutator $\{A,B\} = A B + B A$ of $A$, $B \in \C^{3 \times 3}$, and writing $\D = \D(u^\cF)$ and $\W = \W(u^\cF)$, we get
\begin{equation}\label{eq:notation for Q-tensor}
    \begin{aligned}
        S(\nabla u^\cF,Q^\cF)
        &= -[Q^\cF,\W] + \xi\bigl(\nicefrac{2}{3}\D + \{Q^\cF,\D\} - 2(Q^\cF + \nicefrac{\Id_3}{3}) \tr\bigl(Q^\cF \nabla u^\cF\bigr)\bigr),\\
        H_\rh(Q^\cF)
        &= \Delta Q^\cF - Q^\cF, \enspace H_\rl(Q^\cF) = \bigl((Q^\cF)^2 - \tr\bigl((Q^\cF)^2\bigr)\nicefrac{\Id_3}{3}\bigr) - \tr\bigl((Q^\cF)^2\bigr) Q^\cF,\\
        \sigma(Q^\cF) 
        &= -[Q^\cF,(-\Delta + \Id_3)Q^\cF],\\ \tau_\rh(Q^\cF)
        &= -\xi\bigl(\nicefrac{2}{3}H_\rh + \{Q^\cF,H_\rh\} - 2(Q^\cF + \nicefrac{\Id_3}{3})\tr(Q^\cF H_\rh)\bigr) \tand\\
        \tau_\rl(Q^\cF)
        &= 2 \xi (Q^\cF + \nicefrac{\Id_3}{3}) \bigl(\tr\bigl((Q^\cF)^3\bigr) - \tr\bigl((Q^\cF)^2\bigr)^2\bigr) - \nabla Q^\cF \odot \nabla Q^\cF - 2 \xi(Q^\cF + \nicefrac{\Id_3}{3})H_\rl.
    \end{aligned}
\end{equation}
For $Q^\cF \in \bS_{0,\C}^3$, we define the maps $S_\xi(Q^\cF) \in \cL(\bS_{0,\C}^3,\bM_{0,\C}^3)$ as well as $\tS_\xi(Q^\cF) \in \cL(\bM_{0,\C}^3,\bS_{0,\C}^3)$ by
\begin{equation}\label{eq:S_xi and tS_xi}
    \begin{aligned}
        S_\xi(Q^\cF)A
        &\coloneqq [Q^\cF,A] - \nicefrac{2 \xi}{3}A - \xi\{Q^\cF,A\} + 2 \xi(Q^\cF + \nicefrac{\Id_3}{3}) \tr(Q^\cF A) \tand\\
        \tS_\xi(Q^\cF)B
        &\coloneqq [\overline{Q^\cF},\nicefrac{1}{2}(B - B^\top)] - \nicefrac{\xi}{3}(B + B^\top) - \xi\{\overline{Q^\cF},\nicefrac{1}{2}(B + B^\top)\}\\
        &\quad + \xi(\overline{Q^\cF} + \nicefrac{\Id_3}{3})\tr(\overline{Q^\cF} (B + B^\top)).
    \end{aligned}
\end{equation}
For real $u^\cF$, $Q^\cF$, we get $\tau_\rh(Q^\cF) + \sigma(Q^\cF) = S_\xi(Q^\cF)(\Delta - \Id_3)Q^\cF$ and $-S(\nabla u^\cF,Q^\cF) = \tS_\xi(Q^\cF) \nabla u^\cF$.
The following lemma, for which we refer to \cite[Lemma~5.1]{HHW:24}, establishes a link between $S_\xi$ and $\tS_\xi$.
Note that we also denote the product $\langle A, B \rangle_{\C^{3 \times 3}}$ by $A :B$, and it is defined by $\tr(A \overline{B})$.

\begin{lem}\label{lem:rel of S_xi and tS_xi}
For $\xi \in \R$, $Q^\cF \in \bS_{0}^3$, $A \in \bS_{0}^3$ and $B \in \bM_{0}^3$, we have $\langle S_\xi(Q^\cF)A,B \rangle_{\C^{3 \times 3}} = \langle A,\tS_\xi(Q^\cF)B \rangle_{\C^{3 \times 3}}$.
\end{lem}

Now, we investigate the total energy of the interaction problem \eqref{eq:LC-fluid equationsiso}--\eqref{eq:initial}.
Recalling the free energy functional $\rE_Q$ from \eqref{eq:free energy functional}, we define the total energy, which additionally comprises the kinetic energy and the energy resulting from the rigid body motion, by
\begin{equation}\label{eq:total energy}
    \rE(t) \coloneqq \rE_\mathrm{kin}(t) + \rE_Q(t) + \rE_\mathrm{trans}(t) + \rE_\mathrm{rot}(t), \twhere 
\end{equation}
\begin{equation*}
    \rE_\mathrm{kin}(t) = \frac{1}{2} \int_{\cF(t)} |u^\cF(t,x)|^2 \rd x, \enspace \rE_\mathrm{trans}(t) = \frac{1}{2} \mS |h'(t)|^2 \tand \rE_\mathrm{rot}(t) = \frac{1}{2} (J \Omega \cdot \Omega)(t).
\end{equation*}

The next proposition discusses $\frac{\rd}{\rd t}\rE(t)$ and the resulting set of equilibria $\cE$ to \eqref{eq:LC-fluid equationsiso}--\eqref{eq:bdry cond}. 

\begin{prop}\label{prop:energy and equilibria}
Let $p > 4$ and $z_0 = (v_0,Q_0,\ell_0,\omega_0) \in \rX_\gamma$, for $\rX_\gamma$ as introduced in \eqref{eq:X_gamma}, recall $H$ from~\eqref{eq:H Q-tensor}, and consider a solution $(u^\cF,Q^\cF,h',\Omega,\pi^\cF)$, whose existence will be discussed in \autoref{thm:loc strong wp}, to \eqref{eq:LC-fluid equationsiso}--\eqref{eq:initial} in the regularity class
\begin{equation*}
    \begin{aligned}
        &\rW^{1,p}\bigl(0,T;\rL^2(\cF(\cdot))^3\bigr) \cap \rL^p\bigl(0,T;\rH^{2}(\cF(\cdot))^3\bigr) \times \rW^{1,p}\bigl(0,T;\rH^1(\cF(\cdot),\bS_{0,\R}^3)\bigr) \cap \rL^p\bigl(0,T;\rH^3(\cF(\cdot),\bS_{0,\R}^3)\bigr)\\
        &\times \rW^{1,p}(0,T)^3 \times \rW^{1,p}(0,T)^3 \times \rL^p \bigl(0,T;\rH^{1}(\cF(\cdot)) \cap \rL_0^2(\cF(\cdot))\bigr).
    \end{aligned}
\end{equation*}
\begin{enumerate}[(a)]
    \item The energy functional $\rE(t)$ from \eqref{eq:total energy} satisfies $\frac{\mathrm{d}}{\mathrm{d}t} \rE(t) = -\int_{\cF(t)} |\nabla u^\cF|^2 \rd x - \int_{\cF(t)} \tr(H^2) \rd x \le 0$.
    \item We have $\cE = \{(0,Q_*,0,0) : H(Q_*) = 0, \text{ in } \cF_0, \enspace \del_\nu Q_* = 0, \text{ on } \del \cF_0\}$ for the set of equilibria.
\end{enumerate}
\end{prop}

\begin{proof}
Mimicking the procedure from the proof of \cite[Proposition~3.2]{BBHR:23} with the stress tensor $\Tilde{\Sigma}(u^{\cF},\pi^{\cF},Q^{\cF})$ from \eqref{eq:complete stress tensor Q-tensor}, using integration by parts in conjunction with the Reynolds transport theorem and the boundary conditions \eqref{eq:bdry cond}, for $\tau$ and $\sigma$ from \eqref{eq:tau and sigma}, and with the outer unit normal vector $n$, we infer that
\begin{equation*}
    \frac{\mathrm{d}}{\mathrm{d}t} \rE_\mathrm{kin}(t) = -\int_{\cF(t)} |\nabla u^\cF|^2 \rd x - \int_{\cF(t)} (\tau(Q^\cF) + \sigma(Q^\cF)) : \nabla u^\cF \rd x - \int_{\del \cS(t)} \Tilde{\Sigma}(u^{\cF},\pi^{\cF},Q^{\cF}) \nu(t) \cdot u^\cF \rd \Gamma.
\end{equation*}
Similarly as in \cite[(3.15) and (3.17)]{BBHR:23}, we find $\frac{\mathrm{d}}{\mathrm{d}t} \rE_\mathrm{trans}(t) + \frac{\mathrm{d}}{\mathrm{d}t} \rE_\mathrm{rot}(t) = \int_{\del \cS(t)} \Tilde{\Sigma}(u^{\cF},\pi^{\cF},Q^{\cF}) \nu(t) \cdot u^\cF \rd \Gamma$.
A concatenation of the latter two identities leads to
\begin{equation}\label{eq:ddt E_kin, E_trans and E_rot}
    \frac{\mathrm{d}}{\mathrm{d}t} \rE_\mathrm{kin}(t) + \frac{\mathrm{d}}{\mathrm{d}t} \rE_\mathrm{trans}(t) + \frac{\mathrm{d}}{\mathrm{d}t} \rE_\mathrm{rot}(t) = -\int_{\cF(t)} |\nabla u^\cF|^2 \rd x - \int_{\cF(t)} (\tau(Q^\cF) + \sigma(Q^\cF)) : \nabla u^\cF \rd x.
\end{equation}
It remains to handle $\rE_Q$ from \eqref{eq:free energy functional}.
For this purpose, we multiply the equation \eqref{eq:LC-fluid equationsiso}$_2$ satisfied by the molecular orientation $Q^\cF$ from the right by $-H$ and integrate over $\cF(t)$.
This yields
\begin{equation}\label{eq:tested Q-eq for energy}
    -\int_{\cF(t)} \del_t Q^\cF : H \rd x - \int_{\cF(t)} (u^\cF \cdot \nabla) Q^\cF : H \rd x + \int_{\cF(t)} S(\nabla u^\cF,Q^\cF) : H \rd x = -\int_{\cF(t)} \tr(H^2) \rd x.
\end{equation}
From \autoref{lem:rel of S_xi and tS_xi} and the aforementioned identities $\tau_\rh(Q^\cF) + \sigma(Q^\cF) = S_\xi(Q^\cF)((\Delta - \Id_3)Q^\cF)$ as well as~$-S(\nabla u^\cF,Q^\cF) = \tS_\xi(Q^\cF) \nabla u^\cF$ and~$S_\xi(Q^\cF) H_\rl = \tau_\rl(Q^\cF) + \nabla Q^\cF \odot \nabla Q^\cF$, we deduce that
\begin{equation}\label{eq:int of S:H}
    \int_{\cF(t)} S(\nabla u^\cF,Q^\cF) : H \rd x = -\int_{\cF(t)} \nabla u^\cF : (\tau(Q^\cF) + \sigma(Q^\cF)) \rd x - \int_{\cF(t)} \nabla u^\cF : (\nabla Q^\cF \odot \nabla Q^\cF) \rd x.
\end{equation}
Next, we calculate $\frac{\rd}{\rd t}\rE_Q$.
In fact, the Reynolds transport theorem and computations, also relying on the identity $\tr(\del_t Q^\cF) = \del_t \tr(Q^\cF) = 0$ by $Q^\cF \in \bS_{0,\R}^3$, yield that $\frac{\mathrm{d}}{\mathrm{d}t} \rE_Q(t)$ is given by
\begin{equation*}
    \begin{aligned}
        -\int_{\cF(t)} \del_t Q^\cF : H \rd x
         + \int_{\del \cS(t)} \Bigl(\frac{1}{2}|\nabla Q^\cF|^2 + \frac{1}{2}\tr\bigl((Q^\cF)^2\bigr) - \frac{1}{3}\tr\bigl((Q^\cF)^3\bigr) + \frac{1}{4}\tr\bigl((Q^\cF)^2\bigr)^2\Bigr) (u^\cF \cdot \nu(t)) \rd \Gamma.
    \end{aligned}
\end{equation*}
Putting this together with \eqref{eq:ddt E_kin, E_trans and E_rot}, and plugging in \eqref{eq:tested Q-eq for energy} as well as \eqref{eq:int of S:H}, we conclude that $\frac{\mathrm{d}}{\mathrm{d}t} \rE(t)$ equals
\begin{equation}\label{eq:ddt E}
    \begin{aligned}
        &-\int_{\cF(t)} |\nabla u^\cF|^2 \rd x - \int_{\cF(t)} \tr(H^2) \rd x + \int_{\cF(t)} (u^\cF \cdot \nabla) Q^\cF : H \rd x
        + \int_{\cF(t)} \nabla u^\cF : (\nabla Q^\cF \odot \nabla Q^\cF) \rd x\\
        &\int_{\del \cS(t)} \Bigl(\frac{1}{2}|\nabla Q^\cF|^2 + \frac{1}{2}\tr\bigl((Q^\cF)^2\bigr) - \frac{1}{3}\tr\bigl((Q^\cF)^3\bigr) + \frac{1}{4}\tr\bigl((Q^\cF)^2\bigr)^2\Bigr) (u^\cF \cdot \nu(t)) \rd \Gamma.
    \end{aligned}
\end{equation}
For verifying~(a), it thus remains to show that the last three terms in \eqref{eq:ddt E} cancel out.
To this end, we address $\int_{\cF(t)} (u^\cF \cdot \nabla) Q^\cF : H \rd x$ and discuss the resulting integrals for each term in $H$.
Integrating by parts, and employing $\mdiv u^\cF = 0$, \eqref{eq:bdry cond} and the identities $(u^\cF \cdot \nabla)((Q^\cF)^2) = (u^\cF \cdot \nabla) Q^\cF \cdot Q^\cF + Q^\cF \cdot (u^\cF \cdot \nabla) Q^\cF$, $-\int_{\cF(t)} (u^\cF \cdot \nabla)Q^\cF : (\tr((Q^\cF)^2) \nicefrac{\Id_3}{3}) \rd x = 0$ thanks to $\tr(Q^\cF) = 0$ as well as the identity involving the above inner product $Q^\cF : \bigl((u^\cF \cdot \nabla)(\tr((Q^\cF)^2)Q^\cF)\bigr) = 3 (u^\cF \cdot \nabla)Q^\cF : \bigl(\tr((Q^\cF)^2)Q^\cF\bigr)$, we find that
\begin{equation}\label{eq:colon most terms in H}
    \begin{aligned}
        &\quad \int_{\cF(t)} (u^\cF \cdot \nabla)Q^\cF : \Bigl(-Q^\cF + (Q^\cF)^2 - \tr((Q^\cF)^2) \nicefrac{\Id_3}{3} - \tr\bigl((Q^\cF)^2\bigr)Q^\cF\Bigr) \rd x\\
        &= -\int_{\del \cS(t)} \Bigl(\frac{1}{2} \tr\bigl((Q^\cF)^2\bigr) - \frac{1}{3} \tr\bigl((Q^\cF)^3\bigr) + \frac{1}{4} \tr\bigl((Q^\cF)^2\bigr)^2\Bigr) (u^\cF \cdot \nu(t)) \rd \Gamma.
    \end{aligned}
\end{equation}
By \eqref{eq:H Q-tensor}, it remains to consider the product with $\Delta Q^\cF$.
Integrating by parts, we obtain
\begin{equation*}
    \begin{aligned}
        \int_{\cF(t)} (u^\cF \cdot \nabla)Q^\cF : \Delta Q^\cF \rd x 
        &= -\int_{\cF(t)} \nabla u^\cF : (\nabla Q^\cF \odot \nabla Q^\cF) \rd x - \sum_{i,j,k=1}^3 \int_{\cF(t)} u_k^\cF \del_k \nabla Q_{ij}^\cF \cdot \nabla Q_{ji}^\cF \rd x.
    \end{aligned}
\end{equation*}
For the second term, we obtain $- \sum_{i,j,k=1}^3 \int_{\cF(t)} u_k^\cF \del_k \nabla Q_{ij}^\cF \cdot \nabla Q_{ji}^\cF \rd x = -\frac{1}{2} \int_{\del \cS(t)} |\nabla Q^\cF|^2 (u^\cF \cdot \nu(t)) \rd \Gamma$, so
\begin{equation}\label{eq:colon Delta Q}
    \int_{\cF(t)} (u^\cF \cdot \nabla)Q^\cF : \Delta Q^\cF \rd x = -\int_{\cF(t)} \nabla u^\cF : (\nabla Q^\cF \odot \nabla Q^\cF) \rd x -\frac{1}{2} \int_{\del \cS(t)} |\nabla Q^\cF|^2 (u^\cF \cdot \nu(t)) \rd \Gamma.
\end{equation}
Inserting \eqref{eq:colon most terms in H} and \eqref{eq:colon Delta Q} into \eqref{eq:ddt E}, we recover the assertion of~(a).

Concerning~(b), we argue that the equilibria $\cE$ are the critical points of the energy functional $\rE$, so assume $\left.\frac{\rd}{\rd t} \rE(t) \right|_{t=t_0} = 0$ at $t_0 > 0$.
From~(a), it follows that $\nabla u^\cF = 0$ and thus also $u^\cF = 0$ in $\cF(t_0)$ thanks to $u^\cF = 0$ on $\del \cO$.
As in the proof of \cite[Lemma~2.1]{RT:19}, we deduce that $h' = \Omega = 0$.
The relation in~(a) also shows that $H(Q_*) = 0$ has to hold for an equilibrium, so $\cE$ is of the asserted shape.
\end{proof}

\begin{rem}\label{rem:spat const equilibria}
Combining \autoref{prop:energy and equilibria}(b) and \cite[Remark~3.3]{HHW:24}, stating that $H(Q_*) = 0$ implies $Q_* = 0$ for spatially constant $Q_*$, we conclude that $(0,0,0,0)$ is the only such equilibrium to \eqref{eq:LC-fluid equationsiso}--\eqref{eq:bdry cond}.
\end{rem}

\section{Main results}\label{sec:main results}

The first main result asserts the local strong well-posedness of the interaction problem for initial data in~$\rX_\gamma$.
In contrast to \cite{geng2023global}, which establishes the existence of a weak solution to a slightly simplified version of the $Q$-tensor model, we consider here the general $Q$-tensor model with arbitrary ratio $\xi$ of tumbling and alignment effects in the strong setting, and we obtain uniqueness of the strong solution.

\begin{thm}\label{thm:loc strong wp}
Let $p > 4$, consider a bounded domain $\cO \subset \R^3$ of class $\rC^3$ and the domains $\cS_0$, $\cF_0$ of the rigid body and the fluid at time zero, and let $z_0 = (v_0,Q_0,\ell_0,\omega_0) \in \rX_\gamma$.
If $\dist(\cS_0,\del \cO) > r$ for some $r > 0$, then there exist $T' \in (0,T]$ and $X \in \rC^1([0,T'];\rC^2(\R^3)^3)$ so that $X(\tau,\cdot) \colon \cF(\tau) \to \cF_0$ are $\rC^2$-diffeomorphisms for all $\tau \in [0,T']$, and the interaction problem \eqref{eq:LC-fluid equationsiso}--\eqref{eq:initial} has a unique solution
\begin{equation*}
    \begin{aligned}
        u^\cF 
        &\in \rW^{1,p}\bigl(0,T';\rL^2(\cF(\cdot))^3\bigr) \cap \rL^p\bigl(0,T';\rH^{2}(\cF(\cdot))^3\bigr),\\
        Q^\cF 
        &\in \rW^{1,p}\bigl(0,T';\rH^1(\cF(\cdot),\bS_{0}^3)\bigr) \cap \rL^p\bigl(0,T';\rH^3(\cF(\cdot),\bS_{0}^3)\bigr),\\
        h' 
        &\in \rW^{1,p}(0,T')^3, \enspace h \in \rL^\infty(0,T')^3, \enspace  \Omega \in \rW^{1,p}(0,T')^3 \tand \pi^\cF \in \rL^p\bigl(0,T';\rH^{1}(\cF(\cdot)) \cap \rL_0^2(\cF(\cdot))\bigr).
    \end{aligned}
\end{equation*}
\end{thm}
A similar result can be obtained with external forces and torques in~\eqref{eq:rigid body eqs intro} provided they lie in $\rL^p(0,T)^3$.
The intersection with $\rL_0^2(\cF(\cdot))$ for the pressure rules out constants and thus leads to uniqueness. 

It readily follows that the zero solution is an equilibrium solution to \eqref{eq:LC-fluid equationsiso}--\eqref{eq:bdry cond}.
In \autoref{prop:energy and equilibria} and \autoref{rem:spat const equilibria}, it has been revealed that the zero solution is the only spatially constant equilibrium.
The second main result below states the existence of a unique, \emph{global-in-time}, strong solution for small data as well as the exponential convergence of this solution to the trivial equilibrium.

\begin{thm}\label{thm:global strong wp for small data}
Let $p > 4$, let $\cO \subset \R^3$ be a bounded $\rC^3$-domain, consider the initial domains of the rigid body and the fluid $\cS_0$ and $\cF_0$, and assume $\dist(\cS_0,\del \cO) > r$, for $r > 0$.
Then there is $\eta_0 > 0$ so that for all $\eta \in (0,\eta_0)$, there are constants $\delta_0 > 0$ and $C > 0$ such that for all $z_0 = (v_0,Q_0,\ell_0,\omega_0) \in \overline{B}_{\rX_\gamma}(z_0,\delta)$, with $\delta \in (0,\delta_0)$, there exists a unique solution $(u^\cF,Q^\cF,h',\Omega,\pi^\cF)$ to \eqref{eq:LC-fluid equationsiso}--\eqref{eq:initial} satisfying 
\begin{equation*}
    \begin{aligned}
        &\| \mre^{\eta(\cdot)} u^\cF \|_{\rL^p(0,\infty;\rH^2(\cF(\cdot)))} + \| \mre^{\eta(\cdot)} \del_t u^\cF \|_{\rL^p(0,\infty;\rL^2(\cF(\cdot)))} + \| \mre^{\eta(\cdot)} u^\cF \|_{\rL^\infty(0,\infty;\rB_{2p}^{2-\nicefrac{2}{p}}(\cF(\cdot)))}\\
        & +\| \mre^{\eta(\cdot)} Q^\cF \|_{\rL^p(0,\infty;\rH^3(\cF(\cdot)))} + \| \mre^{\eta(\cdot)} \del_t Q^\cF \|_{\rL^p(0,\infty;\rH^1(\cF(\cdot)))} + \| \mre^{\eta(\cdot)} Q^\cF \|_{\rL^\infty(0,\infty;\rB_{2p}^{3-\nicefrac{2}{p}}(\cF(\cdot)))}\\
        & + \| \mre^{\eta(\cdot)} h' \|_{\rW^{1,p}(0,\infty)} + \| h \|_{\rL^\infty(0,\infty)} + \| \mre^{\eta(\cdot)} \Omega \|_{\rW^{1,p}(0,\infty)} + \| \mre^{\eta(\cdot)} \pi^\cF \|_{\rL^p(0,\infty;\rH^1(\cF(\cdot)) \cap \rL_0^2(\cF(\cdot)))} \le C \delta.
    \end{aligned}
\end{equation*}
Furthermore, $\dist(\cS(t),\del \cO) > \nicefrac{r}{2}$ for all $t \in [0,\infty)$, so no collision with the fluid boundary occurs.
\end{thm}

\section{Transformation to a fixed domain}
\label{sec:change of var}

This section is devoted to the change of variables to the fixed domain.
In the sequel, we present a classical transformation which was e.\ g.\ used in \cite{IW:77}.
Denoting by $m(t)$ the skew-symmetric matrix satisfying $m(t) x = \Omega(t) \times x$, we investigate
\begin{equation}\label{eq:IVP X0}
\left\{
	\begin{aligned}
		\partial_t X_0(t,y) &= m(t)(X_0(t,y) - h(t)) + h'(t), &&\ton (0,T) \times \R^3, \\
		X_0(0,y) &= y, &&\tfor y \in \R^3. 
	\end{aligned}
\right. 	
\end{equation}
The solution to \eqref{eq:IVP X0} is of the form $X_0(t,y) = \bO(t) y + h(t)$ for $\bO(t) \in \SO(3)$, where $\bO \in \rW^{2,p}(0,T)^{3 \times 3}$ holds true if $h'$, $\Omega \in \rW^{1,p}(0,T)^3$.
At the same time, the inverse $Y_0(t,x) = \bO^{\top}(t)(x - h(t))$ of $X_0(t)$ solves $\partial_t Y_0(t,x) = -M(t) Y_0(t,x) - \ell(t)$ on $(0,T) \times \R^3$, where $M(t) \coloneqq \bO^{\top}(t) m(t) \bO(t)$ and $\ell(t) \coloneqq \bO^{\top}(t) h'(t)$, as well as $Y_0(0,x) = x$.

So far, the transformation acts globally in space, so we modify $X_0$ and $Y_0$ such that it only operates in a suitable open neighborhood of the moving body.
For this purpose, we introduce a new diffeomorphism~$X$, defined implicitly to be the solution to
\begin{equation}\label{eq:IVP X}
\left\{
	\begin{aligned}
		\partial_t X(t,y) &= b(t,X(t,y)), &&\ton (0,T) \times \R^3, \\
		X(0,y) &= y, &&\tfor y \in \R^3.
	\end{aligned}
\right. 	
\end{equation}
The right-hand side $b$ will force the transform to fulfill the above requirements.
Let us recall the assumption $\dist(\cS_0,\partial \cO) > r$ for some $r > 0$.
By the continuity of the body velocity, we take into account the solution up to a time for which a small distance such as $\nicefrac{r}{2}$ of the rigid body and the outer boundary is guaranteed.
Thus, we consider $\chi \in \rC^{\infty}(\R^3;[0,1])$ such that $\chi(x) = 1$ if $\dist(x,\partial \cO) \ge r$ and $\chi(x) = 0$ if $\dist(x,\partial \cO) \le \nicefrac{r}{2}$.
With this cut-off, we set the right-hand side $b \colon [0,T] \times \R^3 \to \R^3$ of \eqref{eq:IVP X} to be
\begin{equation*}
    b(t,x) \coloneqq \chi(x)[m(t)(x - h(t)) + h'(t)] - B_{\cF_0}(\nabla \chi(\cdot)[m(t)(\cdot - h(t)) + h'(t)])(x),
\end{equation*}
with $B_{\cF_0} \colon \rC_c^\infty(\cF_0) \to \rC_c^\infty(\cF_0)^3$ being the Bogovski\u{\i} operator associated to $\cF_0$.
Similarly as in \cite[Section~4]{BBHR:23}, we argue that $b \in \rW^{1,p}\bigl(0,T;\rC_{c,\sigma}(\cF_0)^3\bigr)$, where the subscript $_\sigma$ indicates that $\mdiv(b(t,\cdot)) = 0$, and $b|_{\del \cS_0} = m(x-h) + h'$.

Besides, the Picard-Lindel\"of theorem yields the existence of a unique solution $X \in \rC^1\bigl((0,T);\rC^\infty(\R^3)^3\bigr)$ to \eqref{eq:IVP X} for given $h'$, $\Omega \in \rW^{1,p}(0,T)^3$.
This solution has continuous mixed partial derivatives.
From the uniqueness of the solution, we deduce the bijectivity of the function $X(t,\cdot)$, and its inverse is denoted by $Y(t,\cdot)$.
Thanks to $\mdiv b = 0$, Liouville's theorem yields that $X$ and $Y$ are both volume-preserving, $\rJ_X(t,y) \rJ_Y(t,X(t,y)) = \Id_3$ and $\det \rJ_X(t,y) = \det \rJ_Y(t,x) = 1$, with $\rJ_X$ and $\rJ_Y$ representing the Jacobian matrices of $X$ and $Y$ as usual.
The inverse $Y$ of $X$ solves $\partial_t Y(t,x) = b^{(Y)}(t,Y(t,x))$ on~$(0,T) \times \R^3$, where the right-hand side is given by $b^{(Y)}(t,y) \coloneqq -\rJ_X^{-1}(t,y) b(t,X(t,y))$, and $Y(0,x) = x$.
Let us stress that~$b^{(Y)}$ and $Y$ possess the same space and time regularity as $b$ and $X$.
Moreover, the interval of existence of the solution is restricted by the condition $\dist(\cS(t),\partial \cO) > \nicefrac{r}{2}$ for all $t \in (0,T)$.
In fact, we can write
$\dist(\cS(t), \cS_0) \leq \|h(t)\|_{\mathbb{R}^3} + \|\mathbb{O}(t)-\mathbb{I}\|_{\mathbb{R}^{3\times 3}}|y|.$
As the rigid body moves with a continuous velocity, thanks to the estimates in \autoref{sec:proof local}, it is possible to find a sufficiently small time $T > 0$ such that $\dist(\cS(t), \cS_0) \leq \frac{r}{2}$ is satisfied for large initial data.
For small initial data, we will see that the exponential stability of the semigroup associated with the linearized problem at zero as revealed in \autoref{sec:lin theory} ensures that $\dist(\cS(t), \cS_0) \leq \frac{r}{2}$. Combining with the assumption $\dist(\cS_0,\del \cO) > r$, we obtain
$\dist(\cS(t),\partial \cO) > \nicefrac{r}{2}$ for all $t \in (0,\infty)$, so no collision occurs in finite time.

For $X \colon\R^3\rightarrow\R^3$ solving the Cauchy problem \eqref{eq:IVP X} and $(t,y)\in (0,T) \times \R^3$, we define the transformed variables and stress tensor by
\begin{equation*}
    \begin{aligned}
        v(t,y) 
        &\coloneqq \rJ_{Y}(t,X(t,y))u^{\cF}(t,X(t,y)), \enspace Q(t,y)\coloneqq Q^{\cF}(t,X(t,y)), \enspace \ell(t) \coloneqq \bO(t)^{\top}h'(t), \enspace \omega(t) \coloneqq \bO(t)^{\top}\Omega(t),\\
        \pi(t,y)
        &\coloneqq\pi^{\cF}(t,X(t,y))\tand \Sigma(v(t,y),\pi(t,y), Q(t,y)) \coloneqq \bO(t)^{\top}\Tilde{\Sigma}(\bO(t)v(t,y),\pi(t,y),Q(t,y))\bO(t).
    \end{aligned}
\end{equation*}

The metric contravariant $(g^{ij})$ and covariant tensors $(g_{ij})$ and the Christoffel symbol are given by
\begin{equation}\label{eq:contrvar covar and Christoffel}
    g^{ij}= \sum_{k=1}^3 (\partial_k Y_i)(\partial_k Y_j), \enspace g_{ij}= \sum_{k=1}^3 (\partial_i X_k)(\partial_j X_k) \tand \Gamma^{i}_{jk}= \frac{1}{2}\sum_{l=1}^3 g^{il} (\partial_k g_{jl} + \partial_j g_{kl} - \partial_l g_{jk}).
\end{equation}
From \cite[(3.13)]{GGH:13} and \cite[(3.13)]{BBH:23}, we recall the transformed Laplacians
\begin{equation}\label{eq:transformed Laplacians}
    \begin{aligned}
        (\cL_1 v)_{i} 
        &= \sum_{j,k=1}^{3} \partial_j (g^{jk}\partial_k v_i) + 2 \sum_{j,k,l=1}^{3}  g^{kl}\Gamma_{jk}^i(\partial_l v_j) +  \sum_{j,k,l=1}^{3}\Bigl(\partial_k(g^{kl}\Gamma^{i}_{kl}) + \sum_{m=1}^3 g^{kl}\Gamma^{m}_{jl}\Gamma^{i}_{km}\Bigr)v_j \tand\\
        (\cL_2 Q)_{ij} 
        &= \sum\limits_{k,l=1}^3 \left((\partial_{l}\partial_{k}Q_{ij})g^{kl}\right) + \sum\limits_{k=1}^3 (\Delta Y_k) (\partial_k Q_{ij}).
    \end{aligned}
\end{equation}
Now, we compute the transformed time derivative, convective term and gradient of the pressure
\begin{equation*}
    (\cM v)_{i} = \sum\limits_{j=1}^3 (\partial_t Y_j) (\partial_j v_i) + \sum\limits_{j,k=1}^3 \bigl(\Gamma^i_{jk}(\partial_t Y_k) + (\partial_k Y_i)(\partial_j \partial_t X_k)\bigr)v_j, \enspace
    (\cN(v))_{i} = \sum\limits_{j=1}^3 v_j(\partial_j v_i) + \sum\limits_{j,k=1}^3 \Gamma^i_{jk} v_j v_k
\end{equation*}
and $(\cG \pi)_{i} = \sum\limits_{j=1}^3 g^{ij}(\partial_j \pi)$.

Before addressing the transformed terms associated to $\mdiv \sigma$ and $\mdiv \tau$, we settle some notation, where we omit the superscript $^\cF$ for simplicity.
In that respect, we set $B_{\sigma}(Q)A$ so that it corresponds to~$\mdiv \sigma$, i.\ e., $B_{\sigma}(Q)Q = \mdiv(Q \Delta Q - \Delta Q Q)$.
Thus, we define $B_{\sigma}(Q)A \coloneqq B_{\sigma,1}(Q)A + B_{\sigma,2}(Q)A$ by
\begin{equation}\label{eq:B_sigma}
    [B_{\sigma}(Q)A]_i \coloneqq \sum_{j,k=1}^3 \Bigl(\bigl((\del_j Q_{ik}) (\Delta A_{kj}) + Q_{ik} (\del_j \Delta A_{kj})\bigr) - \bigl((\del_j \Delta A_{ik}) Q_{kj} + (\Delta A_{ik}) (\del_j Q_{kj})\bigr)\Bigr).
\end{equation}
The next term to be addressed is $\tau_\rh$ from \eqref{eq:notation for Q-tensor}.
We further split it into $\tau_\rh \coloneqq \xi(\tau_{\rh,1} + \tau_{\rh,2} + \tau_{\rh,3})$, where~$\tau_{\rh,1} \coloneqq -\nicefrac{2}{3}H_\rh$, $\tau_{\rh,2} \coloneqq -\{Q,H_\rh\}$ and $\tau_{\rh,3} \coloneqq 2(Q + \nicefrac{\Id_3}{3})\tr(Q H_\rh)$.
We also introduce $B_{\tau_{\rh,i}}(Q)A$ such that $B_{\tau_{\rh,i}}(Q)Q = \mdiv \tau_{\rh,i}$ for $i=1,2,3$.
The first term is $[B_{\tau_{\rh,1}}A]_i \coloneqq -\nicefrac{2}{3}\sum_{j=1}^3 \del_j (\Delta A_{ij} - A_{ij})$.
Observing that $\{Q,H_\rh\} = Q (\Delta Q) + (\Delta Q) Q - 2 Q^2$, we define
\begin{equation}\label{eq:B_tau_rh_2}
    [B_{\tau_{\rh,2}}(Q)A]_i 
    \coloneqq -(B_{\sigma,1}(Q)A)_i + (B_{\sigma,2}(Q)A)_i + 2 \sum_{j,k=1}^3 \bigl((\del_j A_{ik})Q_{kj} + Q_{ik}(\del_j A_{kj})\bigr).
\end{equation}
In the same way, calculating the term $\tr(Q H_\rh) = \sum_{k,l=1}^3 (Q_{kl} \Delta Q_{lk} - Q_{kl} Q_{lk})$, we introduce 
\begin{equation*}
    \begin{aligned}
        [B_{\tau_{\rh,3}}(Q)A]_i 
        &\coloneqq 2\sum_{j,k,l=1}^3 \Bigl((\del_j Q_{ij}) (Q_{kl} \Delta A_{lk} - Q_{kl} A_{lk}) + (Q_{ij} + \nicefrac{\delta_{ij}}{3})\bigl((\del_j Q_{kl}) \Delta A_{lk} + Q_{kl} (\del_j \Delta A_{lk})\\
        &\quad - (\del_j Q_{kl}) A_{lk} - Q_{kl} (\del_j A_{lk})\bigr)\Bigr).
    \end{aligned}
\end{equation*}
We also write $\tau_\rl = \xi \tau_{\rl,1} + \tau_{\rl,2} + \xi \tau_{\rl,3}$, with $\tau_{\rl,1} \coloneqq 2 (Q + \nicefrac{\Id_3}{3}) (\tr(Q^3) - \tr(Q^2)^2)$, $\tau_{\rl,2} \coloneqq - \nabla Q \odot \nabla Q$ and~$\tau_{\rl,3} \coloneqq 2 (Q + \nicefrac{\Id_3}{3})H_\rl$.
Again, we define $B_{\tau_{\rl,i}}(Q)A$ so that $B_{\tau_{\rl,i}}(Q)Q = \mdiv \tau_{\rl,i}$.
The first term reads as
\begin{equation*}
    \begin{aligned}
        [B_{\tau_{\rl,1}}(Q)A]_i
        &\coloneqq 2 \sum_{j,k,l=1}^3 \Biggl((\del_j A_{ij}) \biggl(\sum_{m=1}^3\Bigl( Q_{kl} Q_{lm} Q_{mk} - \sum_{n=1}^3 Q_{km} Q_{mk} Q_{ln} Q_{nl}\Bigr)\biggr) + (Q_{ij} + \nicefrac{\delta_{ij}}{3})\\
        &\quad \cdot \biggl(\sum_{m=1}^3\Bigl( (\del_j A_{kl})Q_{lm}Q_{mk}+ Q_{kl}(\del_j A_{lm})Q_{mk} + Q_{kl}Q_{lm}(\del_j A_{mk}) - \sum_{n=1}^3 \bigl((\del_j A_{km}) Q_{mk} Q_{ln} Q_{nl}\\
        &\qquad + Q_{km} (\del_j A_{mk}) Q_{ln} Q_{nl} + Q_{km} Q_{mk} (\del_j A_{ln}) Q_{nl} + Q_{km} Q_{mk} Q_{ln} (\del_j A_{nl})\bigr)\Bigr)\biggr)\Biggr).
    \end{aligned}
\end{equation*}
To capture $-\nabla Q \odot \nabla Q$, we set $[B_{\tau_{\rl,2}}(Q)A]_i \coloneqq - \sum_{j,k=1}^2 \bigl((\del_i Q_{jk}) \Delta A_{kj} + \sum_{l=1}^3 (\del_l Q_{kj}) (\del_l \del_i A_{jk})\bigr)$.
The term $B_{\tau_{\rl,3}}(Q)A$ associated to $\tau_{\rl,3}$ is of a similar form as $B_{\tau_{\rl,1}}(Q)A$.

Next, we calculate the respective transformed terms.
The transformed version of $B_\sigma$ from \eqref{eq:B_sigma} is
\begin{equation}\label{eq:cB_sigma}
     \cB_{\sigma}(Q)A \coloneqq \cB_{\sigma,1}(Q)A + \cB_{\sigma,2}(Q)A \coloneqq \cB_{\sigma,11}(Q)A + \cB_{\sigma,12}(Q)A + \cB_{\sigma,2}(Q)A, \twhere
\end{equation}
\begin{equation*}
    \begin{aligned}
        (\cB_{\sigma,11}(Q)A)_i
        &\coloneqq \sum_{j,k=1}^3 (\cL_2 A)_{kj} \sum_{l=1}^3 (\del_l Q_{ik})(\del_j Y_l), \enspace (\cB_{\sigma,12}(Q)A)_i \coloneqq \sum_{j,k=1}^3 Q_{ik} \Biggl(\sum_{l=1}^3 \biggl(\sum_{m=1}^3\Bigl(\sum_{n=1}^3 (\del_n \del_m \del_l A_{kj})\\
        &\quad \cdot (\del_j Y_n) g^{lm}\Bigr) + (\del_m \del_l A_{kj}) (\del_j g^{lm}) + (\del_m \del_l A_{kj})(\del_j Y_m) \Delta Y_l\biggr) + (\del_l A_{kj})(\del_j \Delta Y_l)\Biggr).
    \end{aligned}
\end{equation*}
By \eqref{eq:B_sigma}, the term $\cB_{\sigma,2}(Q)A$ takes a similar shape.
For the transformed versions of $B_{\tau_{\rh,i}}$, we first obtain
\begin{equation}\label{eq:cB_tau_rh_1}
    \begin{aligned}
        (\cB_{\tau_{\rh,1}}A)_i
        &= -\frac{2}{3}\sum_{j,k=1}^3 \biggl(\sum_{l=1}^3\Bigl(\sum_{m=1}^3 \bigl((\del_m \del_l \del_k A_{ij}) (\del_j Y_m) g^{kl}\bigr) + (\del_l \del_k A_{ij}) (\del_j g^{kl})\\
        &\qquad + (\del_l \del_k A_{ij}) (\del_j Y_l) \Delta Y_k\Bigr) + (\del_k A_{ij})(\del_j \Delta Y_k) + (\del_j Y_k)(\del_k A_{ij})\biggr).
    \end{aligned}
\end{equation}
From \eqref{eq:B_tau_rh_2}, it follows that
\begin{equation*}
    (\cB_{\tau_{\rh,2}}(Q)A)_i = -(\cB_{\sigma,1}(Q)A)_i + (\cB_{\sigma,2}(Q)A)_i + 2\sum_{j,k,l=1}^3 \bigl((\del_j Y_l) (\del_l A_{ik})Q_{kj} + Q_{ik}(\del_j Y_l)(\del_l A_{kj})\bigr).
\end{equation*}
The $i$-th component of the last term $\cB_{\tau_{\rh,3}}(Q)A$ in this context is given by
\begin{equation*}
    \begin{aligned}
        (\cB_{\tau_{\rh,3}}(Q)A)_i
        &= 2\sum_{j,k,l,m=1}^3 \biggl((\del_j Y_m) (\del_m Q_{ij})(Q_{kl}(\cL_2A)_{lk} - Q_{kl} A_{lk}) + (Q_{ij} + \nicefrac{\delta_{ij}}{3})\Bigl((\del_j Y_m)(\del_m Q_{kl})(\cL_2 A)_{lk}\\
        &\quad + (\cB_{\sigma,12}(Q)A)_j - (\del_j Y_m)(\del_m Q_{kl})A_{lk} - Q_{kl} (\del_j Y_m) (\del_m A_{lk})\Bigr)\biggr).
    \end{aligned}
\end{equation*}

Similarly, we obtain the transformed terms corresponding to $B_{\tau_{\rl}}$.
For brevity, we omit the details.
Let us summarize the transformed terms given above in a more compact form as
\begin{equation*}
    \begin{aligned}
        \cB_{\tau_\rh}(Q)A 
        &\coloneqq \xi(\cB_{\tau_{\rh,1}}A + \cB_{\tau_{\rh,2}}(Q)A + \cB_{\tau_{\rh,3}}(Q)A), \enspace \cB_{\tau_\rl}(Q)A \coloneqq \xi \cB_{\tau_{\rl,1}}(Q)A + \cB_{\tau_{\rl,2}}(Q)A + \xi \cB_{\tau_{\rl,3}}(Q)A,\\
        \cB_{\tau}(Q)A 
        &\coloneqq \cB_{\tau_\rh}(Q)A + \cB_{\tau_\rl}(Q)A \tand \cB_{\tau}(Q) \coloneqq \cB_{\tau}(Q)Q \taswellas \cB_{\sigma}(Q) \coloneqq \cB_{\sigma}(Q)Q.
    \end{aligned}
\end{equation*}

In the equation for $Q$, we first observe that $\del_t Q^{\cF} = \del_t Q + (\del_t Y \cdot \nabla) Q$.
For the convective term, we obtain $(u^{\cF} \cdot \nabla)Q^{\cF} = \sum_{k=1}^3 \bigl(\sum_{l=1}^3 (\del_l X_k) v_l\bigr) \bigl(\sum_{m=1}^3 (\del_m Q_{ij}) (\del_k Y_m)\bigr) = \sum_{l,m=1}^3 \delta_{lm} v_l (\del_m Q_{ij}) = (v \cdot \nabla) Q$.

Concerning the transformation of the gradient of $u^{\cF}$, we calculate
\begin{equation}\label{eq:transformed grad}
    \frac{\del u_i^{\cF}}{\del x_j} = \sum_{k=1}^3 \Bigl(\sum_{l=1}^3 (\del_l \del_k X_i) (\del_j Y_l) v_k + (\del_k X_i) \sum_{l=1}^3 (\del_l v_k) (\del_j Y_l)\Bigr) \eqqcolon \cD(v)_{ij} \eqqcolon \cD_{ij}.
\end{equation}
Consequently, $S(\nabla u^{\cF},Q^{\cF})$ from \eqref{eq:notation for Q-tensor} transforms as
\begin{equation}\label{eq:cS}
    \begin{aligned}
        (\cS(v,Q))_{ij} 
        &= -\frac{1}{2}\sum_{k=1}^3 \bigl(Q_{ik}(\cD_{kj} - \cD_{jk}) - (\cD_{ik} - \cD_{ki})Q_{kj}\bigr) + \xi\biggl(\frac{1}{3} (\cD_{ij} + \cD_{ji})\\
        &\qquad+ \frac{1}{2} \sum_{k=1}^3 Q_{ik} (\cD_{kj} + \cD_{jk}) + (\cD_{ik} + \cD_{ki})Q_{kj} - 2 \Bigl(\sum_{k,l=1}^3 Q_{kl} \cD_{lk}\Bigr) \Bigl[Q_{ij} + \frac{1}{3}\delta_{ij}\Bigr]\biggr).
    \end{aligned}
\end{equation}
The term $H(Q)$ from \eqref{eq:H Q-tensor} transforms to $\cH(Q) = \cL_2 Q - Q + (Q^2 - \tr(Q^2) \nicefrac{\Id_3}{3}) - \tr(Q^2) Q$.
Then for the unit outward normal $n = \bO^\top(t) \nu(t)$ to $\del \cF_0$, which is thus directed inwards to $\del \cS_0$, the complete transformed system associated to \eqref{eq:LC-fluid equationsiso}--\eqref{eq:initial} on the fixed domain $(0,T)\times\cF_0$ takes the shape
\begin{equation}\label{eq:cv1}
\left\{
    \begin{aligned}
        \partial_t v + (\cM-\cL_1)v +\cN(v)+\cG \pi
        &=\cB_\tau(Q) + \cB_\sigma(Q), \enspace \mdiv v = 0, &&\tin (0,T) \times \cF_0,\\
        {\partial_t Q} + (\partial_t Y \cdot \nabla) Q + (v \cdot \nabla)Q - \cS(v,Q)  
        &=  \cH(Q), &&\tin (0,T) \times \cF_0,\\
        \mS\ell' + \mS(\omega\times\ell)
        &= -\int\limits_{\partial \cS_0} \Sigma(v,\pi,Q) n\rd \Gamma, &&\tin (0,T),\\
        J_0\omega' - \omega\times (J_0\omega) 
        &=-\int\limits_{\partial \cS_0} y \times  \Sigma(v,\pi,Q) n\rd \Gamma, &&\tin (0,T).
    \end{aligned}
\right.
\end{equation}
The boundary and initial conditions are
\begin{equation}\label{eq:cv2}
    v = 0, \enspace \partial_n Q = 0 \ton (0,T) \times \partial \cO, \enspace v = \ell+\omega\times y, \enspace \partial_n Q = 0 \ton (0,T) \times \partial\cS_0 \tand
\end{equation}
\begin{equation}\label{eq:cv3}
    v(0)=v_0, \enspace Q(0) = Q_0, \enspace \ell(0) = \ell_0 \tand \omega(0) = \omega_0.
\end{equation}

Thanks to the above coordinate transform, we may rewrite the two main results, \autoref{thm:loc strong wp} and \autoref{thm:global strong wp for small data}, on the fixed domain $(0,T) \times \cF_0$, which we also refer to as the \emph{reference configuration}.

\begin{thm}\label{thm:loc strong wp in ref config}
Consider $p > 4$, a bounded $\rC^3$-domain $\cO \subset \R^3$, the initial rigid body and fluid domains $\cS_0$ and $\cF_0$ and $z_0 = (v_0,Q_0,\ell_0,\omega_0) \in \rX_\gamma$, with $\rX_\gamma$ as in \eqref{eq:X_gamma}.
If there is $r > 0$ with $\dist(\cS_0,\del \cO) > r$, then there exist $T' \in (0,T]$ and $X \in \rC^1([0,T'];\rC^2(\R^3)^3)$ so that $X(\tau,\cdot) \colon \cF(\tau) \to \cF_0$ are $\rC^2$-diffeomorphisms for all $\tau \in [0,T']$, and the interaction problem \eqref{eq:cv1}--\eqref{eq:cv3} has a unique solution
\begin{equation*}
    \begin{aligned}
        v
        &\in \rW^{1,p}\bigl(0,T';\rL^2(\cF_0)^3\bigr) \cap \rL^p\bigl(0,T';\rH^2(\cF_0)^3\bigr), \enspace Q
        \in \rW^{1,p}\bigl(0,T';\rH^1(\cF_0,\bS_{0}^3)\bigr) \cap \rL^p\bigl(0,T';\rH^3(\cF_0,\bS_{0}^3)\bigr),\\
        \ell&\in \rW^{1,p}(0,T')^3, \enspace \omega \in \rW^{1,p}(0,T')^3 \tand \pi \in \rL^p\bigl(0,T';\rH^1(\cF_0) \cap \rL_0^2(\cF_0)\bigr).
    \end{aligned}
\end{equation*}
\end{thm}

\autoref{thm:global strong wp for small data} can be reformulated in the following manner on the reference configuration.

\begin{thm}\label{thm:global strong wp for small data in ref config}
Let $p > 4$, let $\cO \subset \R^3$ be a bounded domain of class $\rC^3$, and consider the domains of the rigid body and the fluid at time zero $\cS_0$ and $\cF_0$.
Besides, assume that $\dist(\cS_0,\del \cO) > r$ for some $r > 0$.
Then there exists $\eta_0 > 0$ such that for all $\eta \in (0,\eta_0)$, there are $\delta_0 > 0$, $C > 0$ so that for every $\delta \in (0,\delta_0)$ and $z_0 = (v_0,Q_0,\ell_0,\omega_0) \in \overline{B}_{\rX_\gamma}(0,\delta)$, it holds that $X \in \rC^1([0,\infty);\rC^2(\R^3)^3)$, with $X(\tau,\cdot) \colon \cF(\tau) \to \cF_0$ being $\rC^2$-diffeomorphisms for all $\tau \in [0,\infty)$, and \eqref{eq:cv1}--\eqref{eq:cv3} has a unique solution $(v,Q,\ell,\omega,\pi)$ with
\begin{equation*}
    \begin{aligned}
        &\| \mre^{\eta(\cdot)} v \|_{\rL^p(0,\infty;\rH^2(\cF_0))} + \| \mre^{\eta(\cdot)} \del_t v \|_{\rL^p(0,\infty;\rL^2(\cF_0))} + \| \mre^{\eta(\cdot)} v \|_{\rL^\infty(0,\infty;\rB_{2p}^{2-\nicefrac{2}{p}}(\cF_0))}\\
        &+\| \mre^{\eta(\cdot)} Q \|_{\rL^p(0,\infty;\rH^3(\cF_0))} + \| \mre^{\eta(\cdot)} \del_t Q \|_{\rL^p(0,\infty;\rH^1(\cF_0))} + \| \mre^{\eta(\cdot)} Q \|_{\rL^\infty(0,\infty;\rB_{2p}^{3-\nicefrac{2}{p}}(\cF_0))}\\
        &+ \| \mre^{\eta(\cdot)} \ell \|_{\rW^{1,p}(0,\infty)} + \| \mre^{\eta(\cdot)} \omega \|_{\rW^{1,p}(0,\infty)} + \| \mre^{\eta(\cdot)} \pi \|_{\rL^p(0,\infty;\rH^1(\cF_0) \cap \rL_0^2(\cF_0))} \le C \delta.
    \end{aligned}
\end{equation*}
\end{thm}

\section{Maximal regularity of the linearized problem and exponential stability}
\label{sec:lin theory}

This section is central for the article.
Very recently, the sectoriality of the linearized \emph{general} $Q$-tensor problem has been shown in \cite{HHW:24}.
We take advantage of this property to derive the maximal $\rL^p$-regularity of the suitably linearized interaction problem by introducing a so-called \emph{fluid-structure operator} tailored to our setting.
In a second step, we prove the exponential stability of the corresponding semigroup.

Let us first introduce some function spaces.
For $J = (0,T)$, $0 < T \le \infty$, the data space is defined by
\begin{equation}\label{eq:data space}
    \F_T = \F_T^v \times \F_T^Q \times \F_T^\ell \times \F_T^\omega = \rL^p\bigl(J;\rL^2(\cF_0)^3\bigr) \times \rL^p\bigl(J;\rH^1(\cF_0,\bS_0^3)\bigr) \times \rL^p(J)^3 \times \rL^p(J)^3.
\end{equation}
Next, with $z = (v,Q,\ell,\omega)$, we introduce $\rX_0 \coloneqq \rL_{\sigma}^2(\cF_0) \times \rH^1(\cF_0,\bS_0^3) \times \R^3 \times \R^3$ and
\begin{equation*}
    \rX_1 \coloneqq \{z \in \rY_1 : v = 0, \enspace \del_n Q = 0, \ton \del \cO, \tand v = \ell + \omega \times y, \enspace \del_n Q = 0, \ton \del \cS_0\}
\end{equation*}
for $\rY_1 \coloneqq \rH^2(\cF_0)^3 \cap \rL_{\sigma}^2(\cF_0) \times \rH^3(\cF_0,\bS_0^3) \times \R^3 \times \R^3$.
As in \cite[Lemma~5.3]{BBH:23}, $\rX_\gamma$ from \eqref{eq:X_gamma} can be deduced from~$\rX_0$ and $\rX_1$ by real interpolation via $\rX_\gamma = (\rX_0,\rX_1)_{1-\nicefrac{1}{p},p}$.
We then define the solution space $\E_T$ by
\begin{equation}\label{eq:max reg space}
    \E_T \coloneqq \rW^{1,p}(J;\rX_0) \cap \rL^p(J;\rX_1).
\end{equation}
By $\prescript{}{0}{\E_T}$, we denote the associated space with homogeneous initial values.
For simplicity, we also write~$\E_T^Q$ for the $Q$-component of $\E_T$. 
In order to account for the pressure, we further introduce
\begin{equation}\label{eq:max reg space with pressure}
    \tE_T \coloneqq \E_T \times \E_T^{\pi}, \twhere \E_T^{\pi} \coloneqq \rL^p\bigl(J;\rH^1(\cF_0) \cap \rL_0^2(\cF_0)\bigr),
\end{equation}
and $\| \cdot \|_{\E_T^{\pi}}$ represents the norm associated to $\E_T^{\pi}$.
With regard to the stress tensor part in the surface integral, we include precisely its linear components to establish estimates in the proof of \autoref{thm:global strong wp for small data} in \autoref{sec:proof of glob strong wp}.
In addition to the usual Cauchy stress tensor, we thus incorporate $-\nicefrac{2 \xi}{3} (\Delta Q - Q)$, i.\ e.,
\begin{equation}\label{eq:lin stress tensor}
    \rT(v,\pi,Q) \coloneqq \nabla v + (\nabla v)^\top - \pi \Id_3 - \frac{2 \xi}{3} (\Delta Q - Q).
\end{equation}
We need to include the $Q$-terms in order to obtain the exponential stability for the linearization at zero, and to establish the estimates of the surface integrals involving the stress tensor.

As the maximal $\rL^p(\R_+)$-regularity already requires the associated operator to be invertible, we introduce a shift, i.\ e., instead of considering the original linearized problem, we add an artificial term of the form~$\mu \Id$ on the left-hand side of the linearized problem, see \eqref{eq:lin Q-tensor interaction} below.
It does not pose any problems for the local well-posedness, but it is only critical for the global well-posedness.
Thus, for $\mu \ge 0$ as well as~$\hQ \in \rC^1(\overline{\cF_0},\bS_0^3)$, $(f_1,f_2,f_3,f_4) \in \F_T$ and $(v_0,Q_0,\ell_0,\omega_0) \in \rX_\gamma$, recalling $S_\xi$ and $\tS_\xi$ from \eqref{eq:S_xi and tS_xi}, we consider the linearization
\begin{equation}\label{eq:lin Q-tensor interaction}
\left\{
    \begin{aligned}
        \partial_t {v} + (\mu -\Delta ){v} +\nabla \pi - \mdiv S_\xi(\hQ)(\Delta - \Id_3)Q &=f_1, \enspace \mdiv {v} = 0, &&\tin (0,T) \times \cF_0,\\
        {\partial_t {Q}} + \tS_\xi(\hQ) \nabla v + (\mu -(\Delta - \Id_3)) {Q} 
        &= f_2, &&\tin (0,T) \times \cF_0,\\
        \mS({\ell})' + \mu \ell + \int_{\del \cS_0} \rT(v,\pi,Q) n \rd \Gamma &= f_3, &&\tin (0,T),\\ 
        J_0({\omega})' + \mu \omega + \int_{\del \cS_0} y \times \rT(v,\pi,Q) n \rd \Gamma &= f_4, &&\tin (0,T),\\
        {v}={\ell}+{\omega}\times y, \ton (0,T) \times \partial\cS_0, \enspace {v} =0, &\ton (0,T) \times \partial \cO, \enspace
        \partial_n Q = 0, &&\ton (0,T) \times \partial \cF_0,\\
        v(0) = v_0, \enspace Q(0) = Q_0, \enspace \ell(0) &= \ell_0 \tand \omega(0) = \omega_0. 
    \end{aligned}
\right.
\end{equation}
Below, we assert the main result of this section on the maximal $\rL^p$-regularity of \eqref{eq:lin Q-tensor interaction}, up to a shift by~$\mu > 0$ for general $\hQ \in \rC^1(\overline{\cF_0},\bS_0^3)$ and without shift in the case $\hQ = 0$.
Even though we only require this result for $p > 4$, we provide the more general case $p \in (1,\infty) \setminus \{\nicefrac{4}{3}\}$ here.
We do not need a more restrictive condition on $p$, because we consider a linearized problem and take into account~$\hQ \in \rC^1(\overline{\cF_0},\bS_0^3)$.

\begin{thm}\label{thm:lin theory Q-tensor interaction}
Let $p \in (1,\infty) \setminus \{\nicefrac{4}{3}\}$, $\hQ \in \rC^1(\overline{\cF_0},\bS_0^3)$, $T \in (0,\infty]$, $(f_1,f_2,f_3,f_4) \in \F_T$ as well as~$(v_0,Q_0,\ell_0,\omega_0) \in \rX_\gamma$.
\begin{enumerate}[(a)]
    \item There is $\mu_0 \ge 0$ such that for all $\mu > \mu_0$, \eqref{eq:lin Q-tensor interaction} admits a unique solution $(v,Q,\ell,\omega,\pi) \in \tE_T$, and there exists $\Cmr > 0$ such that $\| (v,Q,\ell,\omega,\pi) \|_{\tE_T} \le \Cmr \bigl(\| (f_1,f_2,f_3,f_4) \|_{\F_T} + \| (v_0,Q_0,\ell_0,\omega_0) \|_{\rX_\gamma}\bigr)$.
    The constant $\Cmr > 0$ can be chosen independent of $T$ in the case of homogeneous initial values.
    \item If $\hQ = 0$, then the assertion from~(a) is also valid without shift, that means, there exists a unique solution $(v,Q,\ell,\omega,\pi) \in \tE_T$ to~\eqref{eq:lin Q-tensor interaction} for $\mu = 0$.
    In particular, there exists $\eta_0 > 0$ such that for every $\eta \in [0,\eta_0)$, if $\mre^{\eta t} (f_1,f_2,f_3,f_4) \in \F_\infty$, then there is a unique solution $\mre^{\eta t} (v,Q,\ell,\omega,\pi) \in \tE_\infty$ to \eqref{eq:lin Q-tensor interaction} with $\mu = 0$.
    Besides, there exists a constant $\Cmr > 0$ such that
    \begin{equation*}
        \begin{aligned}
            &\quad \| \mre^{\eta(\cdot)} v \|_{\E_\infty^v} + \| \mre^{\eta(\cdot)} \pi \|_{\E_\infty^{\pi}} + \| \mre^{\eta(\cdot)} Q \|_{\E_\infty^Q} + \| \mre^{\eta(\cdot)} \ell \|_{\rW^{1,p}(0,\infty)} + \| \mre^{\eta(\cdot)} \omega \|_{\rW^{1,p}(0,\infty)}\\
            &\le \Cmr\bigl(\| (v_0,Q_0,\ell_0,\omega_0) \|_{\rX_\gamma} + \| \mre^{\eta(\cdot)}(f_1,f_2,f_3,f_4) \|_{\F_\infty}\bigr).
        \end{aligned}
    \end{equation*}
\end{enumerate}
\end{thm}

A remark on extensions of the Hilbert space case is in order.

\begin{rem}
From \cite[Proposition~7.17]{AH:23}, the maximal $\rL^p$-regularity of the linearized liquid crystal problem on $\rL^q \times \rW^{1,q}$ follows for $q$ close to $2$.
Thus, we expect that an extension of \autoref{thm:lin theory Q-tensor interaction} to $q$ close to~$2$ is possible.
The case of general $q \in (1,\infty)$, however, remains an open problem.
\end{rem}

The rest of this section is dedicated to the proof of \autoref{thm:lin theory Q-tensor interaction}, and the main idea is to reformulate~\eqref{eq:lin Q-tensor interaction} in operator form by invoking an adapted \emph{fluid-structure operator}.
At first, we consider the $Q$-tensor part of the linearized problem~\eqref{eq:lin Q-tensor interaction} with homogeneous boundary conditions.
For simplicity, we omit the shift, while the interface condition is handled by a lifting procedure later on.
We denote by $u = (v,Q)$ the principle variable.
For suitable $f_1$, $f_2$, $v_0$, $Q_0$, and for $S_\xi$ and $\tS_\xi$ from \eqref{eq:S_xi and tS_xi}, the resulting problem is
\begin{equation}\label{eq:lin Q-tensor lc}
\left\{
    \begin{aligned}
        \partial_t {v} -\Delta {v} +\nabla \pi - \mdiv S_\xi(\hQ)(\Delta - \Id_3)Q &=f_1, \enspace \mdiv {v} = 0, &&\tin (0,T) \times \cF_0,\\
        {\partial_t {Q}} + \tS_\xi(\hQ) \nabla v - (\Delta - \Id_3) {Q} 
        &= f_2, &&\tin (0,T) \times \cF_0,\\
        {v} = 0, \enspace \del_n Q 
        &= 0, &&\ton (0,T) \times \partial \cF_0,\\
        v(0) = v_0, \enspace Q(0) 
        &= Q_0, &&\tfor y \in \cF_0. 
    \end{aligned}
\right.
\end{equation}
Following \cite[Section~5]{HHW:24}, we introduce the operator corresponding to \eqref{eq:lin Q-tensor lc}.
First, we recall the Helmholtz projection $\bP \colon \rL^2(\cF_0)^3 \to \rL_{\sigma}^2(\cF_0)$ on $\rL^2$ and invoke the Stokes operator $A_2$ on $\rL_{\sigma}^2(\cF_0)$ defined by
\begin{equation}\label{eq:Stokes op on L2}
    A_2 v \coloneqq \bP \Delta v, \tfor v \in \rD(A_2) \coloneqq \rH^2(\cF_0)^3 \cap \rH_0^1(\cF_0)^3 \cap \rL_{\sigma}^2(\cF_0).
\end{equation}
Besides, the modified Neumann-Laplacian on~$\rH^1(\cF_0,\bS_0^3)$ is defined by
\begin{equation*}
    \DeltaN^1 Q \coloneqq (\Delta - \Id_3) Q, \tfor Q \in \rD(\DeltaN^1) \coloneqq \{Q \in \rH^3(\cF_0,\bS_0^3) : \del_n Q = 0 \ton \del \cF_0\}.
\end{equation*}
We introduce $\rX_0^\lc \coloneqq \rL_{\sigma}^2(\cF_0) \times \rH^1(\cF_0,\bS_0^3)$ as well as $\rX_1^\lc \coloneqq \rD(A_2) \times \rD(\DeltaN^1)$ for the lemma below.
Given~$\hQ \in \rC^1(\overline{\cF_0},\bS_0^3)$, we define the operator matrix $A_\xi(\hQ) \colon \rX_1^\lc \to \rX_0^\lc$ by
\begin{equation}\label{eq:op matrix liquid crystals}
    A_\xi(\hQ) \coloneqq \begin{pmatrix}
        -A_2 & -\bP \mdiv S_\xi(\hQ)(\Delta - \Id_3)\\
        \tS_\xi(\hQ) \nabla & -\DeltaN^1
    \end{pmatrix}.
\end{equation}
We collect auxiliary results for our analysis in the lemma below.
Part~(a) will be an ingredient in the proof of the sectoriality of the operator associated with the linearized interaction problem, while part~(b) will be used for the spectral analysis.

\begin{lem}[{\cite[Section~5]{HHW:24}}]\label{lem:sect of lc op and tested eq lc op}
Let $\hQ \in \rC^1(\overline{\cF_0},\bS_0^3)$ and $u = (v,Q) \in \rX_1^\lc$.
Then
\begin{enumerate}[(a)]
    \item the operator $A_\xi(\hQ)$ from \eqref{eq:op matrix liquid crystals} is sectorial on $\rX_0^\lc$ with spectral angle $\phi_{A_\xi(\hQ)} < \nicefrac{\pi}{2}$, and
    \item $\langle A_\xi(\hQ) u,u\rangle_{\rX_0^\lc} = \| \nabla v \|_{\rL^2(\cF_0)}^2 + \| Q \|_{\rH^2(\cF_0)}^2 + \| \nabla Q \|_{\rL^2(\cF_0)}^2 + \mri \Imp(2\langle \nabla v,S_\xi(\hQ)(-\Delta + \Id_3)Q\rangle_{\rL^2(\cF_0)}$.
\end{enumerate}
\end{lem}

\subsection{Lifting operators and added mass}\label{ssec:lifting ops and added mass}
\

In this subsection, we address the stress tensor part in the surface integrals.
It is sufficient to concentrate on the linearized interaction problem of the fluid part at first and to handle then the $Q$-term in \eqref{eq:lin stress tensor} by means of perturbation theory.
Hence, to shorten notation, we introduce the Cauchy stress tensor
\begin{equation}\label{eq:Cauchy stress}
    \Tilde{\rT}(v,\pi) \coloneqq \nabla v + (\nabla v)^\top - \pi \Id_3.
\end{equation}
A key concept is the \emph{added mass operator}, see e.\ g.\ \cite{Galdi:02} or \cite[Section~4.2]{GGH:13} for further background.
The most important aspects about this approach are the existence of lifting operators for the inhomogeneous interface condition and a representation formula of the pressure appearing in the surface integrals.
This will be of central relevance for the reformulation of \eqref{eq:lin Q-tensor interaction} in operator form.
Let us consider
\begin{equation}\label{eq:stat problem fsi}
    \left\{
    \begin{aligned}
        -\Delta w + \nabla \psi
        &= 0, \enspace \mdiv w = 0, &&\tfor y \in \cF_0,\\
        w
        &= 0, \tfor y \in \del\cO, \enspace
        w
        = \ell + \omega \times y, &&\tfor y \in \del\cS_0,\\
        \int_{\cF_0} \psi \rd y
        &= 0.
    \end{aligned}
    \right.
\end{equation}
The following result on the solvability of \eqref{eq:stat problem fsi} is classical.
We refer for instance to \cite[Lemma~3.6]{MT:18}.
It provides an elementary representation of the solution to the Stokes problem with inhomogeneous boundary conditions at the interface, capturing the motion of the rigid body.

\begin{lem}\label{lem:sol to stat problem fsi}
Consider $(\ell,\omega) \in \C^3 \times \C^3$ and the canonical basis $\{e_i\}$ in $\C^3$.
Then the solution $(w,\psi)$ to \eqref{eq:stat problem fsi} admits the representation $w = \sum_{i=1}^3 \ell_i W_i + \sum_{i=4}^6 \omega_{i-3} W_i$ and $\psi = \sum_{i=1}^3 \ell_i \Psi_i + \sum_{i=4}^6 \omega_{i-3} \Psi_i$, where $(W_i,\Psi_i)$, $i=1,\dots,6$, in turn satisfy $\int_{\cF_0} \Psi_i \rd y = 0$ and
\begin{equation*}
    \left\{
    \begin{aligned}
        -\Delta W_i + \nabla \Psi
        &= 0, \enspace \mdiv W_i = 0, &&\tfor y \in \cF_0,\\
        W_i
        &= 0, &&\tfor y \in \del\cO,\\
        W_i
        = e_i, \tfor i=1,2,3 \tand W_i 
        &= e_{i-3} \times y, \tfor i=4,5,6, &&\tfor y \in \del\cS_0,
    \end{aligned}
    \right.
\end{equation*}
$\B_{i,j} \coloneqq \int_{\cF_0} D W_i : D W_j \rd y$ is invertible, where $D W_i$ represents the gradient of $W_i$, and for~$\Tilde{\rT}$ from \eqref{eq:Cauchy stress}, we have $\begin{pmatrix}
        \int_{\del\cS_0} \Tilde{\rT}(w,\psi) n \rd \Gamma\\
        \int_{\del\cS_0} y \times \Tilde{\rT}(w,\psi) n \rd \Gamma
    \end{pmatrix} = \B \binom{\ell}{\omega}$.
\end{lem}

\autoref{lem:sol to stat problem fsi} allows us to define lifting operators $D_\fl \in \cL(\C^6,\rH^2(\cF_0))$ and $D_\pr \in \cL(\C^6,\rH^1(\cF_0) \cap \rL_0^2(\cF_0))$.
In fact, given $(\ell,\omega) \in \C^3 \times \C^3$, for the solution $(w,\psi)$ to \eqref{eq:stat problem fsi} emerging from \autoref{lem:sol to stat problem fsi}, we set
\begin{equation}\label{eq:lifting ops D}
    D_\fl(\ell,\omega) \coloneqq w \tand D_\pr(\ell,\omega) \coloneqq \psi.
\end{equation}
We also invoke the operator $N \in \cL\bigl(\rH^{\nicefrac{1}{2}}(\del \cF_0) \cap \rL_0^2(\del \cF_0),\rH^2(\cF_0) \cap \rL_0^2(\cF_0)\bigr)$, with $N h \coloneqq \varphi$, for~$\varphi$ solving
\begin{equation}\label{eq:Neumann lifting probl}
    \Delta \varphi = 0, \tin \cF_0, \enspace \del_n \varphi = h, \ton \del \cF_0.
\end{equation}
For the Neumann operator $N$ and $h \in \rH^{\nicefrac{1}{2}}(\del \cS_0) \cap \rL_0^2(\del \cS_0)$, we define the operator $N_{\cS_0}$ by
\begin{equation}\label{eq:lifting op NS}
    N_{\cS_0} h \coloneqq N(\chi_{\del \cS_0}h).
\end{equation}

The above lifting operators allow us to rewrite the resolvent problem linked to $v$ in terms of $\bP v$ and $(\Id_3 - \bP)v$ for the Helmholtz projection $\bP$ on $\rL^2(\cF_0)^3$.
Besides, we obtain a formula for the pressure in terms of $\bP v$ as well as the rigid body velocities.
This will be important for handling the pressure term appearing in the surface integrals of the stress tensor, and it paves the way for a reformulation in operator form.
This splitting idea for fluid-structure interaction problems was e.\ g.\ used by Raymond \cite{Ray:07}.
For a proof of the next lemma, we refer to \cite[Proposition~3.7]{MT:18}.

\begin{lem}
Assume $(f_1,\ell,\omega) \in \rL_{\sigma}^2(\cF_0) \times \C^3 \times \C^3$, and recall the Stokes operator on $\rL_{\sigma}^2(\cF_0)$ from~\eqref{eq:Stokes op on L2}, the Dirichlet lifting operator  $D_\fl$ from~\eqref{eq:lifting ops D}, the lifting operator $N$ associated to the Neumann problem~\eqref{eq:Neumann lifting probl} as well as $N_{\cS_0}$ from \eqref{eq:lifting op NS}.
Then $(v,\pi) \in \rH^2(\cF_0)^3 \times \rH^1(\cF_0) \cap \rL_0^2(\cF_0)$ satisfies
\begin{equation*}
    \left\{
    \begin{aligned}
        \lambda v - \Delta v + \nabla \pi
        &= f_1, \enspace \mdiv v = 0, &&\tfor y \in \cF_0,\\
        v
        = 0, \tfor y \in \del\cO,\enspace
        v
        &= \ell + \omega \times y, &&\tfor y \in \del \cS_0,
    \end{aligned} \quad \text{if and only if}
    \right.
\end{equation*}
\begin{equation}\label{eq:reform resolvent probl fluid part}
    \left\{
    \begin{aligned}
        \lambda \bP v - A_2 \bP v + A_2 \bP D_\fl(\ell,\omega)
        &= \bP f_1,\\
        (\Id_3 - \bP)v
        &= (\Id_3 - \bP)D_\fl(\ell,\omega),\\
        \pi
        &= N(\Delta \bP v \cdot n) - \lambda N_{\cS_0}((\ell + \omega \times y) \cdot n).
    \end{aligned}
    \right.
\end{equation}
\end{lem}

For the reformulation of \eqref{eq:lin Q-tensor interaction} in operator form, we invoke the corresponding resolvent problem.
Indeed, for $\lambda \in \C$ and $(f_1,f_2,f_3,f_4) \in \rL^2(\cF_0)^3 \times \rH^1(\cF_0,\bS_0^3) \times \C^3 \times \C^3$, it reads as
\begin{equation}\label{eq:resolvent probl lin Q-tensor interaction}
    \left\{
    \begin{aligned}
        \lambda v - \Delta v + \nabla \pi - \mdiv S_\xi(\hQ)(\Delta - \Id_3)Q 
        &= f_1, \enspace \mdiv v = 0, &&\tin \cF_0,\\
        \lambda Q + \tS_\xi(\hQ) \nabla v - (\Delta - \Id)Q
        &= f_2, &&\tin \cF_0,\\
        \lambda \mS \ell + \int_{\del \cS_0} \rT(v,\pi,Q) n \rd \Gamma
        &= f_3,\\
        \lambda J_0 \omega + \int_{\del \cS_0} y \times \rT(v,\pi,Q) n \rd \Gamma 
        &= f_4,\\
        v = 0, \ton \del \cO, \enspace v 
        &= \ell + \omega \times y, \ton \del \cS_0, \enspace
        \del_n Q = 0, &&\ton \del \cF_0.
    \end{aligned}
    \right.
\end{equation}

Thanks to \eqref{eq:reform resolvent probl fluid part}, we rewrite \eqref{eq:resolvent probl lin Q-tensor interaction}$_3$ as $\lambda K \binom{\ell}{\omega} = C_1 \bP v + C_2 Q + C_3 \binom{\ell}{\omega} + \binom{f_3}{f_4}$, where
\begin{equation}\label{eq:matrix K and ops C_1, C_2 and C_3}
    \begin{aligned}
        K
        &= \begin{pmatrix}
            \mS \Id_3 & 0\\
            0 & J_0
        \end{pmatrix} + M, \twith
        M \binom{\ell}{\omega} 
        = \begin{pmatrix}
            \int_{\del \cS_0} N_{\cS_0}((\ell + \omega \times y) \cdot n)n \rd \Gamma\\
            \int_{\del \cS_0} y \times N_{\cS_0}((\ell + \omega \times y) \cdot n)n \rd \Gamma
        \end{pmatrix},\\
        C_1 \bP v
        &= -2\begin{pmatrix}
            \int_{\del \cS_0} \D(\bP v) n \rd \Gamma + \int_{\del \cS_0} N(\Delta \bP v \cdot n)n \rd \Gamma\\
            \int_{\del \cS_0} y \times \D(\bP v) n \rd \Gamma + \int_{\del \cS_0} y \times N(\Delta \bP v \cdot n)n \rd \Gamma
        \end{pmatrix},\\
        C_2 Q 
        &= \frac{2 \xi}{3} \begin{pmatrix}
            \int_{\del \cS_0} (\Delta Q - Q) n \rd \Gamma\\
            \int_{\del \cS_0} y \times (\Delta Q - Q) n \rd \Gamma
        \end{pmatrix}, \enspace
        C_3 \binom{\ell}{\omega}
        = -2\begin{pmatrix}
            \int_{\del \cS_0} \D((\Id_3 - \bP)D_\fl(\ell,\omega)) n \rd \Gamma\\
            \int_{\del \cS_0} y \times \D((\Id_3 - \bP)D_\fl(\ell,\omega)) n \rd \Gamma
        \end{pmatrix}.
    \end{aligned}
\end{equation}
Let us note that $M$ above is generally referred to as the \emph{added mass operator}.
From, \cite[Lemma~4.6]{Galdi:02}, see also \cite[Lemma~4.3]{GGH:13}, we recall that the matrix $K$ from \eqref{eq:matrix K and ops C_1, C_2 and C_3} is invertible.
The purpose of the added mass operator is that allows for a reformulation of the surface integrals without the pressure.
More precisely, thanks to the invertibility of $K$, it is possible to reformulate the equations for the rigid body velocities including the surface integrals as a problem of $\bP v$, $Q$ and a lifted problem involving the rigid body velocities $\ell$ and $\omega$ themselves.

\subsection{The fluid-structure operator}
\ 

Thanks to the operators introduced in the previous subsection, we are now in the position to define the \emph{fluid-structure operator} corresponding to the linearized $Q$-tensor interaction problem \eqref{eq:lin Q-tensor interaction}.
For $\hQ \in \rC^1(\overline{\cF_0},\bS_0^3)$, the operator $\Afs(\hQ) \colon \rD(\Afs) \subset \rX \to \rX$ is defined on $\rX \coloneqq \rL_{\sigma}^2(\cF_0) \times \rH^1(\cF_0,\bS_0^3) \times \C^3 \times \C^3$, and with domain $\rD(\Afs) \coloneqq \{(\bP v,Q,\ell,\omega) \in \rX : \bP v - \bP D_\fl(\ell,\omega) \in \rD(A_2)\}$ it takes the shape
\begin{equation}\label{eq:fluid-structure op}
    \begin{aligned}
        \Afs(\hQ) 
        &\coloneqq \begin{pmatrix}
        A_2 & \bP \mdiv S_\xi(\hQ)(\Delta - \Id_3) & -A_2 \bP D_\fl\\
        -\tS_\xi(\hQ)\nabla & \DeltaN^1 & 0\\
        K^{-1} C_1 & K^{-1} C_2 & K^{-1} C_3
    \end{pmatrix}.
    \end{aligned}
\end{equation}

The next proposition reveals that the resolvent problem \eqref{eq:resolvent probl lin Q-tensor interaction} can be rephrased equivalently in terms of the fluid-structure operator $\Afs$ from \eqref{eq:fluid-structure op} and the formulae for $(\Id_3 - \bP)v$ and $p$ from \eqref{eq:reform resolvent probl fluid part}.
It follows from a combination of the liquid crystals part and the structure part as seen in \autoref{ssec:lifting ops and added mass}.

\begin{prop}\label{prop:equiv form resolvent probl}
Let $(f_1,f_2,f_3,f_4) \in \rX$ and $\hQ \in \rC^1(\overline{\cF_0},\bS_0^3)$.
Then $(v,Q,\ell,\omega,\pi)$ with $v \in \rH^2(\cF_0)^3$, $Q \in \rH^3(\cF_0,\bS_0^3)$, $(\ell,\omega) \in \C^3 \times \C^3$ and $\pi \in \rH^1(\cF_0) \cap \rL_0^2(\cF_0)$ is a solution to \eqref{eq:resolvent probl lin Q-tensor interaction} if and only if
\begin{equation*}
    \left\{
    \begin{aligned}
        (\lambda \Id - \Afs(\hQ)) \begin{pmatrix}
            \bP v\\ Q\\ \ell\\ \omega
        \end{pmatrix} 
        &= \begin{pmatrix}
            \bP f_1\\ f_2\\ K^{-1} f_3\\ K^{-1} f_4
        \end{pmatrix}, \enspace
        (\Id_3 - \bP)v
        = (\Id_3 - \bP)D_\fl(\ell,\omega) \tand\\ 
        \pi
        &= N(\Delta \bP v \cdot n) - \lambda N_{\cS_0}((\ell + \omega \times y) \cdot n).
    \end{aligned}
    \right.
\end{equation*}
\end{prop}

After defining the fluid-structure operator, we verify its maximal $\rL^p$-regularity.

\begin{prop}\label{prop:sect and max reg of fluid-structure op}
Let $\hQ \in \rC^1(\overline{\cF_0},\bS_0^3)$, and recall the fluid-structure operator $\Afs(\hQ)$ from \eqref{eq:fluid-structure op}.
Then there exists $\mu \ge 0$ such that $-\Afs(\hQ) + \mu$ is sectorial on $\rX$ with spectral angle $\phi_{-\Afs(\hQ) + \mu} < \nicefrac{\pi}{2}$.
In particular, $-\Afs(\hQ) + \mu$ has maximal $\rL^p$-regularity on $\rX$.
\end{prop}

\begin{proof}
First, we decompose the fluid-structure operator $\Afs(\hQ)$ into $\Afs(\hQ) \coloneqq \AfsI(\hQ) + \Bfs$, where
\begin{equation*}
    \AfsI(\hQ) \coloneqq \begin{pmatrix}
        -A_\xi(\hQ) & 0\\
        0 & 0
    \end{pmatrix} \tand \Bfs \coloneqq \begin{pmatrix}
        0 & B_1\\
        B_2 & B_3
    \end{pmatrix} \coloneqq \begin{pmatrix}
        0 & 0 & -A_2 \bP D_\fl\\
        0 & 0 & 0\\
        K^{-1} C_1 & K^{-1} C_2 & K^{-1} C_3
    \end{pmatrix}
\end{equation*}
for $A_\xi(\hQ)$ from \eqref{eq:op matrix liquid crystals} and the Stokes operator $A_2$ on $\rL_{\sigma}^2(\cF_0)$.
Besides, the Dirichlet lifting operator $D_\fl$ was made precise in \eqref{eq:lifting ops D}, and the operators $C_1$, $C_2$ and $C_3$ and the matrix $K$ are defined in \eqref{eq:matrix K and ops C_1, C_2 and C_3}.

The sectoriality of $-\AfsI(\hQ)$ on $\rX$ with spectral angle less than $\nicefrac{\pi}{2}$ carries over from \autoref{lem:sect of lc op and tested eq lc op}(a).
For the sectoriality of $-\Afs(\hQ)$ up to a shift, we thus verify the relative $-\AfsI(\hQ)$-boundedness of $\Bfs$ with arbitrarily small $-\AfsI(\hQ)$-bound.
To this end, consider $(v,Q,\ell,\omega) \in \rD(\Afs) = \rD(\AfsI)$.
Thanks to mapping properties of the Stokes operator as well as $D_\fl \in \cL(\C^6,\rH^2(\cF_0))$, we deduce that
\begin{equation*}
    \left\| B_1(\ell,\omega) \right\|_{\rL_{\sigma}^2(\cF_0) \times \rH^1(\cF_0)} = \| -A_2 \bP D_\fl(\ell,\omega) \|_{\rL_{\sigma}^2(\cF_0)} \le C \cdot \| D_\fl(\ell,\omega) \|_{\rH^2(\cF_0)} \le C \cdot \| (v,Q,\ell,\omega) \|_{\rX}.
\end{equation*}
This shows that $B_1 \in \cL(\C^3 \times \C^3,\rL_{\sigma}^2(\cF_0) \times \rH^1(\cF_0,\bS_0^3))$, so $B_1$ is a bounded perturbation.
Moreover, we recall from the proof of \cite[Theorem~3.11]{MT:18} that the operators $C_1$ and $C_3$ are relatively bounded with relative bound zero.
For $C_2$, we invoke classical properties of the trace, interpolation and Young's inequality to deduce that for any $\delta > 0$, and with $\eps > 0$ small and $\theta \in (0,1)$, it holds that
\begin{equation*}
    |C_2 Q| 
    \le C \cdot \| Q \|_{\rH^{\nicefrac{5}{2} + \eps}(\cF_0)} 
    \le C \cdot \| Q \|_{\rH^1(\cF_0)}^\theta \cdot \| Q \|_{\rH^3(\cF_0)}^{1-\theta} \le \delta \cdot \| Q \|_{\rH^3(\cF_0)} + C(\delta) \cdot \| Q \|_{\rH^1(\cF_0)}.
\end{equation*}
Thus, $C_2$ is a relatively bounded perturbation with relative bound zero.
In total, $\Bfs$ is a relatively bounded perturbation of $\AfsI(\hQ)$ with $\AfsI(\hQ)$-bound zero.
Perturbation theory for sectorial operators, see e.\ g.\ \cite[Cor.~3.1.6]{PS:16}, then yields the existence of $\mu \ge 0$ such that $-\Afs(\hQ) + \mu = -\AfsI(\hQ) - \Bfs + \mu$ is sectorial on $\rX$ with spectral angle $\phi_{-\Afs(\hQ) + \mu} < \nicefrac{\pi}{2}$.
The second part of the assertion follows by the known equivalence of sectoriality with angle $< \nicefrac{\pi}{2}$ and maximal $\rL^p$-regularity on Hilbert spaces.
\end{proof}

\subsection{Proof of \autoref{thm:lin theory Q-tensor interaction}}\label{ssec:proof of main lin result}
\ 

\autoref{thm:lin theory Q-tensor interaction}(a) directly follows from \autoref{prop:sect and max reg of fluid-structure op}, while~(b) relies on a refined spectral analysis.

\begin{proof}[Proof of \autoref{thm:lin theory Q-tensor interaction}(b)]

For simplicity, we write $\Afs \coloneqq \Afs(0)$ in the sequel.
It can be shown that the shift $\mu \ge 0$ in \autoref{prop:sect and max reg of fluid-structure op} can be be chosen equal to $s(\Afs) \coloneqq \sup\{\Rep \lambda : \lambda \in \sigma(\Afs)\}$ of $\Afs$, see for instance \cite[Cor.~3.5.3]{PS:16}.
Hence, to remove the shift in \autoref{prop:sect and max reg of fluid-structure op}, and thus to obtain the maximal $\rL^p(\R_+)$-regularity of $\Afs$ on $\rX$, we need to prove that the spectral bound satisfies $s(\Afs) < 0$.

We start by noting that $\Afs$ has a compact resolvent by compactness of the embedding $\rD(\Afs) \hookrightarrow \rX$ thanks to the Rellich-Kondrachov theorem.
As a result, $\sigma(\Afs) = \sigma_\mathrm{p}(\Afs)$, i.\ e., the spectrum of $\Afs$ only consists of eigenvalues.
In view of \autoref{prop:equiv form resolvent probl}, we may thus investigate \eqref{eq:resolvent probl lin Q-tensor interaction} with right-hand side zero.
Indeed, if for $\lambda \in \C$, the only solution to \eqref{eq:resolvent probl lin Q-tensor interaction} with $(f_1,f_2,f_3,f_4) = (0,0,0,0) \in \rX$ is zero, then \autoref{prop:equiv form resolvent probl} implies that $\lambda \notin \sigma_p(\Afs)$ and hence $\lambda \in \rho(\Afs)$ by the preceding argument.

Let now $\lambda \in \C$ and $(v,Q,\ell,\omega,\pi) \in \rH^2(\cF_0)^3 \times \rH^3(\cF_0,\bS_0^3) \times \C^3 \times \C^3 \times \rH^1(\cF_0) \cap \rL_0^2(\cF_0)$ satisfy \eqref{eq:resolvent probl lin Q-tensor interaction} with $f_i = 0$ for $i=1,\dots,4$.
Next, we test the equation satisfied by $v$ in \eqref{eq:resolvent probl lin Q-tensor interaction} by $\ovv$, the equation for~$Q$ by~$\Delta \ovQ$, the equation for $\ell$ by $\ovell$ and the equation for $\omega$ by $\ovomega$.
Integrating by parts, invoking \autoref{lem:sect of lc op and tested eq lc op}(b) to handle the liquid crystals part, observing that $S_\xi(0)(-\Delta + \Id_3)Q = \nicefrac{2 \xi}{3}(\Delta Q - Q)$, making use of the interface condition from \eqref{eq:bdry cond}, and recalling $\rT(v,\pi,Q)$ from \eqref{eq:lin stress tensor}, we derive that
\begin{equation*}
    \begin{aligned}
        0
        &= \lambda\left(\| v \|_{\rL^2(\cF_0)}^2 + \| \nabla Q \|_{\rL^2(\cF_0)}^2 + \mS |\ell|^2 + J_0 \omega \cdot \ovomega\right)\\
        &\quad + 2 \| \D(v) \|_{\rL^2(\cF_0)}^2 + \| Q \|_{\rH^2(\cF_0)}^2 + \| \nabla Q \|_{\rL^2(\cF_0)}^2 + \mri \Imp(2\langle \nabla v,S_\xi(0)(-\Delta + \Id_3)Q\rangle_{\rL^2(\cF_0)}.
    \end{aligned}
\end{equation*}
When considering the real part in the above equation, we obtain
\begin{equation*}
    \Rep \lambda\left(\| v \|_{\rL^2(\cF_0)}^2 + \| \nabla Q \|_{\rL^2(\cF_0)}^2 + \mS |\ell|^2\right) + \Rep(\lambda J_0 \omega \cdot \ovomega) + 2 \| \D(v) \|_{\rL^2(\cF_0)}^2 + \| Q \|_{\rH^2(\cF_0)}^2 + \| \nabla Q \|_{\rL^2(\cF_0)}^2 = 0.
\end{equation*}
It follows that $\Rep \lambda \le 0$.
Moreover, if $\Rep \lambda = 0$, then $\D(v) = 0$ and $Q = 0$ on $\cF_0$, so $v$ is especially constant.
Due to $v = 0$ on $\del \cO$, we infer that $v = 0$ on $\cF_0$.
The interface condition $v = \ell + \omega \times y$ on $\del \cS_0$ then yields that $\ell = \omega = 0$.
From \eqref{eq:resolvent probl lin Q-tensor interaction}, it follows that $\nabla \pi = 0$, and $\pi \in \rL_0^2(\cF_0)$ rules out the constants, so $\pi = 0$ and thus $(v,Q,\ell,\omega,\pi) = (0,0,0,0,0)$.
The above arguments imply that $s(\Afs) < 0$, resulting in the sectoriality and then also the maximal $\rL^p(\R_+)$-regularity of $-\Afs$ on $\rX$ even without shift.

In a similar manner as for \autoref{prop:equiv form resolvent probl}, one can show that \eqref{eq:lin Q-tensor interaction} admits an equivalent reformulation in terms of the fluid-structure operator $\Afs$ from \eqref{eq:fluid-structure op}.
Therefore, the first part of the assertion of \autoref{thm:lin theory Q-tensor interaction}(b) already follows, and the maximal regularity estimate is also implied in the case $\mu = 0$.

By the sectoriality of $-\Afs$ with spectral angle $\phi_{-\Afs} < \nicefrac{\pi}{2}$, the operator $\Afs$ generates a bounded analytic semigroup $\left(\mre^{t \Afs}\right)_{t \ge 0}$, see e.\ g.\ \cite[Thm.~3.3.2]{PS:16}.
Thanks to $s(\Afs) < 0$, this semigroup is even of negative exponential type.
Thus, we may choose $\eta_0 = -s(\Afs) > 0$ and consider the operator $\Afs + \eta$ for~$\eta \in [0,\eta_0)$, still generating a bounded analytic semigroup of negative exponential type thanks to the fact that $s(\Afs + \eta) < 0$, to deduce the estimate with exponential weights. 
\end{proof}

\section{Local strong well-posedness for large data}\label{sec:proof local}

We use $z = (v,Q,\ell,\omega)$ again to shorten notation.
First, we derive a reference solution to capture the initial data and to reduce the study to a problem with homogeneous initial values.
The result below is a direct consequence of \autoref{thm:lin theory Q-tensor interaction}(a), because $z_0 \in \rX_\gamma$ implies $Q_0 \in \rC^1(\overline{\cF_0},\bS_0^3)$ for $p > 4$.

\begin{prop}\label{prop:ref sol}
Let $p > 4$ and $z_0 = (v_0,Q_0,\ell_0,\omega_0) \in \rX_\gamma$, where $\rX_\gamma$ has been introduced in \eqref{eq:X_gamma}, and consider $(f_1,f_2,f_3,f_4) = (0,0,0,0)$.
Then for $\mu > \mu_0$, where $\mu_0 \ge 0$ is as in \autoref{thm:lin theory Q-tensor interaction}(a), there is a unique solution $(z^*,\pi^*) \in \tE_\infty$ to \eqref{eq:lin Q-tensor interaction}.
In particular, we have $\| z^* - z_0 \|_{\BUC([0,T];\rX_\gamma)} \to 0$ as $T \to 0$.
\end{prop}

For the reference solution $(z^*,\pi^*)$ from \autoref{prop:ref sol} and a solution $(z,\pi)$ to \eqref{eq:cv1}--\eqref{eq:cv3}, we define~$ \hv \coloneqq v - v^*$, $\hQ \coloneqq Q - Q^*$, $\hell \coloneqq \ell - \ell^*$, $ \homega \coloneqq \omega - \omega^*$ and $\hp \coloneqq \pi - \pi^*$.
For $\mu > \mu_0$, with $\mu_0 \ge 0$ from \autoref{thm:lin theory Q-tensor interaction}(a), the emerging $(\hz,\hp) = (\hv,\hQ,\hell,\homega,\hp)$ solves
\begin{equation}\label{eq:linearized syst with hom initial data}
\left\{
    \begin{aligned}
        \partial_t {\hv} + (\mu-\Delta) {\hv} +\nabla {\hp} - \mdiv S_\xi(Q_0)(\Delta - \Id_3)\hQ &=F_1(\hz,\hp), \enspace \mdiv {\hv} = 0, &&\tin (0,T) \times \cF_0,\\
        {\partial_t {\hQ}} + \tS_\xi(Q_0) \nabla \hv + (\mu - (\Delta - \Id_3)) {\hQ} 
        &= F_2(\hz), &&\tin (0,T) \times \cF_0,\\
        \mS({\hell})' + \mu \hell + \int_{\del \cS_0} \rT(\hv,\hp,\hQ) n \rd \Gamma &= F_3(\hz,\hp), &&\tin (0,T),\\ 
        J_0({\homega})' + \mu \homega + \int_{\del \cS_0} y \times \rT(\hv,\hp,\hQ) n \rd \Gamma &= F_4(\hz,\hp), &&\tin (0,T),\\
        {\hv} ={\hell}+{\homega}\times y, \ton (0,T) \times \partial\cS_0, \enspace
        {\hv} =0, &\ton (0,T) \times \partial \cO, \enspace \partial_n \hQ = 0, &&\ton (0,T) \times \partial \cF_0,\\
        \hv(0) = 0, \enspace \hQ(0) = 0, \enspace \hell(0) &= 0 \tand \homega(0) = 0. 
    \end{aligned}
\right.
\end{equation}
As the linear part of $\tau$ is already captured on the left-hand side of \eqref{eq:linearized syst with hom initial data}, we define the remaining part $\tau_{\rr}$ by $\tau_{\rr}(Q) \coloneqq \tau(Q) + \frac{2 \xi}{3}(\Delta Q - Q)$.
For the transformed terms as made precise in \autoref{sec:change of var}, and with 
\begin{equation*}
    \cT(v,\pi,Q) \coloneqq \bO(t)^\top \rT(\bO(t)v(t,y),\pi(t,y),Q(t,y))\bO(t)
\end{equation*}
for the rotation matrix $\bO(t)$ related to \eqref{eq:IVP X0}, $\rT(v,\pi,Q)$ as in \eqref{eq:lin stress tensor} and $(z,\pi) = (\hz + z^*,\hp + \pi^*)$, we have
\begin{equation}\label{eq:rhs F_1 to F_4}
    \begin{aligned}
        F_1(\hz,\hp)
        &\coloneqq (\cL_1 - \Delta)v + \mu v - \cM v - \cN(v) - (\cG - \nabla)\pi + \cB_{\tau_\rl}(Q)\\
        &\quad + \bigl(\cB_{\tau_\rh}(Q) + \cB_{\sigma}(Q) - \mdiv S_\xi(Q_0)((\Delta - \Id_3)Q)\bigr),\\
        F_2(\hz)
        &\coloneqq (\cL_2 - \Delta)Q + \mu Q - (\del_t Y \cdot\nabla)Q - (v \cdot \nabla)Q + \bigl(\cS(v,Q) + \tS_\xi(Q_0)\nabla v\bigr)\\
        &\quad + Q^2 - \tr\bigl(Q^2\bigr)\nicefrac{\Id_3}{3} - \tr\bigl(Q^2\bigr)Q,\\
        F_3(\hz,\hp)
        &\coloneqq -\mS(\omega \times \ell) + \mu \ell + \int_{\del \cS_0}\bigl((\rT - \cT)(v,\pi,Q) - \bO(t)^\top(\tau_{\rr}(Q) + \sigma(Q))\bO(t)\bigr) n \rd \Gamma \tand\\
        F_4(\hz,\hp)
        &\coloneqq \omega \times (J_0 \omega) + \mu \omega + \int_{\del \cS_0} y \times \bigl((\rT - \cT)(v,\pi,Q) - \bO(t)^\top(\tau_{\rr}(Q) + \sigma(Q))\bO(t)\bigr) n \rd \Gamma.
    \end{aligned}
\end{equation}
We now fix $T_0$, $R_0 > 0$ and consider $T \in (0,T_0]$ and $R \in (0,R_0]$.
Recalling $\E_T$, $\E_T^{\pi}$ and $\tE_T$ from \eqref{eq:max reg space} and \eqref{eq:max reg space with pressure}, we introduce the set $\cK_T^R$ together with the solution map $\Phi_T^R$ defined by
\begin{equation*}
    \cK_T^R \coloneqq \left\{(\tz,\tp) \in \prescript{}{0}{\E_T} \times \E_T^{\pi} : \| (\tz,\tp) \|_{\tE_T} \le R\right\} \tand \Phi_T^R \colon \cK_T^R \to \prescript{}{0}{\E_T} \times \E_T^{\pi}, \twhere \Phi_T^R(\tz,\tp) \coloneqq (\hz,\hp),
\end{equation*}
and $(\hz,\hp)$ is the solution to \eqref{eq:linearized syst with hom initial data} with $F_1(\tz,\tp)$, $F_2(\tz)$, $F_3(\tz,\tp)$ and $F_4(\tz,\tp)$ for $(\tz,\tp) \in \cK_T^R$.
Thanks to the maximal $\rL^p$-regularity from \autoref{thm:lin theory Q-tensor interaction}(a), $\Phi_T^R$ is well-defined if $(F_1,F_2,F_3,F_4) \in \F_T$, with $\F_T$ as made precise in \eqref{eq:data space}.
By construction, the existence and uniqueness of a solution to \eqref{eq:cv1}--\eqref{eq:cv3} is equivalent with $\Phi_T^R$ possessing a unique fixed point upon adding the reference solution $(z^*,\pi^*)$ from \autoref{prop:ref sol}.

\

Next, we address the nonlinear estimates.
For $T \in (0,T_0]$, $R \in (0,R_0]$, let $(\tz,\tp) \in \cK_T^R$.
With~$(z^*,\pi^*)$ from \autoref{prop:ref sol}, we set $(z,\pi) \coloneqq (\tv + v^*,\tQ + Q^*,\tell + \ell^*,\tomega + \omega^*,\tp + \pi^*)$.
As observed in \cite[Remark~5.2]{BBHR:23}, for $C_T^* \coloneqq \| z^* \|_{\E_T} + \| \pi^* \|_{\E_T^{\pi}}$, it is valid that $C_T^* \to 0$ as~$T \to 0$.
By construction, it also follows that
\begin{equation}\label{eq:est of norm of (z,p)}
    \| (z,\pi) \|_{\tE_T} \le R + C_T^*.
\end{equation}
We also set $C_0 \coloneqq \| z_0 \|_{\rX_\gamma} = \| (v_0,Q_0,\ell_0,\omega_0) \|_{\rX_\gamma}$.
At this stage, we invoke the following well-known lemma, allowing for an estimate of $\sup_{t \in [0,T]} \| z(t) \|_{\rX_\gamma}$ by the norm of the initial data and the solution, where the constant is independent of $T$.

\begin{lem}\label{lem:est by init data and norm of sol with appl to ref sol}
For $z \in \E_T$, there is $T$-independent $C > 0$ with $\sup_{t \in [0,T]} \| z(t) \|_{\rX_\gamma} \le C(\| z(0) \|_{\rX_\gamma} + \| z \|_{\E_T})$.
In particular, for $z^*$ from \autoref{prop:ref sol}, we have $\| z^* \|_{\BUC([0,T];\rX_\gamma)} \le C(C_0 + C_T^*)$.
\end{lem}

In the sequel, we derive useful embeddings.
The assertion of \autoref{lem:ests of max reg space}(a) is a consequence of the mixed derivative theorem, see e.\ g.\ \cite[Corollary~4.5.10]{PS:16}, and Sobolev embeddings.
For part~(b), we use \cite[Theorem~III.4.10.2]{Ama:95} joint with interpolation theory and embeddings for the Besov spaces, see for example \cite[Theorem~4.6.1]{Tri:78}. 

\begin{lem}\label{lem:ests of max reg space}
For $p > 4$ as well as $0 < T \le \infty$, recall $\E_T$ and $\rX_\gamma$ from \eqref{eq:max reg space} and \eqref{eq:X_gamma}, respectively.
\begin{enumerate}[(a)]
    \item For every $\theta \in (0,1)$, we find that
    \begin{equation*}
        \E_T  
        \hookrightarrow \rH^{\theta,p}\bigl(0,T;\rH^{2(1-\theta)}(\cF_0)^3 \times \rH^{2(1-\theta)+1}(\cF_0,\bS_0^3) \times \R^3 \times \R^3\bigr) \tand \E_T^Q \hookrightarrow \rL^\infty\bigl(0,T;\rH^2(\cF_0,\bS_0^3)\bigr).
    \end{equation*}
    \item It holds that $ \E_T \hookrightarrow \BUC([0,T];\rX_\gamma)  \hookrightarrow \BUC\bigl([0,T];\rB_{2p}^{2-\nicefrac{2}{p}}(\cF_0)^3 \times \rB_{2p}^{3-\nicefrac{2}{p}}(\cF_0,\bS_0^3) \times \R^6\bigr)$, implying $\E_T \hookrightarrow \BUC\bigl([0,T];\rC(\overline{\cF_0})^3 \times \rC^1(\overline{\cF_0},\bS_0^3) \times \R^6\bigr)$.
    For $T = \infty$, $\BUC([0,T])$ is replaced by $\rL^\infty(0,\infty)$.
\end{enumerate}
The embedding constants can be chosen $T$-independent if the functions have homogeneous initial values.
\end{lem}

Given the transformed body velocities, it is possible to successively deduce the matrix $\bO$, the original body velocities~$h'$ and $\Omega$ as well as the diffeomorphisms $X$ and $Y$ as described in the context of the transformation to the fixed domain in \autoref{sec:change of var}.
For more details, see \cite[Remark~6.1]{BBHR:23}.

In the sequel, we discuss useful estimates of $X$, $Y$ and related terms.
For $(\ell_i,\omega_i) \in \rW^{1,p}(0,T)^6$, the subscript $i$ is used to denote objects associated to~$(\ell_i,\omega_i)$.
Besides, we employ $\| \cdot \|_\infty$ and $\| \cdot \|_{\infty,\infty}$ for the $\rL^\infty$-norm in time and time and space, respectively.
The assertions of \autoref{lem:ests of trafo and related terms}(a) and~(b) follow upon modifying the arguments in \cite[Section~6.1]{GGH:13}.
For the estimates of the second derivatives of $X$ and~$Y$, note that the diffeomorphisms associated to $\ell = \omega = 0$ are the identity, so their second derivatives vanish.
\autoref{lem:ests of max reg space}(b) and \autoref{lem:est by init data and norm of sol with appl to ref sol} then imply $\| \ell_i \|_\infty + \| \omega_i \|_\infty \le C (R + C_0 + C_T^*)$.
For a proof of \autoref{lem:ests of trafo and related terms}(c), we refer to \cite[Lemma~6.5]{GGH:13}, while \autoref{lem:ests of trafo and related terms}(d) can be deduced from~(c).

\begin{lem}\label{lem:ests of trafo and related terms}
Consider $(\ell,\omega)$, $(\ell_1,\omega_1)$, $(\ell_2,\omega_2) \in \rW^{1,p}(0,T)^6$ as well as the associated diffeomorphisms $X$, $X_1$, $X_2$ and $Y$, $Y_1$ and $Y_2$.
\begin{enumerate}[(a)]
    \item For $i=1,2$, we have $X_i$, $Y_i \in \rC^1\left(0,T;\rC^\infty(\R^3)^3\right)$, with $\| \partial^{\alpha} X_i \|_{\infty,\infty} + \| \partial^{\alpha} Y_i \|_{\infty,\infty} \le C$ and
    \begin{equation*}
        \| \partial^{\beta}(X_1 - X_2) \|_{\infty,\infty} + \| \partial^{\beta}(Y_1 - Y_2) \|_{\infty,\infty} \le CT (\| \ell_1 - \ell_2 \|_{\infty} + \| \omega_1 - \omega_2 \|_{\infty})
    \end{equation*}
    for all multi-indices $\alpha$, $\beta$ with $1 \le |\alpha| \le 3$ and $0 \le |\beta| \le 3$.
    For $j,k,l \in \{1,2,3\}$, we have
    \begin{equation*}
        \begin{aligned}
            \| \partial_j \partial_k (X_i)_l \|_{\infty,\infty} 
            &\le CT(R + C_0 + C_T^*), &&\| \partial_j \partial_k (X_i)_l \|_{\rL^\infty(0,T;\rW^{1,\infty}(\cF_0))} \le CT(R + C_0 + C_T^*),\\
            \| \partial_j \partial_k (Y_i)_l \|_{\infty,\infty} 
            &\le CT(R + C_0 + C_T^*) \tand &&\| \partial_j \partial_k (Y_i)_l \|_{\rL^\infty(0,T;\rW^{1,\infty}(\cF_0))} \le CT(R + C_0 + C_T^*).
        \end{aligned}
    \end{equation*}
    \item The matrix $\bO_i$ from \autoref{sec:change of var} fulfills
    \begin{equation*}
        \| \bO_1 - \bO_2 \|_\infty \le C T \cdot \| M_1 - M_2 \|_\infty \le C T \cdot \| \omega_1 - \omega_2 \|_\infty \tand \| \bO_i \|_\infty + \| \bO_i^\top \|_\infty \le C.
    \end{equation*}
    \item For every $\alpha$ with $0 \le |\alpha| \le 1$, the covariant and contravariant tensor and the Christoffel symbol associated to $X$, $X_1$, $X_2$ and $Y$, $Y_1$, $Y_2$ fulfill $\| \partial^\alpha g^{ij} \|_{\infty,\infty} + \| \partial^\alpha g_{ij} \|_{\infty,\infty} + \| \partial^\alpha \Gamma_{jk}^i \|_{\infty,\infty} \le C$ and
    \begin{equation*}
        \| \partial^\alpha ((g_1)^{ij} - (g_2)^{ij}) \|_{\infty,\infty} + \| \partial^\alpha ((g_1)_{ij} - (g_2)_{ij}) \|_{\infty,\infty} + \| \partial^\alpha ((\Gamma_1)_{jk}^i - (\Gamma_2)_{jk}^i) \|_{\infty,\infty} \le C T \cdot \| (\tell_1 - \tell_2,\tomega_1 - \tomega_2) \|_\infty.
    \end{equation*}   
    \item For $i,j \in \{1,2,3\}$, we have $ \| g^{ij} - \delta_{ij} \|_{\infty,\infty}$, $\| \partial_i X_j - \delta_{ij} \|_{\infty,\infty}$, $\| \partial_i Y_j - \delta_{ij} \|_{\infty,\infty} \le C T(R + C_0 + C_T^*)$.
    The estimates by $C T(R + C_0 + C_T^*)$ remain valid when $\| \cdot \|_{\infty,\infty}$ is replaced by $\| \cdot \|_{\rL^\infty(0,T;\rW^{1,\infty}(\cF_0))}$.
\end{enumerate}
\end{lem}

Below, we collect estimates of some terms from \eqref{eq:rhs F_1 to F_4}.
Here terms associated to $z_i$ are indicated by $^{(i)}$.
The estimates of the first four terms in \autoref{lem:prep ests nonlin terms}(a) can be deduced from \cite[Lemma~6.6]{GGH:13}, while the last two terms in~(a) can be handled similarly as in \cite[(6.8) and Proposition~6.6]{BBHR:23} upon invoking the estimates in $\rL^\infty(0,T;\rW^{1,\infty}(\cF_0))$ from \autoref{lem:ests of trafo and related terms} and $\| \tQ \|_{\rL^p(0,T;\rH^2(\cF_0))} \le T^{\nicefrac{1}{p}} \cdot \| \tQ \|_{\rL^\infty(0,T;\rH^2(\cF_0))} \le C T^{\nicefrac{1}{p}} \cdot \| \tz \|_{\E_T}$ for $\tQ \in \prescript{}{0}{\E_T^Q}$ by \autoref{lem:ests of max reg space}(a).
For the estimates in \autoref{lem:prep ests nonlin terms}(b) concerning the Cauchy stress tensor part, we refer to \cite[Lemma~6.6]{GGH:13}, while the estimate of the remaining part can be derived from mapping properties of~$\cJ$ defined below, the estimate $\| \bO - \Id_3 \|_\infty \le C T$, following in turn from \autoref{lem:ests of trafo and related terms}(b), and a suitable expansion of the difference.

\begin{lem}\label{lem:prep ests nonlin terms}
Let $p > 4$, recall $(z^*,\pi^*)$ from \autoref{prop:ref sol}, and consider $(z,\pi) = (\tz + z^*,\tp + \pi^*)$ for~$(\tz,\tp) \in \cK_T^R$.
Then there exists $C(R,T) > 0$, with $C(R,T) \to 0$ as $R \to 0$ and $T \to 0$, such that
\begin{enumerate}[(a)]
    \item $\| (\cL_1 - \Delta)v \|_{\F_T^v}$, $\| \cM v \|_{\F_T^v}$, $\| \cN(v) \|_{\F_T^v}$, $\| (\cG - \nabla)\pi \|_{\F_T^v}$, $\| (\cL_2 - \Delta)Q \|_{\F_T^Q}$, $\| (\del_t Y \cdot \nabla) Q  \|_{\F_T^Q} \le C(R,T)$.
    \item For $\eps > 0$, define $\cJ \colon \rH^{\eps + \nicefrac{1}{2}}(\cF_0)^{3 \times 3} \to \R^6$ by $\cJ(h) \coloneqq (\int_{\del\cS_0} h n \rd \Gamma, \int_{\del\cS_0} y \times h n \rd \Gamma)^\top$.
    Then there is~$C > 0$ with $\| \cJ((\rT - \cT)(v,\pi,Q)) \|_p \le C T \cdot \| z \|_{\E_T}$.
\end{enumerate}
\end{lem}

The lemma below on Lipschitz estimates of the non-transformed terms follows from \cite[Lemma~6.3]{HHW:24}.

\begin{lem}\label{lem:Lipschitz est of lc part}
Let $p > 4$ and $z_1$, $z_2 \in \rX_\gamma$, with $\rX_\gamma$ as defined in \eqref{eq:X_gamma}, and recall $A_\xi$ from \eqref{eq:op matrix liquid crystals}.
Then there exists $C > 0$ such that $\| A_\xi(Q_1) - A_\xi(Q_2) \|_{\cL(\rH^2(\cF_0) \times \rH^3(\cF_0),\rL^2(\cF_0) \times \rH^1(\cF_0))} \le C \cdot \| z_1 - z_2 \|_{\rX_\gamma}$.
\end{lem}

As a preparation of the estimates of $F_1$, we now address the second line of $F_1$ as made precise in \eqref{eq:rhs F_1 to F_4}.

\begin{lem}\label{lem:est highest order F_1}
Consider $p > 4$, and for the reference solution $(z^*,\pi^*)$ from \autoref{prop:ref sol} and $(\tz,\tp) \in \cK_T^R$, let $(z,\pi) = (\tz + z^*,\tp + \pi^*)$.
Then there exists $C(R,T) > 0$ such that $C(R,T) \to 0$ as $R \to 0$ and $T \to 0$, and with $\| \cB_{\tau_\rh}(\tQ + Q^*) + \cB_{\sigma}(\tQ + Q^*) - \mdiv S_\xi(Q_0)(\Delta - \Id_3)(\tQ + Q^*) \|_{\F_T^v} \le C(R,T)$.
\end{lem}

\begin{proof}
The idea is to add and subtract terms, leading to two estimates which only concern the non-transformed terms and an estimate of the difference of transformed and non-transformed terms.
From \autoref{sec:change of var}, we recall $B_{\tau_\rh}(\tQ + Q^*)(\tQ + Q^*) + B_\sigma(\tQ + Q^*)(\tQ + Q^*) = \mdiv S_\xi(\tQ + Q^*)(\Delta - \Id_3)(\tQ + Q^*)$.
With this identity, we expand the difference to be estimated as
\begin{equation*}
    \begin{aligned}
        \rI + \rII + \rIII
        &\coloneqq \cB_{\tau_\rh}(\tQ + Q^*) - B_{\tau_\rh}(\tQ + Q^*)(\tQ + Q^*) + \cB_{\sigma}(\tQ + Q^*) - B_\sigma(\tQ + Q^*)(\tQ + Q^*)\\
        &\quad + \mdiv S_\xi(\tQ + Q^*)(\Delta - \Id_3)(\tQ + Q^*) - \mdiv S_\xi(Q^*)(\Delta - \Id_3)(\tQ + Q^*)\\
        &\quad + \mdiv S_\xi(Q^*)(\Delta - \Id_3)(\tQ + Q^*) - \mdiv S_\xi(Q_0)(\Delta - \Id_3)(\tQ + Q^*).
    \end{aligned}
\end{equation*}
With regard to the terms $\rII$ and $\rIII$, recall that $\tz$, $z^* \in \E_T$ by construction.
Hence, \autoref{lem:ests of max reg space}(b) implies that $\tz(t) + z^*(t)$, $z^*(t) \in \rX_\gamma$ for all $t \in [0,T]$, and $z_0 \in \rX_\gamma$ is valid by assumption.
Thus, \autoref{lem:Lipschitz est of lc part} and the estimate of $\| z \|_{\E_T}$ from \eqref{eq:est of norm of (z,p)} as well as \autoref{lem:ests of max reg space}(b) imply that
\begin{equation*}
    \begin{aligned}
        \| \rII \|_{\F_T^v}
        &\le \Bigl(\int_0^T \bigl\| \bigl(A_\xi(\tQ(t) + Q^*(t)) - A_\xi(Q^*(t))\bigr) (\tv(t) + v^*(t),\tQ(t) + Q^*(t))^\top \bigr\|_{\rL^2(\cF_0) \times \rH^1(\cF_0)}^p \rd t\Bigr)^{\nicefrac{1}{p}}\\
        &\le C \Bigl(\int_0^T \| \tz(t) \|_{\rX_\gamma}^p \cdot \| \tz(t) + z^*(t) \|_{\rX_1}^p \rd t\Bigr)^{\nicefrac{1}{p}} 
        \le C \cdot \| \tz \|_{\BUC([0,T];\rX_\gamma)} \cdot \| z \|_{\E_T}
        \le C R (R + C_T^*).
    \end{aligned}
\end{equation*}
In a similar way, we find that $\| \rIII \|_{\F_T^v} \le C (R + C_T^*) \cdot \| z^* - z_0 \|_{\BUC([0,T];\rX_\gamma)}$.
From \autoref{prop:ref sol}, we recall that $\| z^* - z_0 \|_{\BUC([0,T];\rX_\gamma)}$ tends to zero as $T \to 0$.

It thus remains to estimate $\rI$.
We split this term into the part associated to $\tau_\rh$ and $\sigma$.
For brevity, we only elaborate on the estimate of $\rI_\sigma \coloneqq \cB_{\sigma}(\tQ + Q^*) - B_\sigma(\tQ + Q^*)(\tQ + Q^*)$ and remark that handling the difference $\cB_{\tau_\rh}(\tQ + Q^*) - B_{\tau_\rh}(\tQ + Q^*)(\tQ + Q^*)$ is conceptually similar.
In fact, the idea is to expand the terms to get differences which allow for estimates by a positive power of $T$ by \autoref{lem:ests of trafo and related terms}.

With regard to $\rI_\sigma$, recalling $B_\sigma(\tQ + Q^*)(\tQ + Q^*)$ and $\cB_{\sigma}(\tQ + Q^*)$ from \eqref{eq:B_sigma} and \eqref{eq:cB_sigma}, respectively, we remark that it suffices to consider the first two terms in the difference by symmetry.
In the sequel, we concentrate on the second term, as it involves the highest derivatives.
The first term in the difference~$\rI_{\sigma,1}$ can be treated in a similar fashion.
Thus, we expand $\rI_{\sigma,2} \coloneqq \cB_{\sigma,12}(Q)Q - \sum_{j,k=1}^3 Q_{ik} (\del_j \Delta Q_{kj})$ as
\begin{equation*}
    \begin{aligned}
        \rI_{\sigma,2}
        &= \sum_{j,k=1}^3 Q_{ik}\Biggl(\sum_{l=1}^3 \biggl(\sum_{m=1}^3 \Bigl(\sum_{n=1}^3 (\del_n \del_m \del_l Q_{kj})(\del_j Y_n - \delta_{jn})g^{lm} + (\del_n \del_m \del_l Q_{kj}) \delta_{jn} (g^{lm} - \delta_{lm})\Bigr)\\
        &\qquad + (\del_m \del_l Q_{kj})(\del_j g^{lm}) + (\del_m \del_l Q_{kj})(\del_j Y_m) \Delta Y_l\biggr) + (\del_l Q_{kj})(\del_j \Delta Y_l)\Biggr).
    \end{aligned}
\end{equation*}
The factor $Q_{ik}$ can be bounded in $\rL^\infty(0,T;\rL^\infty(\cF_0))$, since the embeddings from \autoref{lem:ests of max reg space}(b) and the estimate of $\| z^* \|_{\BUC([0,T];\rX_\gamma)}$ from \autoref{lem:est by init data and norm of sol with appl to ref sol} especially imply that
\begin{equation}\label{eq:est of del_l Q_ik}
    \| Q \|_{\rL^\infty(0,T;\rW^{1,\infty}(\cF_0))} \le C( \| \tz \|_{\BUC([0,T];\rX_\gamma)} + \| z^* \|_{\BUC([0,T];\rX_\gamma)}) \le C(R + C_0 + C_T^*).
\end{equation}

Now, we derive estimates of second derivatives of $Q$.
As $p > 4$, there is $q > p$ with $0 < \eta \coloneqq \nicefrac{1}{p} - \nicefrac{1}{q} < \nicefrac{1}{2}$.
From \autoref{lem:ests of max reg space}(a) and Sobolev embeddings, and for $\theta \in (\eta,\nicefrac{1}{2})$, it then follows that 
\begin{equation*}
    \begin{aligned}
        \| Q \|_{\rL^p(0,T;\rH^2(\cF_0))} 
        \le T^{\eta} \cdot \| \tQ \|_{\rL^q(0,T;\rH^2(\cF_0))} + C_T^*
        \le C T^\eta \cdot \| \tQ \|_{\rH^{\theta,p}(0,T;\rH^{2(1-\theta) + 1}(\cF_0))} + C_T^*
        \le C T^\eta R + C_T^*.
    \end{aligned}
\end{equation*}
From \eqref{eq:est of norm of (z,p)}, we deduce $\| \del^\alpha Q_{kj} \|_{\rL^p(0,T;\rL^2(\cF_0))} \le \| z \|_{\E_T} \le R + C_T^*$ for all multi-indices $\alpha$ with $0 \le |\alpha| \le 3$.
By virtue of \autoref{lem:ests of trafo and related terms}(d), the differences $\del_j Y_n - \delta_{jn}$ and $g^{lm} - \delta_{lm}$ admit estimates by $C T (R + C_0 + C_T^*)$ in~$\rL^\infty(0,T;\rL^\infty(\cF_0))$.
On the other hand, the boundedness of terms as $\del_j g^{lm}$, $\del_j Y_m$, $\Delta Y_l$ and $\del_j \Delta Y_l$ follows from \autoref{lem:ests of trafo and related terms}(c) and~(a), while a decay estimate of the term $\del_l Q_{kj}$ is a consequence of the above estimate of $ \| Q \|_{\rL^p(0,T;\rH^2(\cF_0))}$.
In total, for some $C_{\rI_{\sigma,2}}(R,T) > 0$ with $ C_{\rI_{\sigma,2}}(R,T) \to 0$ as $R \to 0$ and $T \to 0$, we conclude $\| \rI_{\sigma,2} \|_{\F_T^v} \le C_{\rI_{\sigma,2}}(R,T)$.
In view of the above remarks, this completes the proof.
\end{proof}

The next proposition asserts the self-map and Lipschitz estimates of $F_1$.
It is a consequence of \autoref{lem:prep ests nonlin terms}, \autoref{lem:est highest order F_1} and the observation that the estimates of the transformed lower order terms~$\cB_{\tau_\rl}$ are conceptually similar to those in \autoref{lem:est highest order F_1}.
The Lipschitz estimates follow in the same way as the self-map estimate, see also \cite[Section~6.1]{GGH:13} or \cite[Section~6.2]{BBHR:23} for examples of such Lipschitz estimates.

\begin{prop}\label{prop:ests of F_1}
Let $p > 4$, and for the reference solution $(z^*,\pi^*)$ from \autoref{prop:ref sol} and $(\tz,\tp)$, $(\tz_i,\tp_i) \in \cK_T^R$, $i=1,2$, consider $(z,\pi) = (\tz + z^*,\tp + \pi^*)$ and $(z_i,\pi_i) = (\tz_i + z^*,\tp_i + \pi^*)$.
Besides, recall the $T$-independent constant $\Cmr > 0$ from \autoref{thm:lin theory Q-tensor interaction}(a).
Then there are $C_1(R,T)$, $L_1(R,T) > 0$ such that $C_1(R,T) < \nicefrac{R}{8 \Cmr}$ for $T > 0$ sufficiently small and $L_1(R,T) \to 0$ as $R \to 0$ and $T \to 0$, and so that
\begin{equation*}
    \| F_1(\tz,\tp) \|_{\F_T^v} \le C_1(R,T) \tand \| F_1(\tz_1,\tp_1) - F_1(\tz_2,\tp_2) \|_{\F_T^v} \le L_1(R,T) \cdot \| (\tz_1,\tp_1) - (\tz_2,\tp_2) \|_{\tE_T}.
\end{equation*}
\end{prop}

In order to estimate $F_2$, we also deal with the most difficult term in advance.

\begin{lem}\label{lem:est highest order F_2}
Let $p > 4$, and for the reference solution $(z^*,\pi^*)$ from \autoref{prop:ref sol} and $(\tz,\tp) \in \cK_T^R$, let~$(z,\pi) = (\tz + z^*,\tp + \pi^*)$ and $(z_i,\pi_i) = (\tz_i + z^*,\tp_i + \pi^*)$.
Then there is $C(R,T) > 0$, with $C(R,T) \to 0$ as~$R \to 0$ and $T \to 0$, such that $\| \cS(\tv + v^*,\tQ + Q^*) + \tS_\xi(Q_0)\nabla(\tv + v^*) \|_{\F_T^Q} \le C(R,T)$.
\end{lem}

\begin{proof}
As in the proof of \autoref{lem:est highest order F_1}, the idea here is to split the difference suitably.
In fact, invoking that $-S(\nabla v,Q) = \tS_\xi(Q)\nabla v$ from \autoref{sec:energy and equilibria}, we write
\begin{equation*}
    \begin{aligned}
        \ri + \rii + \riii
        &\coloneqq \bigl(\cS(\tv + v^*,\tQ + Q^*) - S(\nabla(\tv + v^*),\tQ + Q^*)\bigr) - \bigl(\tS_\xi(\tQ + Q^*)\nabla(\tv + v^*) - \tS_\xi(Q^*)\nabla(\tv + v^*)\bigr)\\
        &\quad - \bigl(\tS_\xi(Q^*)\nabla(\tv + v^*) - \tS_\xi(Q_0)\nabla(\tv + v^*)\bigr).
    \end{aligned}
\end{equation*}
From the proof of \autoref{lem:est highest order F_1}, we recall that $\tz(t) + z^*(t)$, $z^*(t) \in \rX_\gamma$ for all $t \in [0,T]$, so \autoref{lem:Lipschitz est of lc part} is applicable.
As for the estimate of $\rII$ and $\rIII$ in the proof of \autoref{lem:est highest order F_1}, it then follows that
\begin{equation*}
    \| \rii \|_{\F_T^Q} = \bigl\| \bigl(\tS_\xi(\tQ + Q^*) - \tS_\xi(Q^*)\bigr) \nabla(\tv + v^*) \bigr\|_{\rL^p(0,T;\rH^1(\cF_0))} \le C \cdot \| \tz \|_{\BUC([0,T];\rX_\gamma)} \cdot \| z \|_{\E_T} \le C R (R + C_T^*),
\end{equation*}
\begin{equation*}
    \| \riii \|_{\F_T^Q} = \bigl\| \bigl(\tS_\xi(Q^*) - \tS_\xi(Q_0)\bigr) \nabla(\tv + v^*) \bigr\|_{\rL^p(0,T;\rH^1(\cF_0))} \le C (R + C_T^*) \cdot \| z^* - z_0 \|_{\BUC([0,T];\rX_\gamma)}.
\end{equation*}
The last factor tends to zero as $T \to 0$ by \autoref{prop:ref sol}, showing the estimates for $\rii$ and $\riii$.
It remains to deal with the term $\ri$.
To this end, recalling $\cD$ from \eqref{eq:transformed grad}, we expand the difference $\cD_v \coloneqq \cD - \nabla v$ as
\begin{equation}\label{eq:repr diff transformed grad and grad}
    (\cD_v)_{ij} = \sum_{k=1}^3 \sum_{l=1}^3 (\del_l \del_k X_i) (\del_j Y_l) v_k + \sum_{k=1}^3 (\del_k X_i - \delta_{ki}) \sum_{l=1}^3 (\del_l v_k) (\del_j Y_l) + \sum_{l=1}^3 \del_l v_i (\del_j Y_l - \delta_{jl}).
\end{equation}
Making use of H\"older's inequality together with \autoref{lem:ests of trafo and related terms} in order to estimate the second derivatives of $X$ and the differences $\del_k X_i - \delta_{ki}$ and $\del_j Y_l - \delta_{jl}$ by $T(R + C_0 + C_T^*)$, we infer that
\begin{equation}\label{eq:est of diff transformed grad and grad}
    \| \cD_v \|_{\rL^p(0,T;\rH^1(\cF_0))} \le C T (R + C_0 + C_T^*) (R + C_T^*),
\end{equation}
and likewise for $\cD_v^\top = \cD^\top - \nabla v^\top$.
We recall the respective terms from \autoref{sec:change of var} and get
\begin{equation*}
    \begin{aligned}
        \ri = -\frac{1}{2}\bigl([Q,\cD_v] - [Q,\cD_v^\top]\bigr) + \xi \Bigl(\frac{1}{3}\bigl(\cD_v + \cD_v^\top\bigr) + \frac{1}{2}\bigl(\{Q,\cD_v\} + \{Q,\cD_v^\top\}\bigr) - 2 (Q + \nicefrac{\Id_3}{3}) \tr(Q \cD_v)\Bigr).
    \end{aligned}
\end{equation*}
Observe that every addend contains a factor of the form $\cD - \nabla v$ or $\cD^\top - \nabla v^\top$, so \eqref{eq:est of diff transformed grad and grad} can be used.
For the terms involving products with $Q$, we employ H\"older's inequality and invoke the estimate of~$\| Q \|_{\rL^\infty(0,T;\rW^{1,\infty}(\cF_0))}$ from \eqref{eq:est of del_l Q_ik}.
In conclusion, we find that there exists $C_\ri(R,T)$, with $C_\ri(R,T) \to 0$ as $R \to 0$ and $T \to 0$, such that $\| \ri \|_{\F_T^Q} \le C_\ri(R,T)$, finishing the proof.
\end{proof}

The proposition below asserts that $F_2$ admits the desired estimates.

\begin{prop}\label{prop:ests of F_2}
Consider $p > 4$, and for the reference solution $(z^*,\pi^*)$ from \autoref{prop:ref sol} as well as $(\tz,\tp)$, $(\tz_i,\tp_i) \in \cK_T^R$, $i=1,2$, let $(z,\pi) = (\tz + z^*,\tp + \pi^*)$ and $(z_i,\pi_i) = (\tz_i + z^*,\tp_i + \pi^*)$. 
Denote by $\Cmr > 0$ the $T$-independent constant from \autoref{thm:lin theory Q-tensor interaction}(a).
Then there exist $C_2(R,T)$, $L_2(R,T) > 0$ with $C_2(R,T) < \nicefrac{R}{8 \Cmr}$ for $T > 0$ sufficiently small and $L_2(R,T) \to 0$ as $R \to 0$ and $T \to 0$, and with
\begin{equation*}
    \| F_2(\tz) \|_{\F_T^Q} \le C_2(R,T) \tand \| F_2(\tz_1) - F_2(\tz_2) \|_{\F_T^Q} \le L_2(R,T) \cdot \| \tz_1 - \tz_2 \|_{\E_T}.
\end{equation*}
\end{prop}

\begin{proof}
In view of \autoref{lem:prep ests nonlin terms}, \autoref{lem:est highest order F_2} and the fact that the estimates of $\mu Q$ are simple, we still need to estimate $((\tv + v^*) \cdot \nabla)(\tQ + Q^*)$, $(\tQ + Q^*)^2 - \tr((\tQ + Q^*)^2)\nicefrac{\Id_3}{3}$ and $-\tr((\tQ + Q^*)^2) (\tQ + Q^*)$.

First, we discuss estimates of powers of $Q$.
For $m \ge 2$, recalling the estimate of $\| Q \|_{\rL^\infty(0,T;\rW^{1,\infty}(\cF_0))}$ from \eqref{eq:est of del_l Q_ik}, and using that \autoref{lem:ests of max reg space}(b) implies $\| Q \|_{\rL^p(0,T;\rH^1(\cF_0))} \le C_T^* + T^{\nicefrac{1}{p}} R$, we obtain
\begin{equation}\label{eq:est of power}
    \| Q^m \|_{\rL^p(0,T;\rH^1(\cF_0))} \le \| Q \|_{\rL^\infty(0,T;\rW^{1,\infty}(\cF_0))}^{m-1} \cdot \| Q \|_{\rL^p(0,T;\rH^1(\cF_0))} \le C(C_0 + R + C_T^*)^{m-1}(C_T^* + T^{\nicefrac{1}{p}}R).
\end{equation}
As $\tr((\tQ + Q^*)^2)\nicefrac{\Id_3}{3}$ can be handled similarly, we get $\| Q^2 - \tr(Q^2)\nicefrac{\Id_3}{3} \|_{\F_T^Q} \le C(C_0 + R + C_T^*)(C_T^* + T^{\nicefrac{1}{p}}R)$.

Concerning the estimate of $(v \cdot \nabla)Q$, we note that $\rH^{\theta,p}\bigl(0,T;\rH^{2(1-\theta)}(\cF_0)\bigr) \hookrightarrow \rL^q\bigl(0,T;\rW^{1,2s}(\cF_0)\bigr)$ holds for $q > p$ and $\theta \in (\nicefrac{1}{p}-\nicefrac{1}{q},\nicefrac{3}{4s} - \nicefrac{1}{4})$, requiring $s < 3$, or, equivalently, $s' > \nicefrac{3}{2}$ for $\nicefrac{1}{s} + \nicefrac{1}{s'} = 1$. 
For such~$s$, \autoref{lem:ests of max reg space}(a) yields $\| v \|_{\rL^p(0,T;\rW^{1,2s}(\cF_0))} \le T^\eta \cdot \| \tv \|_{\rL^q(0,T;\rW^{1,2s}(\cF_0))} + \| z^* \|_{\E_T} \le C T^\eta R + C_T^*$ for $\eta \coloneqq \nicefrac{1}{p} - \nicefrac{1}{q}$.
On the other hand, thanks to $p > 4$, there is $s' > \nicefrac{3}{2}$ such that $\rB_{2p}^{3-\nicefrac{2}{p}}(\cF_0) \hookrightarrow \rW^{2,2s'}(\cF_0)$.
By \autoref{lem:ests of max reg space}(b), and with $\| z \|_{\BUC([0,T];\rX_\gamma)} \le C_0 + R + C_T^*$, we get $\| \nabla Q \|_{\rL^\infty(0,T;\rW^{1,2s'}(\cF_0))} \le C(C_0 + R + C_T^*)$.
H\"older's inequality and a combination of the above estimates then results in
\begin{equation*}
    \| (v \cdot \nabla) Q \|_{\F_T^Q} \le \| v \|_{\rL^p(0,T;\rW^{1,2s}(\cF_0))} \cdot \| \nabla Q \|_{\rL^\infty(0,T;\rW^{1,2s'}(\cF_0))} \le C(C_0 + R + C_T^*) (T^\eta R + C_T^*).
\end{equation*}

Let us observe that the Lipschitz estimate can be obtained in a similar way.
\end{proof}

Finally, we estimate the terms $F_3$ and $F_4$.
The complexity of the stress tensor and the present anisotropic Hilbert space setting renders renders this task challenging.

\begin{prop}\label{prop:ests of F_3 and F_4}
Let $p > 4$, recall the reference solution $(z^*,\pi^*)$ from \autoref{prop:ref sol}, and for $(\tz,\tp)$, $(\tz_i,\tp_i) \in \cK_T^R$, $i=1,2$, set $(z,\pi) = (\tz + z^*,\tp + \pi^*)$, $(z_i,\pi_i) = (\tz_i + z^*,\tp_i + \pi^*)$.
For $\Cmr > 0$ from \autoref{thm:lin theory Q-tensor interaction}(a), there are $C_3(R,T)$, $C_4(R,T) > 0$, with $C_3(R,T)$, $C_4(R,T) < \nicefrac{R}{8 \Cmr}$ for $T > 0$ small enough, and $L_3(R,T)$, $L_4(R,T) > 0$, with $L_3(R,T)$, $L_4(R,T) \to 0$ as $R \to 0$, $T \to 0$, so that
\begin{equation*}
    \begin{aligned}
        \| F_3(\tz,\tp) \|_{\F_T^\ell} 
        &\le C_3(R,T), \enspace \| F_3(\tz_1,\tp_1) - F_3(\tz_2,\tp_2) \|_{\F_T^\ell} \le L_3(R,T) \cdot \| (\tz_1,\tp_1) - (\tz_2,\tp_2) \|_{\tE_T} \taswellas\\
        \| F_4(\tz,\tp) \|_{\F_T^\omega} 
        &\le C_4(R,T), \enspace \| F_4(\tz_1,\tp_1) - F_4(\tz_2,\tp_2) \|_{\F_T^\omega} \le L_4(R,T) \cdot \| (\tz_1,\tp_1) - (\tz_2,\tp_2) \|_{\tE_T}.
    \end{aligned}
\end{equation*}
\end{prop}

\begin{proof}
Note that the estimates of $\omega \times \ell$, $\omega \times (J_0 \omega)$, $\mu \ell$ and $\mu \omega$ are straightforward, while the estimates of the Cauchy part of the stress tensor follow from \autoref{lem:prep ests nonlin terms}(b).
Hence, it remains to handle the surface integral terms involving $\bO(t)^\top(\tau_{\rr}(Q) + \sigma(Q))\bO(t)$, where the most difficult terms are $\nabla Q \odot \nabla Q$, $2 \xi Q \tr(Q \Delta Q)$ and $2 \xi Q \tr(Q^2) \tr(Q^2)$.
The study of the other terms can be reduced to the latter ones.

The continuity of the trace from $\rH^{\nicefrac{1}{2}+\eps}(\cF_0)$ to $\rL^2(\del\cS_0)$, $\eps > 0$, the bound of $\| \bO \|_\infty$ by \autoref{lem:ests of trafo and related terms}(b) and the embedding from \autoref{lem:ests of max reg space}(b) together with $\rB_{2p}^{3-\nicefrac{2}{p}}(\cF_0) \hookrightarrow \rH^{\nicefrac{3}{2} + \eps}(\cF_0)$ first yield
\begin{equation*}
    \begin{aligned}
        \left \| \int_{\del\cS_0} (\bO^\top \nabla Q \odot \nabla Q \bO) n \rd \Gamma \right\|_p
        &\le C \cdot \| \tQ + Q^* \|_{\rL^\infty(0,T;\rH^{\nicefrac{3}{2}+\eps}(\cF_0))} \cdot \| \tQ + Q^* \|_{\rL^p(0,T;\rH^{\nicefrac{3}{2}+\eps}(\cF_0))}\\
        &\le C\left(\| \tz \|_{\E_T} + \| z^* \|_{\BUC([0,T];\rX_\gamma)}\right) \left(T^{\nicefrac{1}{p}} \cdot \| \tQ \|_{\rL^\infty(0,T;\rH^{\nicefrac{3}{2}+\eps}(\cF_0))} + \| z^* \|_{\E_T}\right)\\
        &\le C(R + C_T^* + C_0) (T^{\nicefrac{1}{p}}R + C_T^*).
    \end{aligned}
\end{equation*}

We proceed with the estimate of $2 \xi Q \tr(Q \Delta Q)$, and as $\xi \in \R$ is fixed, we omit it in the sequel.
Similarly as above, using $\rB_{2p}^{3-\nicefrac{2}{p}}(\cF_0) \hookrightarrow \rW^{\nicefrac{1}{4}+\eps,4}(\cF_0)$ for the first two factors, and employing \autoref{lem:ests of max reg space}(a) joint with $\rH^{\theta,p}\bigl(0,T;\rH^{1+2(1-\theta)}(\cF_0)\bigr) \hookrightarrow \rL^q\bigl(0,T;\rH^{\nicefrac{5}{2}+\eps}(\cF_0)\bigr)$ for $q > p$ and $\theta \in (\nicefrac{1}{p} - \nicefrac{1}{q},\nicefrac{1}{4})$, we get
\begin{equation*}
    \begin{aligned}
        \left \| \int_{\del\cS_0} (\bO^\top Q \tr(Q \Delta Q) \bO) n \rd \Gamma \right\|_p
        &\le C(R + C_T^* + C_0)^2 \left(T^\eta \cdot \| \tQ \|_{\rL^q(0,T;\rH^{\nicefrac{5}{2}+\eps}(\cF_0))} + \| z^* \|_{\E_T}\right)\\
        &\le (R + C_T^* + C_0)^2 (T^\eta R + C_T^*),
    \end{aligned}
\end{equation*}
where $\eta = \nicefrac{1}{p} - \nicefrac{1}{q} > 0$.
The last term is $2 \xi Q \tr(Q^2) \tr(Q^2)$, and we observe that $Q \tr(Q^2) = Q |Q|_2^4$ for the Hilbert-Schmidt norm $|\cdot|_2$.
The resulting first four factors can be treated as above upon noting that~$\rB_{2p}^{3-\nicefrac{2}{p}}(\cF_0) \hookrightarrow \rW^{\nicefrac{1}{8}+\eps,8}(\cF_0)$.
The last factor can be handled as in the preceding estimate, so
\begin{equation*}
    \left \| \int_{\del\cS_0} (\bO^\top Q \tr(Q^2) \tr(Q^2) \bO) n \rd \Gamma \right\|_p \le C(R + C_T^* + C_0)^4 (T^\eta R + C_T^*).
\end{equation*}

Similar arguments yield the Lipschitz estimate of $F_3$ by the multilinear structure of the terms.
Finally, the estimates of the surface integrals for $F_4$ are completely analogous to those of $F_3$.
\end{proof}

We are now in the position to show the local strong well-posedness for large data.

\begin{proof}[Proof of \autoref{thm:loc strong wp}]
We start with the proof of \autoref{thm:loc strong wp in ref config}, so let $(\tz,\tp) \in \cK_T^R$.
\autoref{thm:lin theory Q-tensor interaction}(a) applied to \eqref{eq:linearized syst with hom initial data} together with \autoref{prop:ests of F_1}, \autoref{prop:ests of F_2} and \autoref{prop:ests of F_3 and F_4} results in 
\begin{equation*}
    \| \Phi_T^R(\tz,\tp) \|_{\tE_T} \le \Cmr \cdot \| (F_1(\tz,\tp),F_2(\tz),F_3(\tz,\tp),F_4(\tz,\tp)) \|_{\F_T} \le \Cmr C(R,T).
\end{equation*}
For given $R > 0$, we get $C(R,T) < \nicefrac{R}{2 \Cmr}$ provided $T > 0$ is sufficiently small.
In the same way,
\begin{equation*}
    \| \Phi_T^R(\tz_1,\tp_1) - \Phi_T^R(\tz_2,\tp_2) \|_{\tE_T} \le \Cmr L(R,T) \cdot \| (\tz_1,\tp_1) - (\tz_2,\tp_2) \|_{\tE_T}, \tfor (\tz_1,\tp_1), \, (\tz_2,\tp_2) \in \cK_T^R.
\end{equation*}
As $\Cmr > 0$ is $T$-independent by virtue of the homogeneous initial values, we may choose $R > 0$ and then also $T > 0$ sufficiently small such that $\Cmr C(R,T) \le \nicefrac{R}{2}$ and $\Cmr L(R,T) \le \nicefrac{1}{2}$.
In other words, the solution map $\Phi_T^R$ to the linearized problem \eqref{eq:linearized syst with hom initial data} is a self-map and contraction, giving rise to a unique fixed point $(\hz,\hp) \in \prescript{}{0}{\E_T} \times \E_T^{\pi}$.
An addition of the reference solution $(z^*,\pi^*)$ from \autoref{prop:ref sol} leads to the unique solution $(z,\pi) = (\hz + z^*,\hp + \pi^*) \in \tE_T$ to \eqref{eq:cv1}--\eqref{eq:cv3}.
The proof is completed upon noting that~$(z,\pi) \in \tE_T$ means that $(v,Q,\ell,\omega,\pi)$ is in the regularity class as stated in \autoref{thm:loc strong wp in ref config}.

Finally, we deduce \autoref{thm:loc strong wp} from \autoref{thm:loc strong wp in ref config}.
Thus, consider $(\ell,\omega) \in \rW^{1,p}(0,T)^6$ resulting from \autoref{thm:loc strong wp in ref config}.
We recover $h'$, $\Omega$ and $X$ from there.
The backward change of variables from \autoref{sec:change of var} yields the assertion of \autoref{thm:loc strong wp}, where the uniqueness follows from the uniqueness of the diffeomorphism.
\end{proof}

\section{Global strong well-posedness for small data}\label{sec:proof of glob strong wp}

The purpose of this section is to prove \autoref{thm:global strong wp for small data} on the global strong well-posedness of the interaction problem \eqref{eq:LC-fluid equationsiso}--\eqref{eq:initial} for small initial data.
To this end, we set up a fixed point argument based on the maximal $\rL^p(\R_+)$-regularity of the fluid-structure operator, and we derive new decay estimates of the coordinate transform in order to estimate the nonlinear terms.

In view of \autoref{thm:lin theory Q-tensor interaction}(b), we consider here the linearization at zero, without shift, and without invoking a reference solution thanks to the maximal $\rL^p(\R_+)$-regularity, i.\ e., we take into account
\begin{equation}\label{eq:lin Q-tensor for global for small data}
\left\{
    \begin{aligned}
        \partial_t {v}  -\Delta {v} +\nabla \pi - \mdiv S_\xi(0)(\Delta - \Id_3)Q &= G_1(z,\pi), \enspace \mdiv {v} = 0, &&\tin (0,T) \times \cF_0,\\
        {\partial_t {Q}} + \tS_\xi(0) \nabla v  -(\Delta - \Id_3) {Q} 
        &= G_2(z), &&\tin (0,T) \times \cF_0,\\
        \mS({\ell})' + \int_{\del \cS_0} \rT(v,\pi,Q) n \rd \Gamma &= G_3(z,\pi), &&\tin (0,T),\\ 
        J_0({\omega})' + \int_{\del \cS_0} y \times \rT(v,\pi,Q) n \rd \Gamma &= G_4(z,\pi), &&\tin (0,T),\\
        {v} ={\ell}+{\omega}\times y, \ton (0,T) \times \partial\cS_0, \enspace
        {v} &=0, \ton (0,T) \times \partial \cO, \enspace
        \partial_n Q = 0, &&\ton (0,T) \times \partial \cF_0,\\
        v(0) = v_0, \enspace Q(0) = Q_0, \enspace \ell(0) &= \ell_0 \tand \omega(0) = \omega_0. 
    \end{aligned}
\right.
\end{equation}
For the transformed terms as made precise in \autoref{sec:change of var}, the terms on the right-hand side in \eqref{eq:lin Q-tensor for global for small data} are
\begin{equation}\label{eq:rhs G_1 to G_4}
    \begin{aligned}
        G_1(z,\pi)
        &\coloneqq (\cL_1 - \Delta)v - \cM v - \cN(v) - (\cG - \nabla)\pi + \cB_{\tau}(Q) + \cB_{\sigma}(Q) - \mdiv S_\xi(0)(\Delta - \Id_3)Q,\\
        G_2(z)
        &\coloneqq (\cL_2 - \Delta)Q - (\del_t Y \cdot \nabla) Q - (v \cdot \nabla)Q + \cS(v,Q) + \tS_\xi(0)\nabla v + Q^2 - \tr(Q^2)(\nicefrac{\Id_3}{3} + Q),\\
        G_3(z,\pi)
        &\coloneqq -\mS \omega \times \ell + \int_{\del \cS_0}\bigl((\rT - \cT)(v,\pi,Q) - \bO(t)^\top(\tau_{\rr}(Q) + \sigma(Q))\bO(t)\bigr) n \rd \Gamma \tand\\
        G_4(z,\pi)
        &\coloneqq \omega \times (J_0 \omega) + \int_{\del \cS_0} y \times \bigl((\rT - \cT)(v,\pi,Q) - \bO(t)^\top(\tau_{\rr}(Q) + \sigma(Q))\bO(t)\bigr) n \rd \Gamma.
    \end{aligned}
\end{equation}
For $\eta \in (0,\eta_0)$, where $\eta_0 > 0$ has been introduced in \autoref{thm:lin theory Q-tensor interaction}(b), and $\tE_\infty$ as made precise in \eqref{eq:max reg space with pressure}, we define $\tcK$ to be $\tE_\infty$ with exponential weights, i.\ e., $(z,\pi) = (v,Q,\ell,\omega,\pi) \in \tcK$ if and only if
\begin{equation}\label{eq:def of tcK}
    \mre^{\eta(\cdot)}(v,Q,\ell,\omega,\pi) \in \tE_\infty, \twith \| (v,Q,\ell,\omega,\pi) \|_{\tcK} \coloneqq \bigl\| \bigl(\mre^{\eta(\cdot)} v, \mre^{\eta(\cdot)} Q, \mre^{\eta(\cdot)} \ell, \mre^{\eta(\cdot)} \omega, \mre^{\eta(\cdot)} \pi\bigr) \bigr\|_{\tE_\infty}.
\end{equation}
For $\eps > 0$, the ball with center zero and radius $\eps$ in $\tcK$ is denoted by $\tcK_\eps$, and we define $\Psi \colon \tcK_\eps \to \tcK$ to be the solution map to \eqref{eq:lin Q-tensor for global for small data}.
This means that given $(z,\pi) = (v,Q,\ell,\omega,\pi) \in \tcK_\eps$, we denote by $\Psi(z,\pi)$ the solution to \eqref{eq:lin Q-tensor for global for small data} with $G_1(z,\pi)$, $G_2(z)$, $G_3(z,\pi)$ and $G_4(z,\pi)$.
By \autoref{thm:lin theory Q-tensor interaction}(b), the map $\Psi$ is well-defined if the terms on the right-hand side lie in $\mre^{-\eta(\cdot)} \F_\infty \coloneqq \{z = (v,Q,\ell,\omega) \in \F_\infty : \mre^{\eta(\cdot)} z \in \F_\infty\}$, where~$\F_\infty$ is defined in \eqref{eq:data space}, and $z_0 = (v_0,Q_0,\ell_0,\omega_0) \in \rX_\gamma$, with $\rX_\gamma$ from \eqref{eq:X_gamma}.
Throughout the section, we use $\| \cdot \|_\infty$ to denote the $\rL^\infty(0,\infty)$-norm, while~$\| \cdot \|_{\infty,\infty}$ represents the norm of $\rL^\infty(0,\infty;\rL^\infty(\cF_0))$.

In the lemmas below, we collect important ingredients for the nonlinear estimates.
The first one discusses estimates of terms related to the transform.
\autoref{lem:aux ests trafo global}(a)--(c) and~(e) follow from \cite[Lemma~7.3 and Proposition~7.4]{BBHR:23}.
Concerning \autoref{lem:aux ests trafo global}(d), recall from \cite[(7.17)]{BBHR:23} the estimate $\| \del^\beta b_i \|_{\infty,\infty} \le C \mre^{-\eta t} \eps$.
In view of \eqref{eq:IVP X}, this yields $\| \rJ_X - \Id_3 \|_{\rL^\infty(0,\infty;\rW^{1,\infty}(\cF_0))} \le C \eps$ as well as~$\| \rJ_X \|_{\rL^\infty(0,\infty;\rW^{1,\infty}(\cF_0))} \le C$.
The estimates of the second derivatives, of the terms involving~$Y$ and of the other terms follow in a similar manner with regard to \eqref{eq:contrvar covar and Christoffel}.

\begin{lem}\label{lem:aux ests trafo global}
Recall $X$, $Y$ as well as $g^{ij}$, $g_{ij}$ and $\Gamma_{jk}^i$ from \autoref{sec:change of var}, and let $(v,Q,\ell,\omega,\pi) \in \tcK_\eps$.
Then
\begin{enumerate}[(a)]
    \item there exists an $\eps$-independent constant $C > 0$ with $\| \rJ_X - \Id_3 \|_{\infty,\infty} \le C \eps$, so if $\eps < \nicefrac{1}{2C}$, then $\rJ_X$ is invertible on $(0,\infty) \times \cF_0$,
    \item there is $C > 0$ with $\| \rJ_X \|_{\infty,\infty}$, $\| \rJ_Y \|_{\infty,\infty}$, $\| g^{ij} \|_{\infty,\infty}$, $\| g_{ij} \|_{\infty,\infty} \le C$ for all $i,j,k \in \{1,2,3\}$,
    \item there exists $C > 0$ so that for all $i,j,k \in \{1,2,3\}$, we have
    \begin{equation*}
        \begin{aligned}
            &\| \del_t Y \|_{\infty,\infty} + \| \del_j \del_t X_i \|_{\infty,\infty} + \| \rJ_Y - \Id_3 \|_{\infty,\infty} + \| \del_i g^{jk} \|_{\infty,\infty} + \| \del_i g_{jk} \|_{\infty,\infty}\\
            &\quad + \| \Gamma_{jk}^i \|_{\infty,\infty} + \| g^{ij} - \delta_{ij} \|_{\infty,\infty} + \| \del_i \del_j X_k \|_{\infty,\infty} + \| \del_i \del_j Y_k \|_{\infty,\infty} \le C \eps,
        \end{aligned}
    \end{equation*}
    \item the assertions of~(a)--(c) remain valid when $\| \cdot \|_{\infty,\infty}$ is replaced by $\| \cdot \|_{\rL^\infty(0,\infty;\rW^{1,\infty}(\cF_0))}$, and
    \item for $\bO$ and $\bO^\top$, we get $\| \bO \|_\infty$, $\| \bO^\top \|_\infty \le C$ and $\| \bO - \Id_3 \|_\infty$, $\| \bO^\top - \Id_3 \|_\infty \le C \eps$.
\end{enumerate}
\end{lem}

In the following, we take into account small enough $\eps_0$ such that $\rJ_X$ is invertible for all $\eps \in (0,\eps_0)$ by \autoref{lem:aux ests trafo global}(a).
The next lemma provides estimates of some of the terms appearing in \eqref{eq:rhs G_1 to G_4}.
Here \autoref{lem:further aux ests global}(a) follows as in \cite[Proposition~7.4]{BBHR:23}, while for the estimates in~(b), we invoke \autoref{lem:aux ests trafo global}(d) and proceed as in the proof of \autoref{prop:ests of F_2} for the estimate of $(v \cdot \nabla)Q$.
Concerning \autoref{lem:further aux ests global}(c), we make use of $(\rT - \cT)(v,\pi,Q) = (\nabla v + (\nabla v)^\top)(\Id_3 - \bO) - \frac{2 \xi}{3}(\Id_3 - \bO^\top)(\Delta Q - Q) - \frac{2 \xi}{3} \bO^\top (\Delta Q - Q) (\Id_3 - \bO)$, recall the mapping properties of $\cJ$ from \autoref{lem:prep ests nonlin terms}(b) and employ \autoref{lem:aux ests trafo global}(e).

\begin{lem}\label{lem:further aux ests global}
Let $p > 4$, and for $\eps \in (0,\eps_0)$, consider $(v,Q,\ell,\omega,\pi) \in \tcK_\eps$.
Then there exists $C > 0$ so that
\begin{enumerate}[(a)]
    \item $\| \mre^{\eta(\cdot)} (\cL_1 - \Delta)v \|_{\F_\infty^v}$, $\| \mre^{\eta(\cdot)} \cM v \|_{\F_\infty^v}$, $\| \mre^{\eta(\cdot)} \cN(v) \|_{\F_\infty^v}$, $\| \mre^{\eta(\cdot)} (\cG - \nabla)\pi \|_{\F_\infty^v} \le C \eps^2$,
    \item $\| \mre^{\eta(\cdot)} (\cL_2 - \Delta)Q \|_{\F_\infty^Q}$, $\| \mre^{\eta(\cdot)} (\del_t Y \cdot \nabla) Q \|_{\F_\infty^Q}$, $\| \mre^{\eta(\cdot)} (v \cdot \nabla) Q \|_{\F_\infty^Q} \le C \eps^2$, and
    \item for $\cJ$ as defined in \autoref{lem:prep ests nonlin terms}(b), it is valid that $\| \mre^{\eta(\cdot)} \cJ((\rT - \cT)(v,\pi,Q)) \|_p \le C \eps^2$.
\end{enumerate}
\end{lem}

We now successively prove the estimates of the nonlinear terms from~\eqref{eq:rhs G_1 to G_4}.

\begin{prop}\label{prop:ests G_1}
Let $p > 4$.
Then there is $C > 0$ so that for $\eps \in (0,\eps_0)$ and $(z,\pi)$, $(z_1,\pi_1)$, $(z_2,\pi_2) \in \tcK_\eps$, it holds that $\| \mre^{\eta(\cdot)} G_1(z,\pi) \|_{\F_\infty^v} \le C \eps^2$ and $\| \mre^{\eta(\cdot)} (G_1(z_1,\pi_1) - G_1(z_2,\pi_2)) \|_{\F_\infty^v} \le C \eps \cdot \| (z_1,\pi_1) - (z_2,\pi_2) \|_{\tcK}$.
\end{prop}

\begin{proof}
By \autoref{lem:further aux ests global}(a), it remains to estimate $\cB_{\tau_\rl}(Q) + \cB_{\tau_\rh}(Q) + \cB_{\sigma}(Q) - \mdiv S_\xi(0)(\Delta - \Id_3)Q$.
We begin with $\cB_{\sigma}(Q)$ from \eqref{eq:cB_sigma}.
The first two addends have the same structure as the others, so we concentrate on them.
For this, we recall $\cL_2$ from \eqref{eq:transformed Laplacians} and deduce from \autoref{lem:aux ests trafo global}(c) that $\| \mre^{\eta(\cdot)} \cL_2 Q \|_{\F_\infty^v} \le C \eps$.
Arguing that the factor $Q_{ik}$ can be bounded by $C \eps$ in $\rL^\infty(0,\infty;\rL^\infty(\cF_0))$, estimating the first, second and third derivatives of $Q$ in $\rL^p(0,\infty;\rL^2(\cF_0))$ with factor $\mre^{\eta(\cdot)}$ by $\eps$, and using \autoref{lem:aux ests trafo global} to bound the other factors, we obtain $\| \mre^{\eta(\cdot)} \cB_\sigma(Q) \|_{\F_\infty^v} \le C \eps^2$.
The terms $\cB_{\tau_{\rh,2}}$, $\cB_{\tau_{\rh,3}}$, $\cB_{\tau_{\rl,1}}$, $\cB_{\tau_{\rl,2}}$ and $\cB_{\tau_{\rl,3}}$ allow for analogous estimates thanks to the presence of at least two factors with estimates by $C \eps$ in each addend.

For the estimate of the linear term $\cB_{\tau_{\rh,1}}$ from \eqref{eq:cB_tau_rh_1}, we invoke the term $\mdiv S_\xi(0) (\Delta - \Id_3)Q$.
Recalling that it is given by $-\nicefrac{2 \xi}{3} \mdiv (\Delta Q - Q)$, we expand the difference $\cB_{\tau_{\rh,1}}(Q) - \mdiv S_\xi(0)(\Delta - \Id_3)Q$ as
\begin{equation*}
    \begin{aligned}
        &-\frac{2}{3} \sum_{j,k=1}^3 \Biggl(\sum_{l=1}^3 \biggl(\sum_{m=1}^3 \Bigl((\del_m \del_l \del_k Q_{ij})(\del_j Y_m - \delta_{jm}) g^{kl} + (\del_m \del_l \del_k Q_{ij}) \delta_{jm} (g^{kl} - \delta_{kl})\Bigr)\\
        &\quad + (\del_l \del_k Q_{ij})(\del_j g^{kl}) + (\del_l \del_k Q_{ij}) (\del_j Y_l) \Delta Y_k\biggr) + (\del_k Q_{ij}) (\del_j \Delta Y_k) + (\del_j Y_k - \delta_{jk})(\del_k Q_{ij})\Biggr).
    \end{aligned}
\end{equation*}
For estimates of $\del_j Y_m - \delta_{jm}$ or $g^{kl} - \delta_{kl}$ by $C \eps$, we exploit \autoref{lem:aux ests trafo global}(c), while $Q$-terms yield an $\eps$-power, and $\del_j \Delta Y_k$ can be estimated by $C \eps$ thanks to \autoref{lem:aux ests trafo global}.
The other terms can be bounded by \autoref{lem:aux ests trafo global}, so $\| \mre^{\eta(\cdot)} (\cB_{\tau_{\rh,1}}(Q) - \mdiv S_\xi(0)(\Delta - \Id_3)Q) \|_{\F_\infty^v} \le C \eps^2$.
The Lipschitz estimates can be obtained in the same way upon modifying the above lemmas accordingly. 
\end{proof}

\begin{prop}\label{prop:ests G_2}
Assume $p > 4$.
There exists $C > 0$ so that for $\eps \in (0,\eps_0)$, $(z,\pi)$, $(z_1,\pi_1)$, $(z_2,\pi_2) \in \tcK_\eps$, we have $ \| \mre^{\eta(\cdot)} G_2(z) \|_{\F_\infty^Q} \le C \eps^2$ and $\| \mre^{\eta(\cdot)}(G_2(z_1) - G_2(z_2)) \|_{\F_\infty^Q} \le C \eps \cdot \| (z_1,\pi_1) - (z_2,\pi_2) \|_{\tcK}$.
\end{prop}

\begin{proof}
Thanks to \autoref{lem:further aux ests global}(b), we only need to handle $\cS(v,Q) + \tS_\xi(0)\nabla v$, $Q^2 - \tr(Q^2)\nicefrac{\Id_3}{3}$ and $\tr(Q^2)Q$.
The quadratic shape of the term $Q^2 - \tr(Q^2)\nicefrac{\Id_3}{3}$ implies that it can be estimated by $C \eps^2$ upon invoking the first inequality in \eqref{eq:est of power} which remains valid for $T = \infty$.
The term $\tr(Q^2)Q$ can be treated likewise. 

For the estimates of the remaining terms, we first handle the term $\cD_{ij}$ from \eqref{eq:transformed grad}.
Indeed, making use of H\"older's inequality and \autoref{lem:aux ests trafo global}(d), we infer that $\| \mre^{\eta(\cdot)} \cD_{ij} \|_{\F_\infty^Q} \le C \eps$.
In view of the shape of~$\cS(v,Q)$ from \eqref{eq:cS} and estimates of the $Q$-factors by $\eps$ in $\rL^\infty(0,\infty;\rW^{1,\infty}(\cF_0))$ by virtue of \autoref{lem:ests of max reg space}(b), all terms in $\cS(v,Q)$ except for $\nicefrac{\xi}{3}(\cD_{ij} + \cD_{ji})$ can thus be estimated by $C \eps^2$.
In order to guarantee suitable decay of this term, we invoke $\tS_\xi(0)\nabla v = - \nicefrac{\xi}{3}(\nabla v + (\nabla v)^\top)$.
We recall the representation of~$\cD_v = \cD - \nabla v$ from \eqref{eq:repr diff transformed grad and grad}, exploit \autoref{lem:aux ests trafo global}(d) and observe that $\| \mre^{\eta(\cdot)} \del_l v_k \|_{\rL^p(0,\infty;\rH^1(\cF_0))} \le C \eps$ to conclude $\| \mre^{\eta(\cdot)} \cD_v \|_{\F_\infty^Q} \le C \eps^2$, and likewise for $\cD_v^\top$.
In total, we have $\| \mre^{\eta(\cdot)} (\cS(v,Q) + \tS_\xi(0) \nabla v) \|_{\F_\infty^Q} \le C \eps^2$.
Note that the proof of the Lipschitz estimate is conceptually similar.
\end{proof}

\begin{prop}\label{prop:ests G_3 and G_4}
Let $p > 4$.
Then there is $C > 0$ so that for all $\eps \in (0,\eps_0)$, $(z,\pi)$, $(z_1,\pi_1)$, $(z_2,\pi_2) \in \tcK_\eps$, it is valid that
\begin{equation*}
    \begin{aligned}
        \| \mre^{\eta(\cdot)} G_3(z,\pi) \|_{\F_\infty^\ell}
        &\le C \eps^2 \tand \| \mre^{\eta(\cdot)} (G_3(z_1,\pi_1) - G_3(z_2,\pi_2)) \|_{\F_\infty^\ell} 
        \le C \eps \cdot \| (z_1,\pi_1) - (z_2,\pi_2) \|_{\tcK} \taswellas\\
        \| \mre^{\eta(\cdot)} G_4(z,\pi) \|_{\F_\infty^\omega}
        &\le C \eps^2 \tand \| \mre^{\eta(\cdot)} (G_4(z_1,\pi_1) - G_4(z_2,\pi_2)) \|_{\F_\infty^\omega} 
        \le C \eps \cdot \| (z_1,\pi_1) - (z_2,\pi_2) \|_{\tcK}.
    \end{aligned}
\end{equation*}
\end{prop}

\begin{proof}
It readily follows that the terms $-\mS \omega \times \ell$ as well as $\omega \times (J_0 \omega)$ can be estimated by $C \eps^2$.
With regard to \eqref{eq:rhs G_1 to G_4} and \autoref{lem:further aux ests global}, it remains to discuss the term $-\int_{\del \cS_0} \bO(t)^\top(\tau_{\rr}(Q) + \sigma(Q))\bO(t) n \rd \Gamma$ in~$G_3$ and the analogous term in $G_4$.
As in the proof of \autoref{prop:ests of F_3 and F_4}, we focus on selected terms.
The other terms can be treated in the same way.
We have already discussed the linear part of $\tau$ in \autoref{lem:further aux ests global}, and we observe that every remaining term has at least two factors of terms with $Q$, allowing for an estimate by~$C \eps^2$.
Similarly as in the proof of \autoref{prop:ests of F_3 and F_4}, with the respective embeddings applied with $T = \infty$, and employing \autoref{lem:aux ests trafo global}(e) for the boundedness of $\bO$ and $\bO^\top$, for small $\delta > 0$, we get
$\left \| \mre^{\eta(\cdot)} \int_{\del\cS_0} (\bO^\top \nabla Q \odot \nabla Q \bO) n \rd \Gamma \right\|_p 
    \le C \cdot \| Q \|_{\rL^\infty(0,\infty;\rH^{\nicefrac{3}{2}+ \delta}(\cF_0))} \cdot \| \mre^{\eta(\cdot)} Q \|_{\rL^p(0,\infty;\rH^{\nicefrac{3}{2}+\delta}(\cF_0))} \le C \eps^2$, $\left \| \mre^{\eta(\cdot)} \int_{\del\cS_0} (\bO^\top Q \tr(Q \Delta Q) \bO) n \rd \Gamma \right\|_p \le C \eps^3$ and $\left \| \mre^{\eta(\cdot)} \int_{\del\cS_0} (\bO^\top Q \tr(Q^2) \tr(Q^2) \bO) n \rd \Gamma \right\|_p \le C \eps^5$.
Hence, recalling $\cJ$ from \autoref{lem:prep ests nonlin terms}(b), we deduce that $\| \mre^{\eta(\cdot)} \cJ(-\bO^\top (\tau_{\rr}(Q) + \sigma(Q)) \bO) \|_p \le C \eps^2$, so the first part of the assertion is implied.
The Lipschitz estimate can be obtained likewise.
\end{proof}

These preparations pave the way for the proof of the second main result.

\begin{proof}[Proof of \autoref{thm:global strong wp for small data}]
First, we verify \autoref{thm:global strong wp for small data in ref config}.
We have already argued that $\Psi \colon \tcK_\eps \to \tcK$ is well-defined.
For $\eps \in (0,\eps_0)$, with $\eps_0 > 0$ as above, let $(z,\pi)$, $(z_1,\pi_1)$, $(z_2,\pi_2) \in \tcK_\eps$.
By $\| (v_0,Q_0,\ell_0,\omega_0) \|_{\rX_\gamma} \le \delta$, and exploiting \autoref{thm:lin theory Q-tensor interaction}(b) as well as \autoref{prop:ests G_1}, \autoref{prop:ests G_2} and \autoref{prop:ests G_3 and G_4}, we find
\begin{equation*}
    \| \Psi(z,\pi) \|_{\tcK} \le \Cmr(\delta + C \eps^2) \tand \| \Psi(z_1,\pi_1) - \Psi(z_2,\pi_2) \|_{\tcK} \le \Cmr C \eps.
\end{equation*}
With $\delta \le \nicefrac{1}{2 \Cmr}$ and $\eps \le \min\{\eps_0,\nicefrac{1}{2 \Cmr C},\nicefrac{1}{2C},\sqrt{\nicefrac{\delta}{C}}\}$, we conclude the self-map and contraction property of $\Psi$.
Thus, there exists a unique fixed point $(z_*,\pi_*) = (v_*,Q_*,\ell_*,\omega_*,\pi_*)$, being the unique global strong solution to \eqref{eq:cv1}--\eqref{eq:cv3}.
\autoref{thm:global strong wp for small data in ref config} then is a consequence of the definition of $\tcK$ in \eqref{eq:def of tcK}.

The last task is to deduce \autoref{thm:global strong wp for small data} from \autoref{thm:global strong wp for small data in ref config}.
The latter implies $(\ell,\omega) \in \rW^{1,p}(0,\infty)^6$, so $h'$, $\Omega$ and $X$ can be obtained as indicated in \autoref{sec:change of var}.
For $v$, $Q$ and~$\pi$, we invoke the backward change of coordinates from \autoref{sec:change of var}.
By construction, this is the solution to \eqref{eq:LC-fluid equationsiso}--\eqref{eq:initial} in the asserted regularity class.
Finally, we prove that $\dist(\cS(t),\del \cO) > \nicefrac{r}{2}$ for all $t \in [0,\infty)$. 
The shape of $\cS(t)$ at time~$t$ as introduced in \autoref{sec:model} and the above estimates imply that for all $t \in [0,\infty)$, we have
\begin{equation*}
    \dist(\cS(t),\cS_0) \le |h(t)| + |\bO(t) - \Id_3| < \nicefrac{r}{2}, \tso \dist(\cS(t),\del \cO) > \nicefrac{r}{2} \enspace \text{by} \enspace \dist(\cS_0,\cO) > r. \qedhere
\end{equation*}
\end{proof}

\medskip 

{\bf Acknowledgements}
This research is supported by the Basque Government through the BERC 2022-2025 program and by the Spanish State Research Agency through BCAM Severo Ochoa excellence accreditation CEX2021-01142-S funded by MICIU/AEI/10.13039/501100011033 and through Grant PID2023-146764NB-I00 funded by MICIU/AEI/10.13039/501100011033 and cofunded by the European Union. 
 Felix Brandt would like to express his gratitude to the German National Academy of Sciences Leopoldina for support through the Leopoldina Fellowship Program with grant number~LPDS 2024-07.
 Felix Brandt and Matthias Hieber acknowledge the support by DFG project FOR~5528. 
 Arnab Roy would like to thank the Alexander von Humboldt-Foundation, Grant RYC2022-036183-I funded by MICIU/AEI/10.13039/501100011033 and by ESF+.
 The authors would also like to thank the anonymous referees for the careful reading and valuable comments that helped improve the quality of the manuscript.

\bibliography{lcfsi_short}
\bibliographystyle{siam}

\end{document}